%% file: CombinedVersion.tex
\documentclass[11pt]{amsart}

\usepackage{amsfonts, amstext, amsmath, amsthm, amscd, amssymb} 
\usepackage{mathtools}

\usepackage{float} 
\usepackage{graphics}
\usepackage{caption}
\usepackage{subcaption}
\usepackage{import} 

\usepackage{physics}  

\usepackage{microtype}

\usepackage[hidelinks,pagebackref,pdftex]{hyperref}

\usepackage[dvipsnames]{xcolor}
\definecolor{MyCyan}{HTML}{00F9DE}
\usepackage{enumitem} 
\setlist{nolistsep,leftmargin=*}
\usepackage{transparent}
\usepackage{stackengine} 

\usepackage{marginnote}
\long\def\@savemarbox#1#2{\global\setbox#1\vtop{\hsize\marginparwidth 
  \@parboxrestore\tiny\raggedright #2}}
\newcommand{\RR}{\mathbb{R}}  
\newcommand{\HH}{\mathbb{H}}  
\newcommand{\ZZ}{\mathbb{Z}}  

\newcommand{\calP}{\mathcal{P}}
\newcommand{\calL}{\mathcal{L}}
\renewcommand{\SS}{\mathbb{S}}

\renewcommand{\setminus}{{\smallsetminus}}

\newcommand{\from}{\colon\thinspace} 

\newtheorem{theorem}{Theorem}[section]
\newtheorem*{theorem*}{Theorem}
\newtheorem{proposition}[theorem]{Proposition}
\newtheorem{lemma}[theorem]{Lemma}
\newtheorem{corollary}[theorem]{Corollary}

\newtheorem*{namedtheorem}{\theoremname}
\newcommand{\theoremname}{testing}
\newenvironment{named}[1]{\renewcommand{\theoremname}{#1}\begin{namedtheorem}}{\end{namedtheorem}}

\theoremstyle{definition}
\newtheorem{definition}[theorem]{Definition}

\newtheorem{remark}[theorem]{Remark}

\newcommand{\refthm}[1]{Theorem~\ref{Thm:#1}}
\newcommand{\reflem}[1]{Lemma~\ref{Lem:#1}}
\newcommand{\refprop}[1]{Proposition~\ref{Prop:#1}}
\newcommand{\refcor}[1]{Corollary~\ref{Cor:#1}}

\newcommand{\refdef}[1]{Definition~\ref{Def:#1}}
\newcommand{\refsec}[1]{Section~\ref{Sec:#1}}
\newcommand{\reffig}[1]{Figure~\ref{Fig:#1}}

\title[Modular links: Bunch algorithm and upper volume bounds]{Modular links:\\Bunch algorithm and upper volume bounds} 

\author{Connie On Yu Hui} 
\address[]{School of Mathematics, Monash University, VIC 3800, Australia } 
\email[]{onyu.hui@monash.edu | connieonyuhui@gmail.com} 

\author{Jos\'{e} Andr\'{e}s Rodr\'{\i}guez Migueles}
\address[]{Ludwig Maximilian University of Munich, Germany}
\email[] {migueles@math.lmu.de} 


\begin{document} 
\begin{abstract} 
In the 1970s, Williams developed an algorithm that has been used to construct and study modular links in the Lorenz template. We introduce an improved algorithm, which we call the bunch algorithm, to provide more insights into the geometry of modular links and Lorenz links. Using the machinery developed for the bunch algorithm, we provide the first upper volume bound that is independent of word exponents and quadratic in the braid index of the Lorenz link component for all modular link complements. We find families of modular knot complements with upper volume bounds that are linear in the braid index. A classification of modular link complements based on the relative magnitudes of word exponents is also presented.  
\end{abstract} 

\maketitle 

\section{Introduction} 
In the 1970s, Williams~\cite{Williams:StructureOfLorenzAttractors, Birman-Williams:KnottedPeriodicOrbitsInDynamicalSystemsI} introduced an algorithm to study the knot properties of closed orbits appearing in the Lorenz equations. The algorithm requires the lexicographic order of all words in the same cyclic permutation class of a primitive word, and it has been used to study the corresponding closed orbits, which are called \emph{Lorenz links}. See \cite{Bergeron-Pinsky-Silberman:UpperBound, RodriguezMigueles:PeriodsOfContinuedFractions} for example. In fact, in~\cite{Ghys:KnotsAndDynamics} Ghys observed that the obtained links are isotopic to the periodic orbits of the geodesic flow of the modular surface after the boundary torus is trivially filled. These periodic orbits are called \emph{modular links}. In this paper, we use the notion of bunches to introduce a more insightful and efficient algorithm for constructing modular links, which helps us show the results on volume bounds.  The new algorithm involves a significantly smaller number of items to be ordered when the exponents in the words are large. More importantly, the proposed algorithm allows us to understand each modular link via a relatively small part of it which we call the \emph{full bunches} (see \reffig{SplitTemplateWithLabelsFullBs} and \refdef{FullBs&OrderWithin}). While the current methods of constructing a Lorenz knot such as Williams' algorithm and the method in \cite{Dehornoy:ZeroesOfAlexanderPolyOfLorenzKnot} that uses Young diagrams are easy to implement, the algorithm developed in this paper gives us better theoretical insights into links embedded in a template (see \refsec{WhyBunches} for further explanation) and allows us to find a more intrinsic upper volume bound for all modular link complements and Lorenz link complements. The research on the volume bounds for $3$-manifolds has been active in the past few decades, a general review of the extensive work on volume bounds can be found in~\cite[Section~4]{Futer-Kalfagianni-Purcell:SurveyHypKnotTheory}. 

Any modular link complement $\mathrm{T^1}S_{\textup{mod}}\setminus \widehat\gamma$ admits a finite-volume complete hyperbolic metric~\cite{Foulon-Hasselblatt:ContactAnosov}.  There has been interest in relating the volume of modular link complements and the properties of the associated closed geodesics $\gamma$ in the modular surface (see \cite{Bergeron-Pinsky-Silberman:UpperBound, RodriguezMigueles:LowerBound, RodriguezMigueles:PeriodsOfContinuedFractions} for example).  Bergeron, Pinsky, and Silberman \cite[Section 3]{Bergeron-Pinsky-Silberman:UpperBound} provide an upper volume bound for all modular link complements. Their upper bound depends on the combinatorial terms in the symbolic descriptions of the modular links, namely, the word periods and word exponents (see Sections~\ref{Sec:Background} and \ref{Sec:Preliminaries} for more details).  Such a bound grows faster than quadratically with the word period when the exponents are large (see \refsec{Discussion} for more explanation). In this paper, we use the bunch notion developed in \refsec{NewAlg} to give an upper volume bound that is independent of the word exponents and is quadratic in the sum of word period(s). 

\begin{named}{\refthm{QuadraticUpperBound}}
If $L$ is a modular link with the sum of word period(s) equal $n$, then the hyperbolic volume of the corresponding modular link complement satisfies the following inequality: 
\[\mathrm{Vol}(\mathrm{T^1}S_{\textup{mod}} \setminus L) \leq 12 v_{\textup{tet}} (n^2+3n+2), \]
where $v_{\textup{tet}}\approx 1.01494$ is the volume of the regular ideal tetrahedron. 
\end{named}

\begin{remark} \label{Rmk:TripNum_BraidIndex_PeriodContFraction}
The sum of word period(s) of a modular link $L$ is the same as the trip number of $L$, which equals the braid index of $L$ in $\SS^3$ (see \cite{Franks-Williams:BraidsAndJonesPolynomial, Birman-Williams:KnottedPeriodicOrbitsInDynamicalSystemsI, Birman-Kofman:NewTwistOnLorenzLinks}). The word period of a modular knot is the same as half the period of the continued fraction expansion of the associated geodesic on the modular surface (see \cite{Series:ModularSurface, RodriguezMigueles:PeriodsOfContinuedFractions}). 
\end{remark}

To prove \refthm{QuadraticUpperBound}, we use a technique developed in \cite{RodriguezMigueles:PeriodsOfContinuedFractions} that involves annular Dehn filling and self-intersection numbers of multi-curves in the Seifert-fibred space projection \cite{Cremaschi-RodriguezMigueles:HypOfLinkCpmInSFSpaces}.  The bunch notion allows us to more systematically find the annuli for annular Dehn filling and provide clean arguments for other results in this paper. The bunch perspective also helps us construct the parent link  (see \refdef{ParentLink}) for each modular link in its corresponding Seifert-fibred space and find their upper volume bounds. 

Since Dehn filling decreases volume \cite{Thurston:Geom&TopOf3Mfd}, the upper bound in \refthm{QuadraticUpperBound} also acts as an upper volume bound for all Lorenz link complements that admit complete hyperbolic structures (see \refcor{QuadraticUpperBound}). 

Champanerkar \emph{et~al}~\cite[Theorem 1.7]{C-F-K-N-P:VolBsForGenTwistedTorusLinks} provide upper volume bounds that are quadratic or cubic with respect to the parameter $r_1$ in their \emph{T}-link notation for volumes of generalised twisted torus links. Since the set of generalised twisted torus links contains all \emph{T}-links and \cite{Birman-Kofman:NewTwistOnLorenzLinks} shows that \emph{T}-links coincide with Lorenz links, the collection of Lorenz links is a subset of the collection of all generalised twisted torus links. Whether their bound is quadratic or cubic depends on whether the twists are all full twists. Using the bunch notion introduced in this paper, we observe that their parameter $r_1$ depends on the magnitude of the word exponents: further explanation is provided in \refsec{Discussion}.  As the volume bound in \refthm{QuadraticUpperBound} of this paper does not depend on word exponents, it serves as a better bound for the volumes of hyperbolic Lorenz link complements when the word exponents are large. 

Having understood the modular links from the perspective of bunches, we observe a classification of modular link complements.

\begin{named}{\refthm{Classification} (Classification)}
All modular link complements can be partitioned into classes of members that have the same base orders of their corresponding modular links. All members in each class share the same parent manifold (up to homeomorphism) and the same upper volume bound. 
\end{named} 

Based on the classification result and the construction of parent links, we find classes of modular knot complements with upper bounds depending linearly on the word period (see \refthm{LinearBound}), which generalises the result in \cite[Theorem 1.4]{RodriguezMigueles:PeriodsOfContinuedFractions}.  

\subsection{Background: Modular links and labelled code words} \label{Sec:Background}
Let $S_{\textup{mod}}$ be the \emph{modular surface}, which is the quotient of $\HH^2$ by the modular group $\textup{PSL}(2,\ZZ)$.   It is known that the unit tangent bundle 
$\textup{T}^1S_{\textup{mod}}$ is homeomorphic to the complement of the trefoil knot in $\mathbb{S}^3$. 
Any oriented closed geodesic $\gamma$ on $S_{\textup{mod}}$ has a canonical lift $\widehat\gamma$ in $\textup{T}^1S_{\textup{mod}}$, which is a periodic orbit of the geodesic flow on the unit tangent bundle. The isotopy class of such a canonical lift $\widehat\gamma$ in $\mathrm{T}^1S_{\textup{mod}}$ is denoted as a \emph{modular knot}. A \emph{modular link} $L$ is a finite collection of disjoint modular knots in $\textup{T}^1S_{\textup{mod}}$. 

The \emph{modular link complement} $\textup{T}^1S_{\textup{mod}} \setminus L$ is the complement of the modular link $L$ in $\textup{T}^1S_{\textup{mod}}$ and it can be viewed as the complement of $L$ and the trefoil knot in $\mathbb{S}^3$.    Ghys \cite{Ghys:KnotsAndDynamics} observed that modular knots in $\mathbb{S}^3$ are \emph{Lorenz knots}, which are isotopy classes of periodic orbits of the Lorenz flow. 

Guckenheimer and Williams \cite{Guckenheimer-Williams:StructuralStabilityOfLorenzAttractors,Williams:StructureOfLorenzAttractors} introduced a geometric model for interpreting the dynamics of the flow, which involves the \emph{Lorenz template}: a  branched surface embedded in $\RR^3$ (or $\SS^3$) as illustrated in \reffig{Templates} (left). Such a geometric model was later justified by Tucker \cite{Tucker:Smale14thProblem}. Lorenz knots can thus be regarded as knots embedded in the template and studied using symbolic dynamics (see \cite{Williams:StructureOfLorenzAttractors, Birman-Williams:KnottedPeriodicOrbitsInDynamicalSystemsI}). 

Research related to modular links and symbolic dynamics involves symbolic descriptions of periodic orbits. The symbols $x$ and $y$ represent oriented arcs of modular links that pass through the left and the right ears of the Lorenz template respectively (see \reffig{Templates}, left). The code words arising from symbolic dynamics are the same as the code words obtained from parabolic matrices $x$ and $y$ that generate $\textup{PSL}(2,\ZZ)\cong \ZZ_2 \ast \ZZ_3$ (see \cite{Bergeron-Pinsky-Silberman:UpperBound}, \cite{Brandts-Pinsky-Silberman:VolModularLinks} for more details).  These code words also coincide with the cutting sequences arising from lifting closed oriented geodesics in the modular surface to the hyperbolic upper half plane with Farey tessellation (see \cite{Series:ModularSurface} for details). By \cite[Proposition 3.1]{Birman-Williams:KnottedPeriodicOrbitsInDynamicalSystemsI}, the set of all Lorenz links is in one-to-one correspondence with the collection of all finite sets of distinct cyclic permutation classes of positive aperiodic words in $x$ and $y$.  

Classes of primitive code words up to cyclic permutations have been used to describe and study modular links (See, for example, \cite{Birman-Williams:KnottedPeriodicOrbitsInDynamicalSystemsI, Bergeron-Pinsky-Silberman:UpperBound, RodriguezMigueles:PeriodsOfContinuedFractions}).  Instead of considering just the code words, we initiate the use of \emph{labelled code words}, which are code words that have base labels distinguishing the $x$-bases and $y$-bases, to introduce the idea of bunches in modular links and 
study the hyperbolic volumes of the modular link complements.  

\subsection{Organisation} 
\refsec{Preliminaries} contains preliminaries that clarify terms and notations used in all other sections. In \refsec{NewAlg}, we introduce the notion of bunches in modular links and present an algorithm for constructing modular links in the Lorenz template. This gives us tools to find an upper volume bound for all modular link complements in \refsec{UpperVolBound}, a way of classifying modular link complements in \refsec{ClassifyingModLinks}, and families of modular knot complements that have upper volume bounds depending linearly on word period in \refsec{LinearInPeriod}.  \refsec{Discussion} provides a further discussion on the comparison of bounds, 
 the proof techniques used, and related questions.   

\subsection{Acknowledgements} 
We thank Pierre Dehornoy, Jessica Purcell, Tali Pinsky, Layne Hall, Norman Do, and Panagiotis Papadopoulos for helpful discussions.  We also thank the organisers of the conference Dynamics, Foliation, and Geometry III, held at MATRIX Australia, where the collaboration for this work started. The second author recognizes the support from the Special Priority Programme SPP 2026 Geometry at Infinity funded by the DFG. We thank the anonymous reviewer for the helpful suggestions and comments. 

\section{Preliminaries} \label{Sec:Preliminaries} 
Recall that the \emph{Lorenz template} is a branched surface embedded in $\RR^3$ (or $\SS^3$) as illustrated in \reffig{Templates} (left) where the letter symbol $x$ represents a loop in the left ear and the letter symbol $y$ represents a loop in the right ear. 

Throughout this paper, the term \emph{code word} (or simply \emph{word}) is the same as the \emph{Lorenz word} defined in \cite{Birman-Williams:KnottedPeriodicOrbitsInDynamicalSystemsI} unless otherwise specified. That is, a \emph{code word} is a cyclic permutation class of a finite sequence of letters $x$ and $y$, each letter $x$ (or resp. $y$) in the sequence represents a loop in the left (or resp. right) ear of the Lorenz template that forms part of the Lorenz knot. By abuse of terminology, we sometimes call the finite sequence of letters a \emph{word} when the context is clear.  From now on, we consider a modular knot $K$ as a simple closed curve embedded in the Lorenz template with code word containing both $x$ and $y$ and the corresponding \emph{modular knot complement} $\mathrm{T^1}S_{\textup{mod}} \setminus K$ as the complement of $K$ in the unit tangent bundle $\mathrm{T^1}S_{\textup{mod}}$. 

Let $n$ be a positive integer. Given any code word $w$ representing a modular knot $K$, we can cyclically permute $w$ to give a code word $w'$ such that $w'$ has the form $x^{k_1} y^{l_1} \ldots x^{k_n} y^{l_n}$ where the first and the last letters are different and any neighbouring bases are different. We call $w$ a \emph{period-$n$ code word} for $K$. The positive integers $k_1, \ldots k_n$ are called the \emph{$x$-exponent(s)} of the word and the positive integers $l_1, \ldots l_n$ are called the \emph{$y$-exponent(s)} of the word. The \emph{maximal $x$-exponent} of $w$ is $\mathrm{max}\{k_1, \ldots, k_n\}$ and the \emph{maximal $y$-exponent} of $w$ is $\mathrm{max}\{l_1, \ldots, l_n\}$.  

For convenience purposes, we use the split template (\reffig{Templates}, right) instead of the usual Lorenz template to illustrate most of our examples in this paper. Note that by gluing the $x$-split line and $y$-split line to the branch line accordingly, we obtain the template in \reffig{Templates}, left. Unless otherwise specified, we will assume the gluings exist in the split template such that the split template is topologically the same as the Lorenz template. 

\begin{figure}
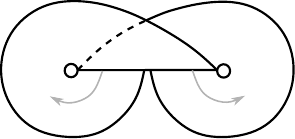
\hspace{5mm} 
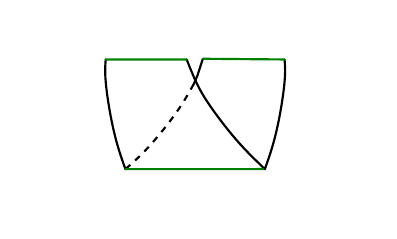

\caption{Left: The Lorenz template. Right: The split template, i.e., the Lorenz template cut along the branch line. If the $y$-split line is glued to the right half of the branch line and the $x$-split line is glued to the left half of the branch line (without twisting of strips), we obtain the Lorenz template.
\label{Fig:Templates}}
\end{figure}  

Observe that there are two split line segments on top of the split template. 
We call the left one the \emph{$x$-split line} and the right one the \emph{$y$-split line} of the split template (\reffig{Templates}, right).
For the sake of simplicity, we call the split branch line at the bottom the \emph{branch line}. The \emph{left ear} of the Lorenz template is the part of the split template from the $x$-split line to the branch line while the \emph{right ear} is the part from the $y$-split line to the branch line.   

\subsection{Williams' algorithm} 
We will go through Williams' algorithm using an example. 

Let $w=x^{10}y^2x^5y^2x^7y^6x^2y^2x^5y^3$ be a code word for a modular knot $K$. 

To construct the modular knot $K$ in the split Lorenz template, we first list all the members in a cyclic permutation class of the code word $w$ in the following way: Starting from the second word in the list, the word is obtained by putting the first letter in the last word to the end of the word. 

Then, we find the lexicographic order (magenta numbers on the right) of all these words like the following:

\begin{align*}
    x^{10}y^2x^5y^2x^7y^6x^2y^2x^5y^3 \quad \quad &\textcolor{magenta}{1} \\ 
    x^{9}y^2x^5y^2x^7y^6x^2y^2x^5y^3x^{1} \quad \quad &\textcolor{magenta}{2} \\ 
    x^{8}y^2x^5y^2x^7y^6x^2y^2x^5y^3x^{2} \quad \quad &\textcolor{magenta}{3} \\ 
    x^{7}y^2x^5y^2x^7y^6x^2y^2x^5y^3x^{3} \quad \quad &\textcolor{magenta}{4} \\
    x^{6}y^2x^5y^2x^7y^6x^2y^2x^5y^3x^{4} \quad \quad &\textcolor{magenta}{6} \\ 
    x^{5}y^2x^5y^2x^7y^6x^2y^2x^5y^3x^{5} \quad \quad &\textcolor{magenta}{9} \\
    x^{4}y^2x^5y^2x^7y^6x^2y^2x^5y^3x^{6} \quad \quad &\textcolor{magenta}{13} \\ 
    x^{3}y^2x^5y^2x^7y^6x^2y^2x^5y^3x^{7} \quad \quad &\textcolor{magenta}{17} \\ 
    x^{2}y^2x^5y^2x^7y^6x^2y^2x^5y^3x^{8} \quad \quad &\textcolor{magenta}{21} \\ 
    x^{1}y^2x^5y^2x^7y^6x^2y^2x^5y^3x^{9} \quad \quad &\textcolor{magenta}{26} \\ 
    y^2x^5y^2x^7y^6x^2y^2x^5y^3x^{10} \quad \quad &\textcolor{magenta}{37} \\ 
    y^1x^5y^2x^7y^6x^2y^2x^5y^3x^{10}y^1 \quad \quad &\textcolor{magenta}{32} \\ 
    x^5y^2x^7y^6x^2y^2x^5y^3x^{10}y^2 \quad \quad &\textcolor{magenta}{8} \\ 
    x^{4}y^2x^7y^6x^2y^2x^5y^3x^{10}y^2x^{1} \quad \quad &\textcolor{magenta}{12} \\ 
    x^{3}y^2x^7y^6x^2y^2x^5y^3x^{10}y^2x^{2} \quad \quad &\textcolor{magenta}{16} \\ 
    x^{2}y^2x^7y^6x^2y^2x^5y^3x^{10}y^2x^{3} \quad \quad &\textcolor{magenta}{20} \\ 
    x^{1}y^2x^7y^6x^2y^2x^5y^3x^{10}y^2x^{4} \quad \quad &\textcolor{magenta}{25} \\ 
    y^2x^7y^6x^2y^2x^5y^3x^{10}y^2x^{5} \quad \quad &\textcolor{magenta}{36} \\ 
    y^{1}x^7y^6x^2y^2x^5y^3x^{10}y^2x^{5}y^{1} \quad \quad &\textcolor{magenta}{31} \\ 
    x^7y^6x^2y^2x^5y^3x^{10}y^2x^{5}y^{2} \quad \quad &\textcolor{magenta}{5} \\ 
    x^{6}y^6x^2y^2x^5y^3x^{10}y^2x^{5}y^{2}x^{1} \quad \quad &\textcolor{magenta}{7} \\ 
    x^{5}y^6x^2y^2x^5y^3x^{10}y^2x^{5}y^{2}x^{2} \quad \quad &\textcolor{magenta}{11} \\ 
    x^{4}y^6x^2y^2x^5y^3x^{10}y^2x^{5}y^{2}x^{3} \quad \quad &\textcolor{magenta}{15} \\ 
    x^{3}y^6x^2y^2x^5y^3x^{10}y^2x^{5}y^{2}x^{4} \quad \quad &\textcolor{magenta}{19} \\ 
    x^{2}y^6x^2y^2x^5y^3x^{10}y^2x^{5}y^{2}x^{5} \quad \quad &\textcolor{magenta}{24} \\ 
    x^{1}y^6x^2y^2x^5y^3x^{10}y^2x^{5}y^{2}x^{6} \quad \quad &\textcolor{magenta}{29} \\ 
    y^6x^2y^2x^5y^3x^{10}y^2x^{5}y^{2}x^{7} \quad \quad &\textcolor{magenta}{44} \\ 
    y^{5}x^2y^2x^5y^3x^{10}y^2x^{5}y^{2}x^{7}y^{1} \quad \quad &\textcolor{magenta}{43} \\ 
    y^{4}x^2y^2x^5y^3x^{10}y^2x^{5}y^{2}x^{7}y^{2} \quad \quad &\textcolor{magenta}{42} \\ 
    y^{3}x^2y^2x^5y^3x^{10}y^2x^{5}y^{2}x^{7}y^{3} \quad \quad &\textcolor{magenta}{41} 
\end{align*}

\begin{align*}
    y^{2}x^2y^2x^5y^3x^{10}y^2x^{5}y^{2}x^{7}y^{4} \quad \quad &\textcolor{magenta}{39} \\ 
    y^{1}x^2y^2x^5y^3x^{10}y^2x^{5}y^{2}x^{7}y^{5} \quad \quad &\textcolor{magenta}{34} \\ 
    x^2y^2x^5y^3x^{10}y^2x^{5}y^{2}x^{7}y^{6} \quad \quad &\textcolor{magenta}{22} \\ 
    x^1y^2x^5y^3x^{10}y^2x^{5}y^{2}x^{7}y^{6}x^1 \quad \quad &\textcolor{magenta}{27} \\
    y^2x^5y^3x^{10}y^2x^{5}y^{2}x^{7}y^{6}x^2 \quad \quad &\textcolor{magenta}{38} \\ 
    y^1x^5y^3x^{10}y^2x^{5}y^{2}x^{7}y^{6}x^2y^1 \quad \quad &\textcolor{magenta}{33} \\
    x^5y^3x^{10}y^2x^{5}y^{2}x^{7}y^{6}x^2y^2 \quad \quad &\textcolor{magenta}{10} \\ 
    x^{4}y^3x^{10}y^2x^{5}y^{2}x^{7}y^{6}x^2y^2x^{1} \quad \quad &\textcolor{magenta}{14} \\ 
    x^{3}y^3x^{10}y^2x^{5}y^{2}x^{7}y^{6}x^2y^2x^{2} \quad \quad &\textcolor{magenta}{18} \\ 
    x^{2}y^3x^{10}y^2x^{5}y^{2}x^{7}y^{6}x^2y^2x^{3} \quad \quad &\textcolor{magenta}{23} \\ 
    x^{1}y^3x^{10}y^2x^{5}y^{2}x^{7}y^{6}x^2y^2x^{4} \quad \quad &\textcolor{magenta}{28} \\ 
    y^3x^{10}y^2x^{5}y^{2}x^{7}y^{6}x^2y^2x^{5} \quad \quad &\textcolor{magenta}{40} \\ 
    y^{2}x^{10}y^2x^{5}y^{2}x^{7}y^{6}x^2y^2x^{5}y^{1} \quad \quad &\textcolor{magenta}{35} \\ 
    y^{1}x^{10}y^2x^{5}y^{2}x^{7}y^{6}x^2y^2x^{5}y^{2} \quad \quad &\textcolor{magenta}{30} 
\end{align*}

There are $44$ letters in the code word $w$, and thus there are $44$ members in the cyclic permutation class of the code word. Mark $44$ points on the branch line of the split template as shown in \reffig{WilliamsAlg1}.  

Consider the first few numbers in the sequence of magenta numbers, \textcolor{magenta}{$1, 2, 3, 4, 6, 9, 13, \ldots$}. As shown in \reffig{WilliamsAlg1}, we draw an overstrand from $1$ (top) to $2$ (bottom), and then from $2$ (top) to $3$ (bottom), $3$ (top) to $4$ (bottom), $4$ (top) to $6$ (bottom), $6$ (top) to $9$ (bottom), $9$ (top) to $13$ (bottom), etc.   

After drawing all the strands using the sequence of numbers (coloured in magenta), the desired modular knot represented in the split template is obtained (see \reffig{WilliamsAlg2}).  

\begin{figure}
\captionsetup[subfigure]{position=b}
\centering
\subcaptionbox{Split template with 44 points marked on the branch line. The seven black line segments correspond to the first seven overstrands that are constructed. 
\label{Fig:WilliamsAlg1} }
{\hspace{-20pt}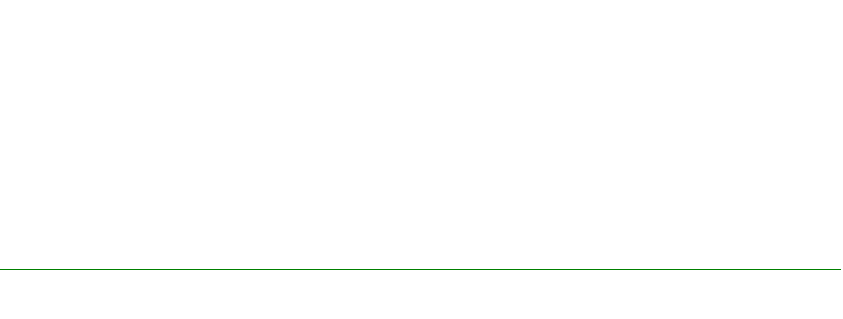}  

\vspace{20pt}
\subcaptionbox{The modular knot with the code word $w=x^{10}y^2x^5y^2x^7y^6x^2y^2x^5y^3$ constructed using Williams' algorithm. 
\label{Fig:WilliamsAlg2}}
{\hspace{-20pt}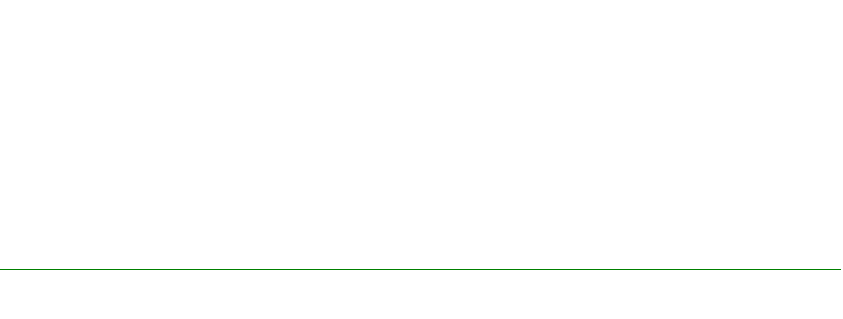}  
\caption{Illustration of Williams' algorithm \label{Fig:
WilliamsAlg}} 
\end{figure} 

\subsection{Motivation for defining bunches}

In \refsec{NewAlg}, we first define full bunches (\refdef{FullBs&OrderWithin}) and then define bunches (\refdef{Bunches&OrderWithin}). We will explain the motivation that led us to such an order of definitions. 

\begin{figure}
\stackunder[5pt]{\hspace{-20pt}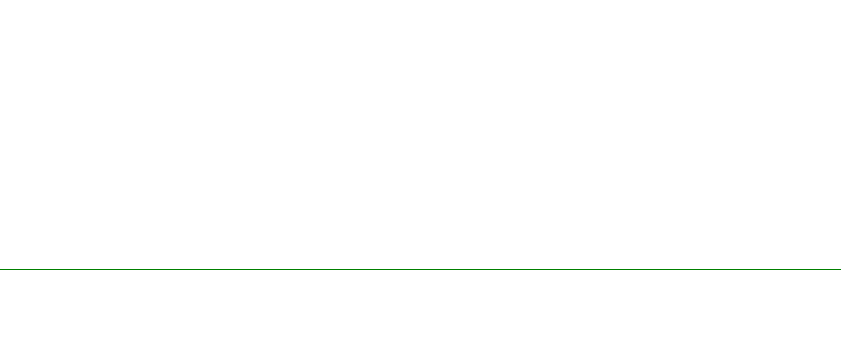}{}
\caption{Colour the letters in the code word $w$ and then colour the corresponding strands in \reffig{WilliamsAlg2} with the corresponding colour. 
\label{Fig:WilliamsAlg2_coloured}}
\end{figure}  

Consider \reffig{WilliamsAlg2}, which shows a modular knot with code word $w$ in the split Lorenz template. If we colour (or label) the $x$-letters with decreasing intensities of the blue colour (or label) and the $y$-letters with decreasing intensities of red colour, and if we colour the strand corresponding to each letter with the same colour, we obtain \reffig{WilliamsAlg2_coloured}. 

In \reffig{WilliamsAlg2_coloured}, we observed that there was a nontrivial permutation of colour intensities in the rightmost five blue strands that pass from the $x$-split line to the right side of the branch line. Similarly, we observed a nontrivial permutation of colour intensities in the leftmost five red strands that pass from the $y$-split line to the left side of the branch line. We wondered how we can know the order of the colour intensities based on the code word. Hence, the idea of full bunches came first in a natural way before the idea of bunches, the precise definition of full bunches and such an order can be found in \refdef{FullBs&OrderWithin}. And then we observed collections of neighbouring strands that give a subsequence of the aforementioned sequence of colour intensities. Hence, we naturally defined full bunches and the orders within the full bunches (\refdef{FullBs&OrderWithin}) first, and then used those definitions to define the notion of bunches (\refdef{Bunches&OrderWithin}).

\section{A new algorithm for constructing modular links} \label{Sec:NewAlg}
In this section, we introduce the notion of bunches and an algorithm for constructing modular links in the Lorenz template.  The key idea of the new algorithm is to find the order within full bunches (see \reffig{SplitTemplateWithLabelsFullBs}, \refprop{OrderWithinFullBs} and the example that follows), which helps us construct parent manifolds in the next section.  Readers may find examples and pictures in this section more illuminating than the definitions or statements at times.  

\begin{definition} \label{Def:BasesOfWord}
    Given a period-$n$ code word $w$ of a modular knot, $w$ can be expressed in the form $x^{k_1} y^{l_1} \ldots x^{k_n} y^{l_n}$. We call each of the $x$'s in the expression with some $k_i$ as exponent an \emph{$x$-base}, and each of the $y$'s in the expression with some $l_j$ as exponent a \emph{$y$-base}. A \emph{base of the word} $w$ is an $x$-base or a $y$-base of $w$.

An \emph{$x$-turn} of the word $w$ is the overcrossing strand corresponding to an $x$-letter in $w$, it is homeomorphic to some half-closed half-open interval $[r,s)$. Similarly, a \emph{$y$-turn} of the word $w$ is the undercrossing strand corresponding to a $y$-letter in $w$. A \emph{turn} is an $x$-turn or a $y$-turn. 
\end{definition}
 For example, the word $x^5y^3x^1y^6$ has two $x$-bases and two $y$-bases. 

From now on, we assume the following unless otherwise specified: 

\begin{enumerate}
    \item Let $w\coloneqq x^{k_1} y^{l_1} \ldots x^{k_n} y^{l_n}$ be a period-$n$ word of a modular knot $K$. 
    \item Let $\mu, \lambda\in[1,n]\cap\ZZ$ such that $k_{\mu}$ and $l_{\lambda}$ are a maximal $x$-exponent and a maximal $y$-exponent of $w$ respectively. We may assume $\mu=1$ because all cyclic permutations of a word represent the same modular knot. 
  \item To distinguish different $x$-bases in the same word $w$ and to distinguish different subwords of $w$ such as $x^{k_1}$ and $x^{k_2}$ with $k_1 = k_2$, we label (or equivalently, colour) the $x$-bases in the word $w$ with integral subscripts (or colours of different intensities). We label (or colour) $y$-bases for a similar reason. That is, we view each code word $w$ as the \emph{labelled code word} $x_1^{k_1} y_1^{l_1} \ldots x_n^{k_n} y_n^{l_n}$ with subscripts added to $x$-base(s) and $y$-base(s). By abuse of terminology, the terms \emph{word}, \emph{code word}, and \emph{labelled code word} will sometimes be used interchangeably to refer to the cyclic permutation class of primitive labelled code words when the context is clear.  For simplicity purposes, we use colours of different intensities when drawing strands in the split template (see \reffig{SplitTemplateWithLabelsWhole} for example) instead of labelling each strand with a labelled base. Note that the labelled bases and coloured bases describe the same concept, we will use them interchangeably depending on which presentation style is better. 
\end{enumerate}

Throughout this section, we use the period-$5$ word $x^{10}y^2x^5y^2x^7y^6x^2y^2x^5y^3$ to exemplify the algorithm in the figures. 
\subsection{The notion of bunches in a modular link} 
In the split template (see \reffig{SplitTemplateWithLabels}), draw a vertical dotted line between the $x$- and $y$-split lines. Mark the number zero at the intersection between the vertical dotted line and the branch line.  

Consider the branch line as part of a real line with the zero and assume each strand associating with a letter in the word $w = x_1^{k_1} y_1^{l_1} \ldots x_n^{k_n} y_n^{l_n}$ starts from the branch line. Note that each subword of the form $x_i^{k_i}$ corresponds to an arc that intersects the branch line $k_i$ times on the left of the zero. Recall that $k_{\mu}$ is a maximal $x$-exponent, we label the intersections between the $x_{\mu}^{k_{\mu}}$-arc and the branch line with $-k_{\mu}, \ldots, -1$. Similarly, we label the intersections between the $y_{\lambda}^{l_{\lambda}}$-arc and the branch line with $+1, \ldots, +l_{\lambda}$. These nonzero integers mark the starting ``centres'' of the bunches in the modular knot, a more precise definition is given later. \reffig{SplitTemplateWithLabels} shows an example: The maximal exponents are $k_{\mu}=k_1=10$ and $l_{\lambda}=l_3=6$. The negative (or positive) integers mark the intersection points between the ${x_1}^{10}$-arc (or ${y_3}^6$-arc) and the branch line.   

\begin{figure}
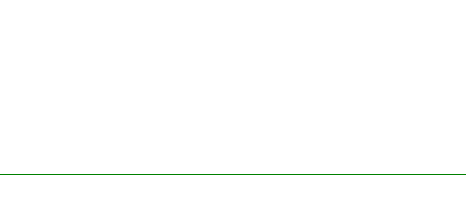\vspace{6pt}
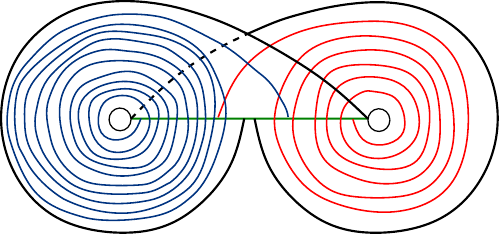
\caption{Top: Split template with the ${x_1}^{10}$-arc and the ${y_3}^6$-arc for the labelled code word ${x_1}^{10}{y_1}^2{x_2}^5{y_2}^2{x_3}^7{y_3}^6{x_4}^2{y_4}^2{x_5}^5{y_5}^3$. Note that the ${x_1}^{10}$-arc has $10$ $x$-turns and the ${y_3}^6$-arc has $6$ $y$-turns.\\  Bottom: Lorenz template with the same arcs, obtained after gluing the split template.
\label{Fig:SplitTemplateWithLabels}}
\end{figure}

\begin{definition} \label{Def:(x_i,j)-turn}
    Let $w = x_1^{k_1} y_1^{l_1} \ldots x_n^{k_n} y_n^{l_n}$ be a labelled code word of a modular knot. Let $i\in[1,n]\cap\ZZ$. For any $j\in[1,k_i]\cap\ZZ$, an \emph{$(x_i, j)$-letter} denotes the $(k_1+l_1+\ldots + k_{i-1} + l_{i-1}+ j)^{\textup{th}}$ letter in $w$, and an \emph{$(x_i, j)$-turn} denotes the $x$-turn corresponding to the $(x_i, j)$-letter. 
    The terms \emph{$(y_i, j)$-letter} and \emph{$(y_i, j)$-turn} are defined similarly. If $w_s$ is a subword of $w$, we call the union of all turns corresponding to the letters in the subword $w_s$ the \emph{$w_s$-arc}. 
\end{definition} 

Take the period-$5$ code word $w=x^{10}y^2x^5y^2x^7y^6x^2y^2x^5y^3$ as an example. \reffig{SplitTemplateWithLabelsTurns} shows two turns for the first and the $11^{\textup{th}}$ letters of the word $w$. 

\begin{figure}
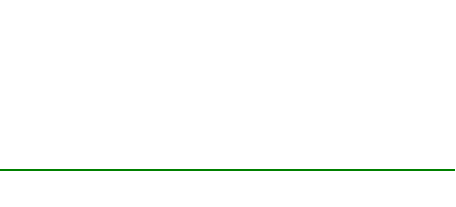
\caption{The $(x_1,1)$-turn and $(y_1,1)$-turn of the labelled code word $w={x_1}^{10}{y_1}^2{x_2}^5{y_2}^2{x_3}^7{y_3}^6{x_4}^2{y_4}^2{x_5}^5{y_5}^3$. These turns correspond to the $1^{\textup{st}}$ and $11^{\textup{th}}$ letters in $w$ respectively. 
\label{Fig:SplitTemplateWithLabelsTurns}}
\end{figure}  

\begin{definition}
    For each integer $i\in\{-k_{\mu}, ..., +l_{\lambda}\}\setminus\{0\}$, we call the interval $(i-0.4, i+0.4)$ in the branch line the \emph{$i^{\textup{th}}$ interval}. When $i$ is negative, we call the interval an \emph{$x$-interval} or a \emph{negative interval}. When $i$ is positive, we call the interval a \emph{$y$-interval} or a \emph{positive interval}. 
\end{definition}

Denote the distinct cyclic permutation classes of labelled code words for a modular link $L$ with $c$ link components by the following: 
\begin{align*}  \label{Words} 
\begin{dcases}
    w_1 &= x_{1}^{k_{1}} y_{1}^{l_{1}} x_{2}^{k_{2}} y_{2}^{l_{21}} \ldots x_{n_1}^{k_{n_1}} y_{n_1}^{l_{n_1}}, \\ 
w_2 &= x_{n_1+1}^{k_{n_1+1}} y_{n_1+1}^{l_{n_1+1}} x_{n_1+2}^{k_{n_1+2}} y_{n_1+2}^{l_{n_1+2}} \ldots x_{n_1+n_2}^{k_{n_1+n_2}} y_{n_1+n_2}^{l_{n_1+n_2}},  \\ 
&\ldots \\ 
w_c &= x_{\Sigma(c) + 1}^{k_{\Sigma(c) + 1}} y_{\Sigma(c) + 1}^{l_{\Sigma(c) + 1}} x_{\Sigma(c) + 2}^{k_{\Sigma(c) + 2}} y_{\Sigma(c) + 2}^{l_{\Sigma(c) + 2}} \ldots x_{\overline{n}}^{k_{\overline{n}}} y_{\overline{n}}^{l_{\overline{n}}}, 
\end{dcases} \tag{*}
\end{align*} 
where $\Sigma(c)$ denotes the sum $\Sigma_{i=1}^{c-1} n_i$ of the word periods for words $w_1, \ldots, w_{c-1}$ if $c>1$; and $\overline{n}$ denotes the sum of all word periods for $L$, that is, $\overline{n} = \Sigma(c+1) = \Sigma_{i=1}^{c} n_i$. Let $\mu, \lambda\in [1,\overline{n}]\cap\ZZ$ such that $k_{\mu}$ and $l_{\lambda}$ are respectively a maximal $x$- and $y$-exponents among all words associated to $L$. 

Observe that there are exactly $\overline{n}$ overcrossing strands from the $x$-split line to the positive side of the branch line and exactly $\overline{n}$ undercrossing strands from the $y$-split line to the negative side of the branch line. As all overcrossing strands do not intersect, the intersection points between the $\overline{n}$ overcrossing strands and the $x$-split line are the rightmost $\overline{n}$ intersection points in the $x$-split line. We may assume all these $\overline{n}$ $x$-turns start within the $-1^{\textup{st}}$ interval $(-1-0.4,-1+0.4)$. Note that each of the $\overline{n}$ overcrossing strands corresponds to each of the $\overline{n}$ $x$-bases in the words associated with $L$.  Similar observations hold for the leftmost undercrossing strands from the $y$-split line to the negative side of the branch line. (See \reffig{SplitTemplateWithLabelsFullBs}.)

\begin{definition} \label{Def:FullBs&OrderWithin} Let $L$ be a modular link with words as denoted in (\ref{Words}).  
\begin{enumerate}
    \item \label{Def:FullBs} The \emph{full $x$-bunch} in the modular link $L$ is the ordered $\overline{n}$-tuple of the $\overline{n}$ rightmost $x$-turns in the Lorenz template, denoted by \[((x_{\sigma(1)},k_{\sigma(1)}),\ldots,(x_{\sigma(\overline{n})},k_{\sigma(\overline{n})})),\] from left to right in the split template. Similarly,  the \emph{full $y$-bunch} in $L$ is the ordered $\overline{n}$-tuple of the $\overline{n}$ leftmost $y$-turns, denoted by $((y_{\tau(1)},l_{\tau(1)}),\ldots,(y_{\tau(\overline{n})},l_{\tau(\overline{n})}))$, from left to right in the split template. 
    
    \item \label{Def:OrderWithin} The \emph{order within the full $x$-bunch} is the corresponding ordered $\overline{n}$-tuple $(x_{\sigma(1)}, \ldots, x_{\sigma(\overline{n})})$ of labelled $x$-bases. The \emph{order within the full $y$-bunch} is the corresponding ordered $\overline{n}$-tuple $(y_{\tau(1)}, \ldots, y_{\tau(\overline{n})})$ of labelled $y$-bases. 
\end{enumerate} 
\end{definition}  

\begin{remark}
    We use ordered tuples of labelled bases instead of just their subscripts in 
 Definition~\ref{Def:FullBs&OrderWithin}(\ref{Def:OrderWithin}) because the former is invariant under word cyclic permutations while the latter is not. The functions $\sigma$ and $\tau$ in Definition~\ref{Def:FullBs&OrderWithin} are permutations on the set $\{1, 2,\ldots, \overline{n}\}$. 
\end{remark}

Take the labelled (or equivalently, coloured) code word $w$ in \reffig{SplitTemplateWithLabelsFullBs} as an example. The blue dots in $-1^{\textup{st}}$ interval denote the start points of the five $x$-turns of different intensities (left to right: dark blue $x_2$, very dark blue $x_1$, pale blue $x_4$, very pale blue $x_5$, blue $x_3$), which form a full $x$-bunch $((x_2, 5),(x_1, 10),(x_4, 2),(x_5, 5),(x_3, 7))$ in the modular knot with word $w$. Similarly, the red dots in the $+1^{\textup{st}}$ interval denote the start points of the five $y$-turns, which form a full $y$-bunch $((y_5, 3),(y_2, 2),(y_1, 2),(y_4, 2),(y_3, 6))$. 

\begin{figure}
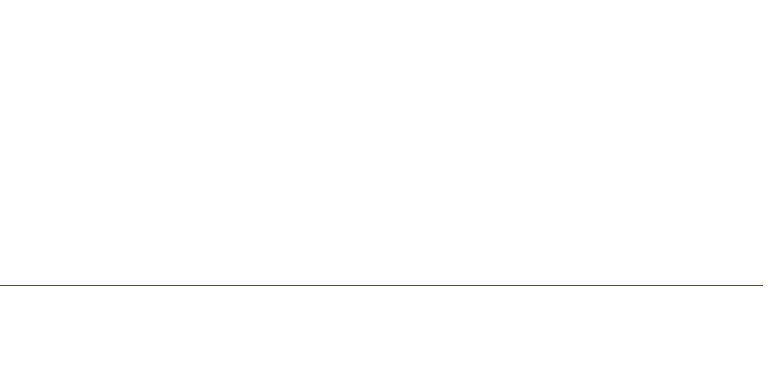
\caption{A full $x$-bunch (blue) and a full $y$-bunch (red) in the modular knot with labelled (or coloured) code word $w$.
\label{Fig:SplitTemplateWithLabelsFullBs}} 
\end{figure}

Observe that the start point of each turn of the full $x$-bunch in the $-1^{\textup{st}}$ interval marks the end of either an overcrossing strand or an undercrossing strand in the split template. We consider only the points that are the ends of the overcrossing strands, the ordered collection of such overcrossing strands (from left to right as appeared in the split template) is the list below: 
\begin{align*}
    (x_{21},k_{21}-1)\text{-turn}, \\
    (x_{22},k_{22}-1)\text{-turn}, \\
    \ldots,  \\
    (x_{2(b_{-2})},k_{2(b_{-2})}-1)\text{-turn},
\end{align*}
 where $(x_{21}, x_{22}, \ldots, x_{2(b_{-2})})$ is a subsequence of $(x_{\sigma(1)}, x_{\sigma(2)}, \ldots, x_{\sigma(\overline{n})})$. We call such list the \emph{negative second $x$-bunch}, denoted by an ordered $(b_{-2})$-tuple 
\[\bigg(\big(x_{21}, k_{21}-1\big),\ldots, \big(x_{2(b_{-2})}, k_{2(b_{-2})}-1\big)\bigg)\] 
of ordered pairs. 
Without loss of generality, we assume all the turns in the $-2^{\textup{nd}}$ $x$-bunch start from the  $-2^{\textup{nd}}$ interval $(-2-0.4,-2+0.4)$. Note that there are two indices in the subscript of each letter $x$, the first one, which is $2$ in this case, corresponds to the label of bunch or interval.  

Now consider the start points of turns in the $-2^{\textup{nd}}$ interval. Some of them mark the ends of some overcrossing strand(s) in the split template. The ordered collection of such overcrossing strands (from left to right as appeared in the split template) is the list of the following turns: 
\begin{align*}
    (x_{31},k_{31}-2)\text{-turn}, \\
    (x_{32},k_{32}-2)\text{-turn}, \\
    \ldots, \\
    (x_{3(b_{-3})},k_{3(b_{-3})}-2))\text{-turn},
\end{align*}
 where $(x_{31}, x_{32}, x_{33}, \ldots, x_{3(b_{-3})})$ is a subsequence of $(x_{21}, x_{22}, x_{23}, \ldots, x_{2(b_{-2})})$. We call such list the \emph{negative third $x$-bunch}, denoted by an ordered $(b_{-3})$-tuple 
\[\bigg(\big(x_{31}, k_{31}-2\big),\ldots, \big(x_{3(b_{-3})}, k_{3(b_{-3})}-2\big)\bigg)\] 
of ordered pairs. Without loss of generality, we assume that all the turns in the $-3^{\textup{rd}}$ $x$-bunch start from the $-3^{\textup{rd}}$ interval $(-3-0.4,-3+0.4)$.  

Inductively, we can define the $u^{\textup{th}}$ $x$-bunch for any $u\in [-k_{\mu},-1]\cap\ZZ$; and assume all turn(s) in the $u^{\textup{th}}$ $x$-bunch start(s) from the $u^{\textup{th}}$ interval. Here comes the more precise definition of bunches:  

\begin{definition}[Bunches in a modular link] \label{Def:Bunches&OrderWithin} Let $L$ be a modular link with labelled code words denoted in (\ref{Words}).  Let $(x_{\sigma(1)}, \ldots, x_{\sigma(\overline{n})})$ and $(y_{\tau(1)}, \ldots, y_{\tau(\overline{n})})$ be orders within the full $x$-bunch and full $y$-bunch respectively. 

\begin{enumerate}
    \item  \label{Def:xBunches&OrderWithin}  For any $u\in [1, k_{\mu}]\cap\ZZ$, the \emph{$-u^{\textup{th}}$ $x$-bunch} in $L$ is the ordered $(b_{-u})$-tuple 
    \[\bigg(\big(x_{u1},k_{u1}-u+1),\ldots,(x_{u(b_{-u})},k_{u(b_{-u})}-u+1\big)\bigg)\] 
    of $x$-turns starting from the $-u^{\textup{th}}$ interval, where the \emph{order within the $x$-bunch} $(x_{u1},  \ldots, x_{u(b_{-u})})$ is a subsequence of $(x_{\sigma(1)}, \ldots, x_{\sigma(\overline{n})})$. 
    \item  \label{Def:yBunches&OrderWithin}  For any $v\in [1,l_{\lambda}]\cap\ZZ$, the \emph{$+v^{\textup{th}}$ $y$-bunch} in $L$ is the ordered $(b_{+v})$-tuple 
    \[((y_{v1},l_{v1}-v+1),\ldots,(y_{v(b_{+v})},l_{v(b_{+v})}-v+1))\] 
    of $y$-turns starting from the $+v^{\textup{th}}$ interval, where the \emph{order within the $y$-bunch} $(y_{v1},  \ldots, y_{v(b_{+v})})$ is a subsequence of $(y_{\tau(1)}, \ldots, y_{\tau(\overline{n})})$. 
\end{enumerate} 
A \emph{bunch} in a modular link $L$ is either an $x$-bunch or a $y$-bunch in $L$.  The union of all entries in a bunch is said to be a \emph{bunch of turns}. 
\end{definition} 

\reffig{SplitTemplateWithLabelsWhole} illustrates a modular knot with ten $x$-bunches and six $y$-bunches. Note that each $x$-bunch (or $y$-bunch) has exactly one $x_{\mu}$-turn (or $y_{\lambda}$-turn). 

\begin{figure}
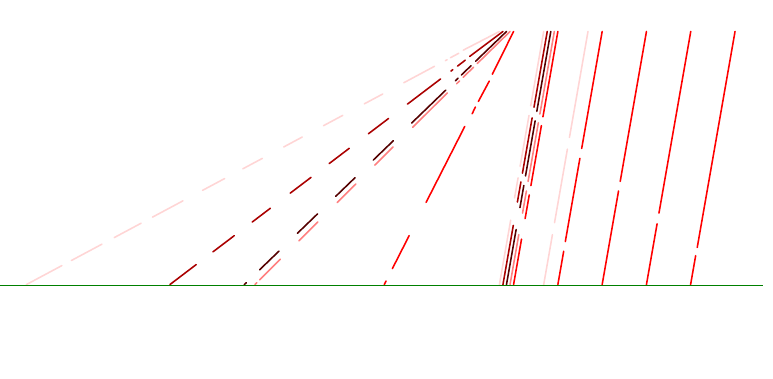 
\caption{A modular knot with coloured (or equivalently, labelled) code word $w$ in the split template. Each turn in the split template has the same colour as the letter in the word. This knot has 10 $x$-bunches and 6 $y$-bunches, where 10 is the maximal $x$-exponent and 6 is the maximal $y$-exponent in the code word.  
\label{Fig:SplitTemplateWithLabelsWhole}}
\end{figure}   

The following is the first main result. 

\begin{theorem} \label{Thm:UnionBunchesOfTurns}
Each modular link $L$ with labelled code words denoted in (\ref{Words}) is ambient isotopic to a union of bunches of turns that are ordered from left to right in the split template according to the order of turns listed in each bunch, and the union of bunches satisfy both conditions below:
\begin{enumerate}
    \item For any $i\in [1,\overline{n}]\cap\ZZ$, the $x_i^{k_i}$-arc starts from a point in the $-{k_i}^{\textup{th}}$ interval. 
    \item For any $j\in [1,\overline{n}]\cap\ZZ$, the $y_j^{l_j}$-arc starts from a point in the $+{l_j}^{\textup{th}}$ interval. 
\end{enumerate}
\end{theorem} 

\begin{proof} 
Observe that any modular link embedded in the Lorenz template can be ambient-isotoped to a position such that all of its intersections with the branch line lie in the union of all positive and negative intervals, we may assume all intersection points between $L$ and the branch line lie in either positive or negative intervals. Since each word for $L$ contains both $x$ and $y$, each $x$-turn starts from a negative real number in the branch line and ends at a strictly larger real number.  Similarly, each $y$-turn starts from a positive real number and ends at a strictly smaller real number. 

The following argument focuses on explaining why we can view collections of overcrossing strands (or more precisely, collections of $x$-turns) as $x$-bunches. The case for $y$-bunches follows similarly.  

Let $((x_{\sigma(1)},k_{\sigma(1)}),\ldots,(x_{\sigma(\overline{n})},k_{\sigma(\overline{n})}))$ be the full $x$-bunch in $L$, i.e., the $-1^{\textup{st}}$ bunch in $L$. Note that each of the start points in the $-1^{\textup{st}}$ bunch is the end of either an overcrossing strand (i.e., an $x$-turn) or an undercrossing strand (i.e., a $y$-turn). If there are overcrossing strands that end at the $-1^{\textup{st}}$ interval, we may ambient-isotope $L$ such that all these overcrossing strands start from the $-2^{\textup{nd}}$ interval. 

Since all overcrossing strands are pairwise non-intersecting and the template where the modular link is embedded is a two-dimensional object, the order of bases corresponding to all $x$-turns (from left to right as appeared in the split template) starting from the $-2^{\textup{nd}}$ interval 
is a subsequence of the order within the full $x$-bunch. Thus, by Definition~\ref{Def:Bunches&OrderWithin}(\ref{Def:xBunches&OrderWithin}), the ordered tuple of those $x$-turns is a $-2^{\textup{nd}}$ $x$-bunch, which can be denoted by  \[\bigg(\big(x_{21}, k_{21}-1\big), \big(x_{22}, k_{22}-1\big),\ldots,\big(x_{2(b_{-2})}, k_{2(b_{-2})}-1\big)\bigg), \] 
where $(x_{21}, x_{22}, \ldots, x_{2(b_{-2})})$ is a subsequence of $(x_{\sigma(1)}, x_{\sigma(2)}, \ldots, x_{\sigma(\overline{n})})$.  

Inductively, for any $m\in [1, k_{\mu}-1]\cap\ZZ$, we can ambient-isotope  $L$ (without affecting the already defined bunches) such that all the $x$-turns that end 
at the $-m^{\textup{th}}$ interval are starting from the $-(m+1)^{\textup{th}}$ interval.  The ordered tuple of those $x$-turns (from left to right as appeared in the split template) is a $-(m+1)^{\textup{th}}$ $x$-bunch, which can be denoted by an ordered $(b_{-(m+1)})$-tuple  
\[\bigg(\big(x_{(m+1)1}, k_{(m+1)1}-m\big),\ldots,\big(x_{(m+1)(b_{-(m+1)})}, k_{(m+1)(b_{-(m+1)})}-m\big)\bigg), \]
where the order within the $-(m+1)^{\textup{th}}$ bunch is a subsequence of the order within the $-m^{\textup{th}}$ bunch. The order within the $-(m+1)^{\textup{th}}$ bunch is thus a subsequence of the order within the full $x$-bunch $(x_{\sigma(1)}, x_{\sigma(2)}, \ldots, x_{\sigma(\overline{n})})$.  

Similarly, we can view collections of undercrossing strands as $y$-bunches. Note that the order within each $y$-bunch describes the order of $y$-turns from left to right as appeared in the split template. 

Since the second entries of the ordered pairs, namely, 
\[k_{(m+1)1}-m, \ldots, k_{(m+1)(b_{-(m+1)})}-m,\] 
are positive integers by \refdef{(x_i,j)-turn}, it follows that for any $i\in [1,\overline{n}]\cap\ZZ$, the $x_i^{k_i}$-arc starts from a point in the $-k_i^{\textup{th}}$ interval. 
Similarly, for any $j\in [1,\overline{n}]\cap\ZZ$, the $y_j^{l_j}$-arc starts from a point in the $+l_j^{\textup{th}}$ interval. 
\end{proof}

\reffig{SplitTemplateWithLabelsWhole} shows a modular knot with code word $x^{10}y^2x^5y^2x^7y^6x^2y^2x^5y^3$ in the split template. The start point of the $x_3^7$-arc is the rightmost blue dot in the $-7^{\textup{th}}$ interval. The start point of the $y_4^2$-arc is the pale red dot (the second rightmost dot) in the $+2^{\textup{nd}}$ interval. 

\subsection{Order within a full bunch} 
If we colour the subwords $x_1^{k_1}, \ldots, x_{\overline{n}}^{k_{\overline{n}}}$ of the labelled code words denoted in (\ref{Words}) and their corresponding unions of $x$-turns with decreasing intensity of blue (see \reffig{SplitTemplateWithLabelsWhole} for example of a modular knot with $\overline{n}=5$), we can observe that the full $x$-bunch consists of blue strands of $\overline{n}$ different intensities and are permuted in a way that is not necessarily a decreasing order of intensities. Similar observations hold if we colour the subwords $y_1^{l_1}, \ldots, y_{\overline{n}}^{l_{\overline{n}}}$ and their corresponding collections of $y$-turns with decreasing intensity of red. 

Now a question arises: How do we determine the colour intensity orders for the blue and red strands in the full bunches? Or equivalently, how do we determine the order within the full $x$-bunch and the order within the full $y$-bunch (see Definition ~\ref{Def:FullBs&OrderWithin}(\ref{Def:OrderWithin}))?  It turns out such orders depend on the descending order of the $x$-exponents and/or the ascending order of the $y$-exponents, a more precise statement is \refprop{OrderWithinFullBs}, the proof of which is followed by an example of finding order within full bunches. 

From now on, unless otherwise specified, we assume the following: 
\begin{itemize} 
    \item All subscripts of bases and exponents of the word $w_1$ in (\ref{Words}) are living in the cyclic group $\ZZ_{n_1} = (\{1, \ldots, n_1\}, +)$. 
    \item If the number $c$ of link components of the modular link $L$ is greater than one, then for any $r\in [2,c]\cap\ZZ$, all subscripts of bases and exponents of the word $w_r$ in (\ref{Words}) are living in the cyclic group $\ZZ_{n_r}~=~(\{\Sigma(r)+1, \ldots, \Sigma(r)+n_r\}, +)$, where $\Sigma(r)\coloneqq \Sigma_{i=1}^{r-1} n_i$. 
\end{itemize}

\begin{proposition} \label{Prop:OrderWithinFullBs}
        The order within a full $x$-bunch in a modular link $L$ with labelled code words denoted in (\ref{Words})
can be determined as follows: 
        \begin{enumerate}[noitemsep]
            \item \label{Case&Statement:n=1} If $\overline{n}=1$, then the order within the full $x$-bunch is $(x_1)$.
            \item \label{Case:n>1} Suppose $\overline{n}>1$. 
            \begin{enumerate}
                \item \label{Statement:AllDistinct} If the $y$-exponents $l_1, l_2, \ldots, l_{\overline{n}}$ are all distinct and their ascending order is \[(l_{\sigma(1)}, l_{\sigma(2)}, \ldots, l_{\sigma(\overline{n})})\] 
                for some permutation map $\sigma \from \{1,\ldots,\overline{n}\} \to \{1,\ldots,\overline{n}\}$, then the order within the full $x$-bunch is $(x_{\sigma(1)}, x_{\sigma(2)}, \ldots, x_{\sigma(\overline{n})})$. 
             
                \item \label{Statement:SomeSame}  If the $y$-exponents $l_1, l_2, \ldots, l_{\overline{n}}$ are not all distinct, then the steps below determine the order within the full $x$-bunch. 

                \begin{enumerate}
                    \item Order as many labelled bases $x_{1}, x_{2}, \ldots, x_{\overline{n}}$ as possible according to an \textbf{ascending} order of $l_1, l_2, \ldots, l_{\overline{n}}$.  
                    \item For the $l_i$'s that equal to the same integer, order the corresponding $x_i$'s as many as possible according to a \textbf{descending} order of the corresponding $k_{i+1}$'s. 
                    \item If two or more $k_{i+1}$'s equal to the same integer, order the corresponding $x_i$'s as many as possible according to an \textbf{ascending} order of the corresponding $l_{i+1}$'s. 
                    \item If two or more $l_{i+1}$'s equal to the same integer, order the corresponding $x_i$'s as many as possible according to a \textbf{descending} order of the corresponding $k_{i+2}$'s. 
                \end{enumerate}

                By repeating the steps recursively and finitely many times, all labelled $x$-bases can be ordered, and such ordered $\overline{n}$-tuple is the order within the full $x$-bunch. 
            \end{enumerate}
            
        \end{enumerate}

        Similarly, the order within the full $y$-bunch can be determined by ordering the labelled $y$-bases according to descending order(s) of $x$-exponents and ascending order(s) of $y$-exponents. 
\end{proposition}

\begin{proof}
Statement~(\ref{Case&Statement:n=1}) holds because the link is a modular knot with period $1$ and the full $x$-bunch consists of only one $x$-turn.  

Suppose $\overline{n}>1$. Recall that the $x$-turns in the full $x$-bunch are the only $x$-turns that start from a negative interval and end at a positive interval. The ends of all these $x$-turns are the starts of the $y$-turns that immediately follow. Hence, the order of the $x$-turns in the full $x$-bunch is determined by the order of the starting points of $y$-turns that immediately follow.  

By \refthm{UnionBunchesOfTurns}, each $y_j^{l_j}$-arc starts from a point in the $+{l_j}^{\textup{th}}$ interval. Hence the ascending order of $l_1, l_2, \ldots, l_n$ determines the order of the $x$-turns in the full $x$-bunch (from left to right) in the split template if $l_1, l_2, \ldots, l_n$ are all distinct.  Thus, Statement~(\ref{Statement:AllDistinct}) holds.  

If two or more $l_i$'s are the same integer such that the order of the corresponding $x_i$-turns cannot be determined,  we can trace through the corresponding $y_i^{l_i}$-arcs, whose ends mark the starting points of the $x_{i+1}$-turns that immediately follow the arcs.  By \refthm{UnionBunchesOfTurns}, each $x_{i+1}^{k_{i+1}}$-arc starts from a point in the $-{k_{i+1}}^{\textup{th}}$ interval. Hence, if all such $k_{i+1}$'s are distinct, we can order the corresponding $x_i$-turns (from left to right in the split template) according to the descending order of those $k_{i+1}$'s. Otherwise, we trace through the $x_{i+1}^{k_{i+1}}$-arcs that have the same exponent and repeat similar ordering processes. 

As the word is primitive and finite, the order of all labelled $x$-bases in the full $x$-bunch can be determined after finitely many steps. 

The proof for the order within the full $y$-bunch follows similarly. 
\end{proof} 

Take the labelled code word $w={x_1}^{10}{y_1}^2{x_2}^5{y_2}^2{x_3}^7{y_3}^6{x_4}^2{y_4}^2{x_5}^5{y_5}^3$ in \reffig{SplitTemplateWithLabelsWhole} as an example.  To determine the order within the full $x$-bunch, we go through the following steps. 
\begin{enumerate}
    \item[1.] One possible choice of the ascending order is $(l_1, l_2, l_4, l_5, l_3)=(2,2,2,3,6)$, we can then order as many $x$-bases as possible \[(\{x_1, x_2, x_4\}, x_5, x_3),\] where the order of $x_1, x_2, x_4$ is not known at this stage. 
    \item[2.]  A descending order of $k_{1+1}, k_{2+1}, k_{4+1}$ is $(k_{2+1},k_{1+1},k_{4+1})$ $=(7,5,5)$. Further order as many $x$-bases in the above curly brackets as possible: \[(x_2, \{x_1, x_4\}, x_5, x_3),\] where the order of $x_1, x_4$ is not known at this stage. 
    \item[3.]  The ascending order of $l_{1+1}, l_{4+1}$ is $(l_{1+1}, l_{4+1})=(2,3)$, where there are no identical exponents. The order within the full $x$-bunch is thus \[(x_2, x_1, x_4, x_5, x_3).\] 
\end{enumerate} 

Similarly, to determine the order within the full $y$-bunch, we look at the descending order $(k_1, k_3, k_2, k_5, k_4)=(10,7,5,5,2)$ to obtain $(y_5, y_2, \{y_1, y_4\}, y_3)$ and then from the ascending order $(l_2,l_5)=(2,3)$, we can conclude the order within the full $y$-bunch is $(y_5, y_2, y_1, y_4, y_3)$.  
 
\subsection{The bunch algorithm for constructing modular links} \label{Sec:Algorithm}

The following corollary follows from \refthm{UnionBunchesOfTurns}. 

\begin{corollary}[The bunch algorithm] \label{Cor:Algorithm}
Let $L$ be a modular link with labelled code words denoted in (\ref{Words}). Recall that $\overline{n}$ denote the sum of all word periods associated with $L$. Suppose $K$ is a link component of $L$ with labelled code word $w_1 = x_1^{k_1} y_1^{l_1} \ldots x_{n_1}^{k_{n_1}} y_{n_1}^{l_{n_1}}$. The link component $K$ can be drawn in the split template using the following steps: 
\begin{enumerate}
    \item Mark the branch line with integers $-k_{\mu}, \ldots, +l_{\lambda}$, where $k_{\mu}$ and $l_{\lambda}$ are a maximal $x$- and $y$-exponents among all words associated to $L$ respectively. 
    
    \item Determine the orders within the full bunches in the modular link $L$ using \refprop{OrderWithinFullBs}. In each negative (or positive) interval, draw $\overline{n}$ points with $x$-base (or $y$-base) labels from left to right according to the order within the full $x$-bunch (or full $y$-bunch). 
    
    \item Construct $K$ by first drawing the $x_1^{k_1}$-arc, which starts from the $x_1$-point in the $-{k_1}^{\textup{th}}$ interval and ends at the $y_1$-point in the $+{l_1}^{\textup{th}}$ interval. Then, draw the $y_1^{l_1}$-arc which starts from the $y_1$-point in the $+{l_1}^{\textup{th}}$ interval and ends at the $x_2$-point in the $-{k_2}^{\textup{th}}$ interval. Continue drawing the arcs for subwords $x_2^{k_2}, y_2^{l_2}, \ldots, x_{n_1}^{k_{n_1}}, y_{n_1}^{l_{n_1}}$ with known starting and ending points to obtain $K$ in the split template. 
\end{enumerate}
    
By repeating Step (3) for all other link components of $L$ with code words $w_2, w_3, \ldots, w_c$, the modular link (or called Lorenz link) $L$ embedded in the Lorenz template can be obtained. 
\end{corollary} 

\begin{remark}
    We can actually draw any subwords of any code words in any order in \refcor{Algorithm}.  The modular knot $K$ can be obtained as long as the turn corresponding to each letter in the word is drawn. 
\end{remark} 

Take the labelled code word $w={x_1}^{10}{y_1}^2{x_2}^5{y_2}^2{x_3}^7{y_3}^6{x_4}^2{y_4}^2{x_5}^5{y_5}^3$ as an example. To construct the modular knot with code word $w$, instead of marking $44$ points in the branch line in William's algorithm (\reffig{WilliamsAlg1}), the bunch algorithm needs to mark only integers from $-10$ to $+6$. Instead of finding the lexicographic order of the $44$ members in the permutation class of words, the bunch algorithm needs to order only five $x$-bases and five $y$-bases, which can be found almost immediately from the relative magnitudes of the word exponents. Such orders were $(x_2, x_1, x_4, x_5, x_3)$ and $(y_5, y_2, y_1, y_4, y_3)$, where detailed explanation can be found in the example following \refprop{OrderWithinFullBs}.  

Having known the order within full $x$-bunch at each negative integer on the branch line and the order within full $y$-bunch at each positive integer, we can construct the modular knot as appeared in \reffig{SplitTemplateWithLabelsWhole} just by reading the code word exponents, e.g. to draw the $({x_1}^{10}{y_1}^2)$-arc, we start from the $x_1$-point at the integer $-10$ in the branch line, and then draw ten $x$-turns with the last turn ending at the $y_1$-point at $+2$ in the branch line. We then draw two $y$-turns with the last turn ending at the $x_2$-point at $-5$ in the branch line. The orders within the full bunches allow us to know where each turn should end.

\subsection{Why are we introducing the notion of bunches?}  \label{Sec:WhyBunches}
The notion of bunches allows us to more easily find the parent links and annuli for the annular Dehn filling technique that appears in \refsec{UpperVolBound} of this paper and \cite{RodriguezMigueles:PeriodsOfContinuedFractions}. This new perspective also provides a more direct relationship between the geometry of a modular link and its associated code word. For example, viewing a modular link as a union of bunches allows us to tell the starting and ending points of the arc corresponding to a subword directly from the subscripts and exponents in the subword.  This allows us to compare with other bounds of volumes as discussed in \refsec{Discussion}.  

The key step in the bunch algorithm is finding the orders within the full bunches in the modular link $L$ using \refprop{OrderWithinFullBs}, which involves ordering only $\overline{n}$ $x$-bases and $\overline{n}$ $y$-bases, where $\overline{n}$ is the sum of word periods or the braid index (see Remark~\ref{Rmk:TripNum_BraidIndex_PeriodContFraction}). Note that Williams' algorithm involves $\Sigma_{i=1}^{\overline{n}} (k_i + l_i)$ items to be ordered, where $\Sigma_{i=1}^{\overline{n}} (k_i + l_i)$ is the sum\footnote{This sum is also called the \emph{sum of lengths of words} in literature.} of all word exponents $k_i$'s and $l_i$'s denoted in (\ref{Words}). As word exponents are positive integers, we have $\Sigma_{i=1}^{\overline{n}} (k_i + l_i) \geq 2\overline{n}$. The bunch algorithm involves a smaller number of items to be ordered, especially when the word exponents are large. 

\section{An upper volume bound for all modular link complements} \label{Sec:UpperVolBound}

A key result that we use in proving the upper volume bound is Theorem~1.5 in \cite{Cremaschi-RodriguezMigueles:HypOfLinkCpmInSFSpaces}, such result can be stated as follows: 

\begin{theorem}[Theorem~1.5 in \cite{Cremaschi-RodriguezMigueles:HypOfLinkCpmInSFSpaces}]\footnote{By abuse of notation, the symbol $\calL$ in the statement represents an embedding in the function composition and represents a set otherwise. }  \label{Thm:SelfIntersectionNumber}
Let $M$ be a Seifert-fibred space over a hyperbolic $2$-orbifold $B$ with bundle projection map $\calP\from M\to B$. If $\calL$ is a link embedded in $M$ such that $\calP\circ \calL$ is a collection of loops that intersect themselves only transversely and finitely many times with exactly two pre-image points for each self-intersection point, 
then the simplicial volume of $M\setminus \calL$ satisfies the following inequality: 
\[\norm{M\setminus \calL} \leq 8\ \iota(\calP(\calL),\calP(\calL)), \] 
where $\iota(\calP(\calL),\calP(\calL))$ is the number of self-intersections of $\calP(\calL)$. 
\end{theorem} 

For details on simplicial volume and its relationship with hyperbolic volume, please refer to \cite[6.5.4]{Thurston:Geom&TopOf3Mfd}. A fact that we will use is $v_{\text{tet}}\norm{M} = \mathrm{Vol}(M)$ if $M$ is a complete hyperbolic manifold of finite volume and $v_{\textup{tet}}\approx 1.01494$ denotes the volume of the regular ideal tetrahedron. 


Given a modular link complement $\mathrm{T^1}S_{\textup{mod}} \setminus L$ with the modular link $L$ embedded in the Lorenz template, we can obtain the $3$-manifold $\mathrm{T^1}S_{\textup{mod}} \setminus L$ by Dehn filling some parent manifold, the construction of which can be found in \refdef{ParentMfd}. We can view such parent manifold as a complement of a link in the Seifert-fibred space with base space equal to a punctured disc and find its upper volume bound using \refthm{SelfIntersectionNumber}. As Dehn filling decreases volume \cite{Thurston:Geom&TopOf3Mfd}, we can thus obtain an upper volume bound for all modular link complements such that the bound depends only on the sum of word periods of $L$. \reffig{LinkCompleInSFSpaceResized} shows an example of the Seifert-fibred space that we will apply \refthm{SelfIntersectionNumber} to.

The type of Dehn surgery that we use is a special kind called the annular Dehn surgery. Before giving a more precise definition, let us clarify our sign of slope or the frame on each torus boundary. Applying a Dehn surgery to an unknot $\partial_{\textup{out}} A$ embedded in $\SS^3$ along the slope $+\frac{1}{n}$ (for some positive integer $n$) means a new $3$-manifold can be obtained by the following steps:
\begin{enumerate}
    \item Drill $\SS^3$ along the unknot $\partial_{\textup{out}} A$, 
    \item Cut $\SS^3\setminus N(\partial_{\textup{out}} A)$ along a compression disc bounded by $\partial N(\partial_{\textup{out}} A)$ to obtain a solid cylinder.\footnote{The notation $N(S)$ denotes an open tubular neighbourhood of a $1$- or $2$-submanifold $S$.} 
    \item Twist one end of the solid cylinder (viewed from the interior) anti-clockwise for $2\pi n$.  
    \item Glue back the two bases of the solid cylinder naturally. 
    \item Trivially Dehn-fill the unknot complement obtained from the last step. 
\end{enumerate} 

\begin{definition}
    Given an annulus $A$ embedded with soul the unknot, such that the outer boundary $\partial_{\textup{out}} A$ do not twist along its soul, in a plane in $\SS^3 = \RR^3\cup \{\infty\}$. Applying Dehn surgery to $\partial_{\textup{out}} A$  in $\SS^3$ along the slope $+\frac{1}{n}$ and applying Dehn surgery to the inner boundary $\partial_{\textup{in}} A$ in $\SS^3$ along the slope $-\frac{1}{n}$ is said to be \emph{applying $(+\frac{1}{n})$-annular Dehn surgery to $A$}.  
\end{definition}

\reffig{AnnularDehnFilling} shows an example of a
 $(+\frac{1}{3})$-annular Dehn surgery. 

\begin{figure}
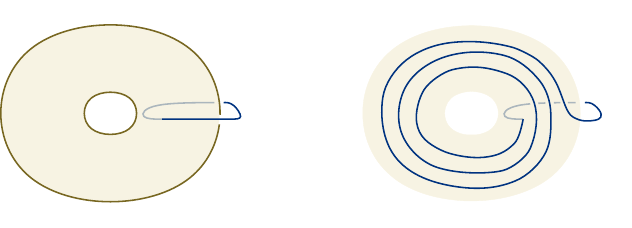
\caption{Applying $(+\frac{1}{3})$-annular Dehn surgery to $A$, where $R$ is an unknot intersecting $A$ once before the surgery.  
\label{Fig:AnnularDehnFilling}}
\end{figure} 

Next, we define the Seifert-fibred space, parent link, and parent manifold of a modular link. Let $L$ be a modular link with labelled code words denoted in (\ref{Words}).  Let $(d_{\#k}, \ldots, d_1)$ be the descending order of the $x$-exponents of all words of $L$, where $\#k$ is the total number of distinct $x$-exponents.  Let $(a_{1}, \ldots, a_{\#l})$ be the ascending order of the $y$-exponents of all words of $L$, where $\#l$ is the total number of distinct $y$-exponents. 

\begin{definition} \label{Def:SFSpaceOfModularLink}
    The \emph{Seifert-fibred space of the modular link} $L$ is the Seifert-fibred space $\SS^3\setminus \calL_{\textup{SF}}$ over a punctured disc, where $\calL_{\textup{SF}}$ is a link in $\SS^3$ consisting of the following link components (see \reffig{LinkCompleInSFSpaceResized_JustSF} for example): 
    \begin{itemize}
        \item an unknot $U$ that passes through each of the two holes of the Lorenz template once, 
        \item all boundary components of the concentric annuli $A^x_{d_1}$ (furthest from the left hole), \ldots, $A^x_{d_{\#k}}$ (closest to the left hole) embedded in the left ear of the Lorenz template with the left hole of the template as their common centre, that is, a collection of simple closed curves passing through points along the left side of the branch line, and  
        \item all boundary components of the concentric annuli $A^y_{a_1}$ (furthest from the right hole), \ldots, $A^y_{a_{\#l}}$ (closest to the right hole) embedded in the right ear of the Lorenz template with the right hole of the template as their common centre, that is, a collection of simple closed curves passing through points along the right side of the branch line. 
    \end{itemize}
\end{definition}

\begin{figure}
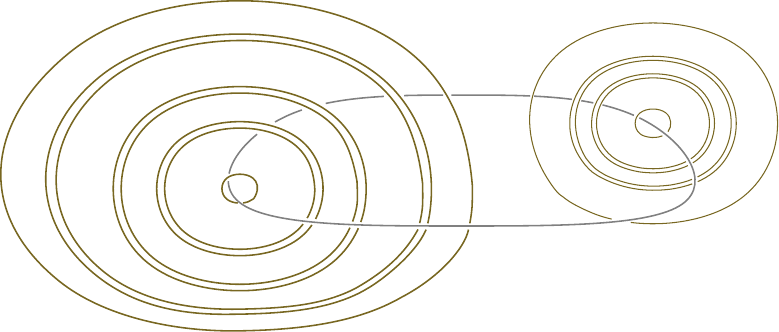
\caption{Seifert-fibred space $\SS^3\setminus \calL_{\textup{SF}}$ of the modular knot $K$ with code word $x^{10}y^2x^5y^2x^7y^6x^2y^2x^5y^3$, where $\calL_{\textup{SF}}$ consists of the unknot $U$ and all annuli boundaries.   The base space of the Seifert-fibred space is equal to a $14$-punctured disc bounded by $U$. 
\label{Fig:LinkCompleInSFSpaceResized_JustSF}}
\end{figure}

\begin{remark}
    The Seifert-fibred space $\SS^3\setminus \calL_{\textup{SF}}$ of the modular link $L$ is a Seifert-fibred space with base space equal to a $2(\#k+\#l)$-punctured disc bounded by $U$. \reffig{LinkCompleInSFSpaceResized} shows an example. 
\end{remark} 

\begin{figure}
\captionsetup[subfigure]{position=b}
\centering
\subcaptionbox{Left: Step~(1): Draw the horizontal line segments based on the orders within full bunches, which are $(x_2, x_1, x_4, x_5, x_3)$ and $(y_5, y_2, y_1, y_4, y_3)$. \\
Right: The full picture illustrating Step~(1). This indicates where the horizontal line segments are located in \reffig{LinkCompleInSFSpaceResized}.
\label{Fig:Step1} }
{\hspace{-50pt}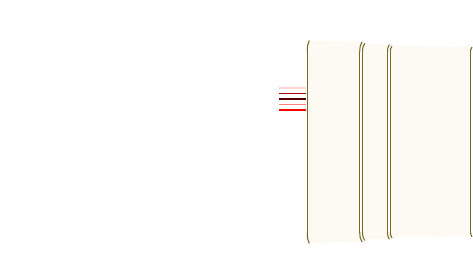  \hspace{20pt}
 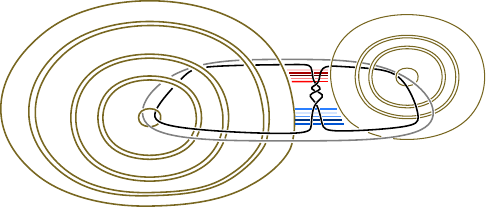} 

\vspace{20pt}
\subcaptionbox{Left: Step~(2a): Extend the line segments based on the exponents in the code word. E.g. Since $x_1$ has exponent $k_1=10$, the $x_1$-line segment (i.e. the darkest blue line segment) is extended such that it intersects the annuli $A^x_{2}$, $A^x_{5}$, $A^x_{7}$, and $A^x_{10}$. \\
Right: Step~(2b): Further extend the other side of each $x_i$-line segment (blue) to match with the extended $y_i$-line segment (red). Similarly, further extend the other side of each $y_i$-line segment to match with the extended $x_{i+1}$-line segment. 
\label{Fig:Step2ab}}
{\hspace{-50pt}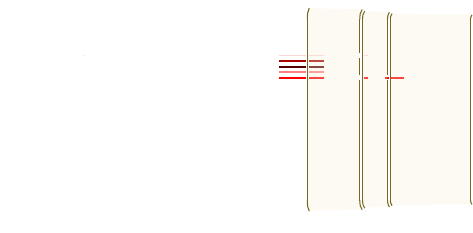 \hspace{20pt} 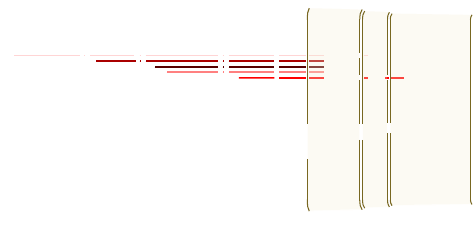}  

\vspace{20pt}
\subcaptionbox{Step~(3): Connect the endpoints via vertical line segments. (We are viewing the parent knot projection from above the disc bounded by the unknot $U$.) 
\label{Fig:Step3} }
{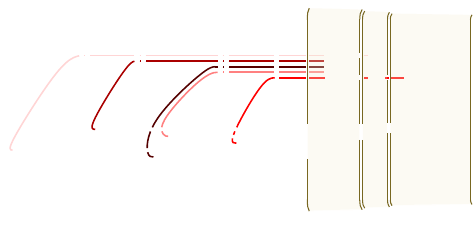}
\caption{Construction of the parent knot of the modular knot $K$ with word $x^{10}y^2x^5y^2x^7y^6x^2y^2x^5y^3$. \label{Fig:ParentKnotSteps_New}} 
\end{figure}

\begin{definition}[Parent link] \label{Def:ParentLink}
Suppose $(x_{\sigma(1)}, \ldots, x_{\sigma(\overline{n})})$ and $(y_{\tau(1)}, \ldots, y_{\tau(\overline{n})})$ are orders within the full $x$-bunch and full $y$-bunch in the modular link $L$ respectively (We can use \refprop{OrderWithinFullBs} to find these orders). 
The \emph{parent link} $L_{\textup{p}}$ of $L$ is a disjoint union of circle(s) embedded in the Seifert-fibred space $\SS^3\setminus \calL_{\textup{SF}}$ such that $L_{\textup{p}}$ can be constructed by the following steps: Assume all annuli intersect the disc bounded by $U$ perpendicularly and assume we are viewing the disc from above. 
\begin{enumerate}
   \item Between the two sets of annuli $\{A^x_{d_1}, \ldots, A^x_{d_{\#k}}\}$ and $\{A^y_{a_1}, \ldots, A^y_{a_{\#l}}\}$ (refer to \reffig{Step1}~ for example): 
    \begin{enumerate}
        \item \label{Item:RedLineSegments} Draw $\overline{n}$ disjoint horizontal line segments\footnote{These $y_i$-line segments become undercrossing strands after annular Dehn surgeries.} from top to bottom and label each of them from top to bottom with $y_{\tau(1)}, \ldots, y_{\tau(\overline{n})}$ respectively (or equivalently, colour them with the corresponding red intensities). 
        \item \label{Item:BlueLineSegments}  Below the set of all horizontal line segments constructed in Step (\ref{Item:RedLineSegments}), draw another set of $\overline{n}$ disjoint horizontal line segments\footnote{These $x_i$-line segments become overcrossing strands after annular Dehn surgeries.} from bottom to top and label each of them from bottom to top with $x_{\sigma(1)}, \ldots, x_{\sigma(\overline{n})}$ respectively (or equivalently, colour them with the corresponding blue intensities).  
    \end{enumerate}

    \item Horizontal extensions (refer to \reffig{Step2ab} for example):
    \begin{enumerate}
        \item Horizontally extend each (labelled or coloured) horizontal line segment and specify the over/under-crossings such that each horizontally extended $x_i$-line segment intersects each of the annuli $A^x_{d_1}, A^x_{d_2}, \ldots, A^x_{k_i}$ transversely once and each horizontally extended $y_j$-line segment intersects each of $A^y_{a_1}, A^y_{a_2}, \ldots, A^y_{l_j}$ transversely once. 
        \item Further extend each (blue) $x_i$-line segment on the other side such that the new extension lies in front of all $y$-annuli and ends at a point vertically below the right endpoint of the extended $y_i$-line segment. Similarly, further extend each (red) $y_i$-line segment on the other side such that the new extension lies behind all $x$-annuli and ends at a point vertically above the left endpoint of the extended $x_{i+1}$-line segment, where the subscript addition is the cyclic group operation\footnote{See assumptions right before \refprop{OrderWithinFullBs}.} corresponding to the word containing $x_i$. 
    \end{enumerate}

    \item\label{Step:VertLineSegments} Vertical connections (refer to \reffig{Step3} for example): Connect the $2\overline{n}$ extended line segments according to each word in (\ref{Words}). In other words, we start from the left endpoint of the extended $x_1$-line segment, trace along the line segment and connect its right endpoint to the right endpoint of the extended $y_1$-line segment by a (blue) vertical\footnote{We use the term ``vertical'' here loosely, the line segment does not need to be strictly vertical as long as the graph structure is preserved.} line segment. If the word period $n_1$ is greater than one, we connect the left endpoint of the extended $y_1$-line segment to the left endpoint of the extended $x_2$-line segment by a (red) vertical line segment. Similarly, we connect the right endpoint of the extended $x_2$-line segment to the right endpoint of the extended $y_2$-line segment by a (blue) vertical line segment. Continue the process until we reach the last labelled base $y_{n_1}$ of the first word, 
    we then connect the left endpoint of the extended $y_{n_1}$-line segment to the left endpoint of the extended $x_{1}$-line segment by a (red) vertical line segment to obtain the projection of the first link component in the disc bounded by $U$. Repeat the process for the remaining words in (\ref{Words}). 

    \item Adjustments: 
    \begin{enumerate}
        \item Adjust the horizontal extensions such that for each set of two or more vertical line segments approaching the same annulus, their corresponding labelled/coloured arcs (each of which consists of one vertical and one horizontal line segment of the same label/colour) do not intersect. 
        \item \label{Step:SpecifyOverUnder} Specify the over/under-crossings at each self-intersection of the projection of $L_{\textup{p}}$ such that any (red) $y$-arc goes under any (blue) $x$-arc, and any vertical (blue) $x$-arc goes under any horizontal (blue) $x$-arc.  
    \end{enumerate}
\end{enumerate}
\end{definition}  

Note that a parent link is a link in a Seifert-fibred space, we call such a link with exactly one link component a \emph{parent knot}. 

\begin{remark} \label{Rmk:NoCrossings}
    In Step~(\ref{Step:SpecifyOverUnder}) of \refdef{ParentLink}, we do not need to consider the over/under-crossing information among the (red) $y$-arcs. The reason is as follows: Observe that the (red) $y$-arcs will become the undercrossing strands in the split template.  We have been viewing the projection of the parent link in a way that the $y$-arcs start from the top right of the diagram and end at the bottom left of the diagram. This is consistent with the overall direction of the non-intersecting undercrossing strands in the split template. Thus there are no intersections among the (red) $y$-arcs using the above construction.  
\end{remark}

\reffig{SplitTemplateWithLabelsAnnuliKnotRing} shows an example of the parent knot $K_{\textup{p}}$ of the modular knot $K$ with word $x^{10}y^2x^5y^2x^7y^6x^2y^2x^5y^3$. 

To define the parent manifold of a modular link $L$, recall that $(a_{1}, \ldots, a_{\#l})$ is the ascending order of the $y$-exponents of all words of $L$, where $\#l$ is the number of distinct $y$-exponents of all words of $L$.

\begin{definition}[Parent manifold] \label{Def:ParentMfd}
    The \emph{parent manifold $\SS^3\setminus(\calL_{\textup{SF}}\cup \calL_{L})$ of the modular link $L$} is the complement of the link $\calL_L$ in the Seifert-fibred space $\SS^3\setminus \calL_{\textup{SF}}$ of the modular link $L$ (see \refdef{SFSpaceOfModularLink}), where $\calL_L$ consists of the following: 
    \begin{itemize}
        \item the parent link $L_{\textup{p}}$ (see \refdef{ParentLink} for construction), 
        \item the trefoil knot $T$ associated to $\mathrm{T^1}S_{\textup{mod}}$, and
        \item the unknots $V_{a_1}, \ldots, V_{a_{\#l}}$, where each $V_j$ encircles all the blue vertical line segment(s) (as defined in Step~(\ref{Step:VertLineSegments}) of \refdef{ParentLink}) with the connecting red horizontal line segment(s) that intersect the annulus $A^y_j$. 
    \end{itemize} 
\end{definition} 

\reffig{LinkCompleInSFSpaceResized} shows an example of the parent manifold of a modular knot $K$, which is a one-component modular link. 

\begin{figure}
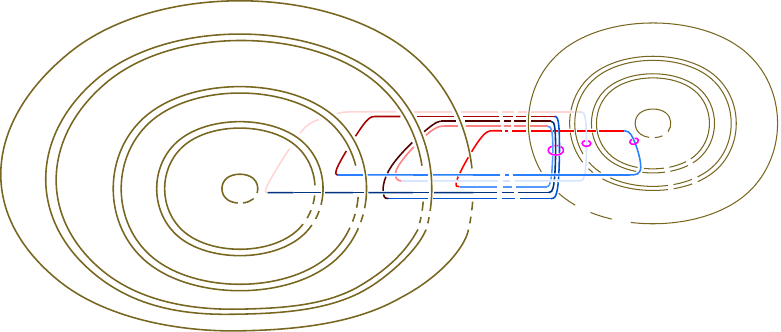
\caption{Parent manifold $\SS^3\setminus (\calL_{\textup{SF}}\cup \calL_K)$ of the modular knot $K$ with code word $x^{10}y^2x^5y^2x^7y^6x^2y^2x^5y^3$, where $\calL_{\textup{SF}}$ consists of the unknot $U$ and all annuli boundaries, and $\calL_K$ consists of the parent knot $K_{\textup{p}}$, the trefoil knot $T$, and the unknots $V_2$, $V_3$, $V_6$.   The Seifert-fibred space is $\SS^3\setminus \calL_{\textup{SF}}$ with base space equal to a $14$-punctured disc bounded by $U$. 
\label{Fig:LinkCompleInSFSpaceResized}}
\end{figure}

\begin{figure}
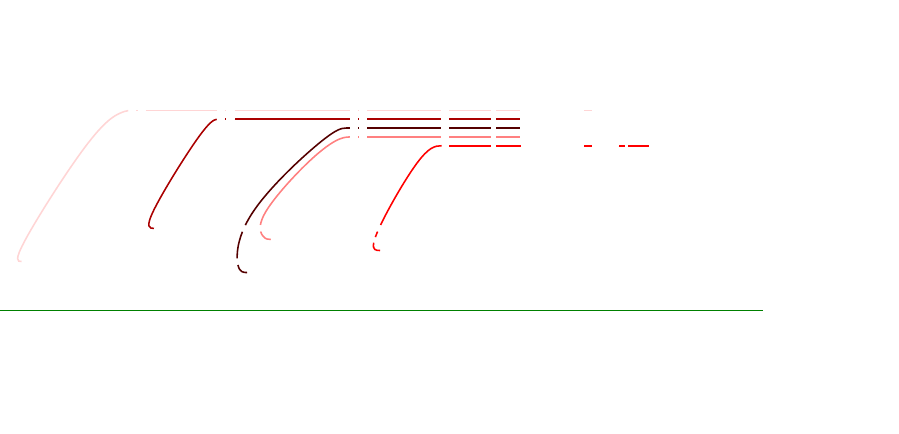
\caption{Parent knot of the modular knot with coloured word $w$, the annuli, and the unknots $V_2, V_3, V_6$.  
\label{Fig:SplitTemplateWithLabelsAnnuliKnotRing}}
\end{figure} 

\begin{proposition} \label{Prop:DehnFillParentMfd}
    Each modular link complement $\mathrm{T^1}S_{\textup{mod}}\setminus L$ can be obtained by applying Dehn fillings on its parent manifold $\SS^3\setminus(\calL_{\textup{SF}}\cup \calL_L)$.  
\end{proposition} 

\begin{proof}
Let $L$ be a modular link with labelled code words denoted in (\ref{Words}). Let $(d_{\#k}, \ldots, d_1)$ be the descending order of all $x$-exponents, where $\#k$ is the number of distinct $x$-exponents of all words of $L$.  Let $(a_{1}, \ldots, a_{\#l})$ be the ascending order of all $y$-exponents, where $\#l$ is the number of distinct $y$-exponents of all words of $L$. 

Recall from \refdef{SFSpaceOfModularLink} that the Seifert-fibred space $\SS^3\setminus \calL_{\textup{SF}}$ contains the interiors of all $x$-annuli $A^x_{d_1}, \ldots, A^x_{d_{\#k}}$ centred at the left hole of the template and the interiors of all $y$-annuli $A^y_{a_1}, \ldots, A^y_{a_{\#l}}$ centred at the right hole of the template.  All these annuli are embedded in the Lorenz template and we have integral markings $\{-k_{\mu}, \ldots, +l_{\lambda}\}$ in the branch line of the template defined by the arcs of the subwords $x_{\mu}^{k_{\mu}}$ and $y_{\lambda}^{l_{\lambda}}$ with maximal exponents. We can ambient-isotope the annuli such that the following happens (see \reffig{SplitTemplateWithLabelsAnnuliKnotRing}):  
\begin{itemize}
    \item for any $i\in[1,\#k]\cap\ZZ$, the $x$-annulus $A^x_{d_i}$ has inner boundary $\partial_{\textup{in}}A^x_{d_i}$ appearing as a vertical line at $-d_i-0.45$ in the split template and outer boundary $\partial_{\textup{out}}A^x_{d_i}$ appearing as a vertical line at $-d_{i-1}-0.55$ in the split template, with $d_0\coloneqq 0$; and 
    
    \item for any $i\in[1,\#l]\cap\ZZ$, the $y$-annulus $A^y_{a_i}$ has inner boundary $\partial_{\textup{in}}A^y_{a_i}$ appearing as a vertical line at $+a_i+0.45$ in the split template and outer boundary $\partial_{\textup{out}}A^y_{a_i}$ appearing as a vertical line at $+a_{i-1}+0.55$ in the split template, with $a_0\coloneqq 0$. 
\end{itemize}

By \refthm{UnionBunchesOfTurns}, the modular link $L$ consists of all
the $-u^{\textup{th}}$ $x$-bunches in $L$, with integer $u$ ranging from $1$ to $k_{\mu}=d_{\#k}$, and all  
the $+v^{\textup{th}}$ $y$-bunches in $L$, with integer $v$ ranging from $1$ to $l_{\lambda}=a_{\#l}$.  The same theorem also implies the following:
\begin{itemize}
    \item for any $i\in[1,\#k]\cap\ZZ$, there is/are $(d_i-d_{i-1})$ consecutive $x$-bunch(es), each of which has $(\overline{n}-\sum_{\delta=1}^{i-1}|\{k_j:k_j = d_{\delta}\}|)$ $x$-turn(s), in $L$; and 
    \item for any $i\in[1,\#l]\cap\ZZ$, there is/are $(a_i-a_{i-1})$ consecutive $y$-bunch(es), each of which has $(\overline{n}-\sum_{\alpha=1}^{i-1}|\{l_j:l_j = a_{\alpha}\}|)$ $y$-turn(s), in $L$,
\end{itemize}
where the summation notation is assumed to be zero when $i=1$. 

Hence, by applying the following annular Dehn fillings, we can obtain all bunches (possibly containing undesired full twists) in $L$. (For simplicity, we only show the resultant knot after two annular Dehn fillings in \reffig{SplitTemplateWithLabelsAnnuliKnotRing_AfterDF}.)

\begin{itemize} 
    \item For any $i\in[1,\#k]\cap\ZZ$, apply $(+\frac{1}{{d_i}-{d_{i-1}}})$-annular Dehn filling to $A^x_{d_i}$.   
\end{itemize} 

\begin{itemize} 
    \item For any $i\in[1,\#l]\cap\ZZ$, apply $(+\frac{1}{{a_i}-{a_{i-1}}})$-annular Dehn filling to $A^y_{a_i}$.   
\end{itemize} 

Let $j\in\{a_1,\ldots,a_{\#l}\}$.
If there are two or more vertical line segments of the parent knot $L_{\textup{p}}$ approaching the same $y$-annulus $A^y_j$ in the parent manifold, the knot obtained after the corresponding annular Dehn filling will contain an undesired full twist (see \reffig{SplitTemplateWithLabelsAnnuliKnotRing_AfterDF}). In such case, we apply a Dehn filling to the unknot $V_j$ along the slope $-1$ to untwist the full twist. Otherwise, we apply trivial Dehn filling (with slope = $+\frac{1}{0}$) to $V_i$. 

\begin{figure}
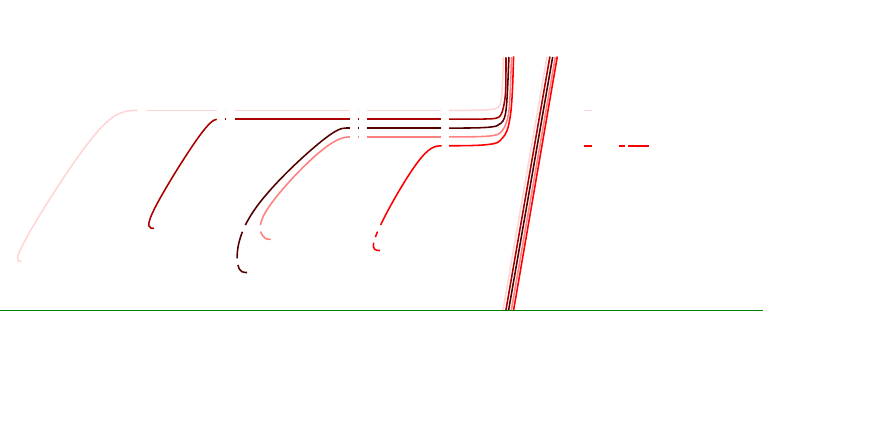
\caption{Resultant knot obtained after applying $(+\frac{1}{{10}-{7}})$-annular Dehn filling to $A^x_{10}$ and $(+\frac{1}{{2}-{0}})$-annular Dehn filling to $A^y_{2}$ for the parent manifold of the modular knot with code word $x^{10}y^2x^5y^2x^7y^6x^2y^2x^5y^3$. Note that there were three 
vertical blue line segments with corresponding red line segments intersecting the annulus $A^y_2$ before the Dehn fillings, the  $(+\frac{1}{{2}-{0}})$-annular Dehn filling to $A^y_{2}$ produced a Reidemeister~I~kink in those three parallel strands, which is actually a (undesired) full twist of the three strands. 
\label{Fig:SplitTemplateWithLabelsAnnuliKnotRing_AfterDF}}
\end{figure}

By \refdef{ParentMfd}, the aforementioned situations (i.e., two or more vertical line segments approaching the same $y$-annulus) are the only possible cases where the link obtained after all annular Dehn fillings may not be embedded in the Lorenz template. After untwisting any undesired full twists, the link obtained can be embedded in the Lorenz template.    

From the construction in \refdef{ParentMfd}, the trefoil knot $T$ and the unknot $U$ are positioned in a way such that their ambient-isotopy classes are not changed after the above annular Dehn fillings. By applying a final trivial Dehn filling to the unknot $U$, we obtain the modular link complement $\mathrm{T^1}S_{\textup{mod}}\setminus L$.  \end{proof}

\begin{remark} 
    When there is only one blue strand approaching a $y$-annulus, the annular Dehn filling will not create any undesired full twist. 
\end{remark} 

For example, \reffig{LinkCompleInSFSpaceResized} shows the parent manifold $\SS^3\setminus (\calL_{\textup{SF}}\cup \calL_K)$ of the modular knot $K$ with the labelled (or coloured) code word \[w={x_1}^{10}{y_1}^2{x_2}^5{y_2}^2{x_3}^7{y_3}^6{x_4}^2{y_4}^2{x_5}^5{y_5}^3.\]
A close-up of the parent knot $K_{\textup{p}}$, the annuli, and the unknots for removing undesired full twists after the annular Dehn fillings are illustrated in \reffig{SplitTemplateWithLabelsAnnuliKnotRing}.  

The descending order of the $x$-exponents in the example word $w$ is 
\[(d_4, d_3, d_2, d_1) = (k_{1}, k_{3}, k_{2}=k_{5}, k_{4}) = (10,7,5,2),\] 
and the ascending order of the $y$-exponents in $w$ is  
\[(a_1, a_2, a_3) = (l_{1} = l_{2} = l_{4}, l_{5}, l_{3}) = (2,3,6).\] 
By applying the following Dehn fillings, we obtain the modular knot complement $\mathrm{T^1}S_{\textup{mod}}\setminus K$. 

For the $x$-annuli: 
\begin{itemize} 
    \item Apply $(+\frac{1}{2})$-annular Dehn filling to $A^x_2$.
    \item Apply $(+\frac{1}{5-2})$-annular Dehn filling to $A^x_5$. 
    \item Apply $(+\frac{1}{7-5})$-annular Dehn filling to $A^x_7$.  
    \item Apply $(+\frac{1}{10-7})$-annular Dehn filling to $A^x_{10}$. 
\end{itemize} 

For the $y$-annuli: 
\begin{itemize} 
    \item Apply $(+\frac{1}{2})$-annular Dehn filling to $A^y_2$. 
    \item Apply $(+\frac{1}{3-2})$-annular Dehn filling to $A^y_3$. 
    \item Apply $(+\frac{1}{6-3})$-annular Dehn filling to $A^y_6$. 
\end{itemize} 

After the above annular Dehn fillings, apply a Dehn filling to the unknot $V_2$ along the slope $-1$ (see  \reffig{LinkCompleInSFSpaceResized}) to untwist a full twist (as shown in \reffig{SplitTemplateWithLabelsAnnuliKnotRing_AfterDF}). By further Dehn filling the unknots $U$, $V_3$, and $V_6$ trivially, we obtain the modular knot complement $\mathrm{T^1}S_{\textup{mod}}\setminus K$. (See \reffig{SplitTemplateWithLabelsWhole}.)

\begin{lemma} \label{Lem:SelfIntersectionsP(Lp)}
    Let $L$ be a modular link with labelled code words denoted in (\ref{Words}) with $\overline{n}$ denoting the sum of all word period(s). Let $\calP\from M\to B$ be the bundle projection map of the Seifert-fibred space $M\coloneqq \SS^3\setminus \calL_{\textup{SF}}$ of $L$. The parent link $L_{\textup{p}}$ can be ambient-isotoped to a position such that the number of self-intersections of $\calP(L_{\textup{p}})$ is at most $\frac{3}{2}\overline{n}(\overline{n}-1)$. 
\end{lemma}

\begin{proof} 
    We ambient-isotope $L_{\textup{p}}$ such that $\calP(L_{\textup{p}})$ is the same as the link diagram we built for defining $L_{\textup{p}}$ in \refdef{ParentLink}. 
    From the construction in \refdef{ParentLink}, $L_{\textup{p}}$ consists of (red) $y$-arcs and (blue) $x$-arcs only (see \reffig{SplitTemplateWithLabelsAnnuliKnotRing} for example).  Each (red) $y$-arc consists of an extended horizontal $y_i$-line segment and a vertical $y_i$-line segment approaching an $x$-annulus from behind. Similarly, each (blue) $x$-arc consists of an extended horizontal $x_i$-line segment and a vertical $x_i$-line segment approaching a $y$-annulus from the front. It suffices to count the maximum possible number of intersection(s) each vertical line segment has and sum them up. 
    
    If $\overline{n}=1$, the statement holds because there is no self-intersection in $\calP(L_{\textup{p}})$ and $\frac{3}{2}\overline{n}(\overline{n}-1) = 0$. Now suppose $\overline{n}\geq 2$. 
    
    Let $(x_{\sigma(1)}, \ldots, x_{\sigma(\overline{n})})$ and $(y_{\tau(1)}, \ldots, y_{\tau(\overline{n})})$ be orders within the full $x$-bunch and full $y$-bunch in the modular link $L$ respectively.
    
    By Remark~\ref{Rmk:NoCrossings}, there is no intersection among the (red) $y$-arcs. Thus, for (red) vertical $y$-line segment(s), we only need to consider their intersections with the (blue) horizontal $x$-line segment(s).  
    For each $m\in[2,\overline{n}]\cap\ZZ$, the (red) vertical $y_{\tau(m)}$-line segment has at most $(m-1)$ intersection(s) with the (blue) horizontal line segment(s). Thus, there are at most $1+\ldots+(\overline{n}-1) = \frac{1}{2}\overline{n}(\overline{n}-1)$ intersections coming from these (red) vertical line segments.  

    Similarly, the set of all (blue) vertical $x$-line segments has at most $\frac{1}{2}\overline{n}(\overline{n}-1)$ intersections with the (red) horizontal $y$-line segments and at most $\frac{1}{2}\overline{n}(\overline{n}-1)$ intersections with the (blue) horizontal $x$-line segments. 

    Hence, the number of self-intersections of $\calP(L_{\textup{p}})$ is at most $\frac{3}{2}\overline{n}(\overline{n}-1)$.  
\end{proof}

\begin{theorem}   \label{Thm:QuadraticUpperBound}
If $L$ is a modular link with the sum of word period(s) equal $\overline{n}$, then the hyperbolic volume of the corresponding modular link complement satisfies the following inequality: 
\[\mathrm{Vol}(\mathrm{T^1}S_{\textup{mod}} \setminus L) \leq 12 v_{\textup{tet}} (\overline{n}^2+3\overline{n}+2), \]
where $v_{\textup{tet}}\approx 1.01494$ is the volume of the regular ideal tetrahedron. 
\end{theorem} 

\begin{proof}
By \refprop{DehnFillParentMfd}, the modular link complement $\mathrm{T^1}S_{\textup{mod}} \setminus L$ can be obtained by applying Dehn fillings on its parent manifold $\SS^3\setminus(\calL_{\textup{SF}}\cup \calL_L)$. As Dehn filling decreases volume  \cite{Thurston:Geom&TopOf3Mfd}, it suffices to show the simplicial volume of the parent manifold is bounded above by the quadratic polynomial of $\overline{n}$. 

By \reflem{SelfIntersectionsP(Lp)}, the parent link $L_{\textup{p}}$ can be ambient-isotoped to a position such that the number of self-intersections of $\calP(L_{\textup{p}})$ is at most $\frac{3}{2}\overline{n}(\overline{n}-1)$. Using the same link projection, we can ambient-isotope the trefoil knot $T$ such that it accounts for $4\overline{n}+3$ self-intersections of $\calP(\calL_L)$.  The unknot(s) $V_{a_1}, \ldots, V_{a_{\#l}}$ for untwisting undesired full twist(s) account(s) for the $2\overline{n}$ self-intersections of $\calP(\calL_L)$ because there are a total of $\overline{n}$ (blue) vertical line segment(s) approaching the $y$-annuli and each unknot has intersection number twice the number of strands it encircles.  
(See \reffig{LinkCompleInSFSpaceResizedProjection} for example.)

Hence the total number of self-intersections of $\calP(\calL_L)$ is 
\[\frac{3}{2}\overline{n}(\overline{n}-1) + (4\overline{n}+3) + 2\overline{n} = \frac{3}{2}\overline{n}^2+\frac{9}{2}\overline{n}+3\]

By \refthm{SelfIntersectionNumber} (Theorem~1.5 in \cite{Cremaschi-RodriguezMigueles:HypOfLinkCpmInSFSpaces}), we have 
\[\norm{\SS^3\setminus(\calL_{\textup{SF}}\cup \calL_L)} \leq 8 \ \iota(\calP(\calL_L),\calP(\calL_L)) \leq 12 (\overline{n}^2+3\overline{n}+2).\]

As Dehn filling decreases volume \cite[Proposition~6.5.2]{Thurston:Geom&TopOf3Mfd}, it follows that
\[\mathrm{Vol}(\mathrm{T^1}S_{\textup{mod}} \setminus L)=v_{\textup{tet}}\norm{\mathrm{T^1}S_{\textup{mod}} \setminus L}\leq v_{\textup{tet}} \norm{\SS^3\setminus(\calL_{\textup{SF}}\cup \calL_L)} \]
\[\leq 12 \ v_{\textup{tet}} (\overline{n}^2+3\overline{n}+2),\] 
an upper bound that grows quadratically with the sum of word period(s). 
\end{proof} 

Since Dehn filling decreases volume \cite[Proposition~6.5.2]{Thurston:Geom&TopOf3Mfd} and each Lorenz link complement can be obtained by applying trivial Dehn filling to the trefoil cusp of the corresponding modular link complement, we have the following corollary. 

\begin{corollary}   \label{Cor:QuadraticUpperBound}
If $L$ is a hyperbolic Lorenz link with braid index equal $\overline{n}$, then we have the following inequality: 
\[\mathrm{Vol}(\SS^3 \setminus L) \leq 12 v_{\textup{tet}} (\overline{n}^2+3\overline{n}+2), \]
where $v_{\textup{tet}}\approx 1.01494$ is the volume of the regular ideal tetrahedron.
\end{corollary}

\begin{figure}
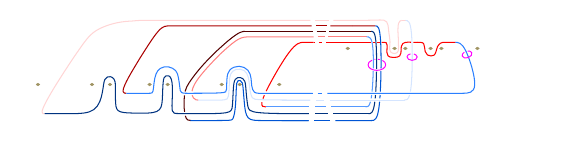
\caption{The projection of the link $\calL_K = K_{\textup{p}} \cup T \cup V_2 \cup V_3\cup V_6$ in the punctured disc bounded by the unknot $U$.  
\label{Fig:LinkCompleInSFSpaceResizedProjection}}
\end{figure}

\section{A classification of modular link complements} \label{Sec:ClassifyingModLinks}
From \refdef{ParentLink} and \refprop{DehnFillParentMfd}, we observe that the set of all modular link complements can be partitioned into classes of modular link complements that have the same parent manifold, the construction of which depends only on the base orders (see \refdef{BaseOrders}) of the modular link.  

\begin{definition}[Base orders of a modular link]  \label{Def:BaseOrders}
Let $L$ be a modular link with the collection $\mathcal{C}=\{w_1, w_2, \ldots, w_c\}$ of labelled code words denoted in (\ref{Words}). Suppose $(d_{\#k}, \ldots, d_1)$ is the descending order of the $\#k$ distinct $x$-exponents of all words of $L$ and $(a_{1}, \ldots, a_{\#l})$ is the ascending order of the $\#l$ distinct $y$-exponents of all words of $L$.  
\begin{itemize}
    \item The \emph{$x$-base order} of $\mathcal{C}$ is the set of $c$ ordered $\#k$-tuple(s) $(X^i_{d_{\#k}}, \ldots, X^i_{d_1})$ such that each set $X^i_{d_j}$ is the set of all $x$-base(s) of $w_i$ that has/have exponent(s) equal to $d_j$.  
    \item The \emph{$y$-base order} of $\mathcal{C}$ is the set of $c$ ordered $\#l$-tuple(s) $(Y^i_{a_1}, \ldots, Y^i_{a_{\#l}})$ such that each set $Y^i_{a_j}$ is the set of all $y$-base(s) of $w_i$ that has/have exponent(s) equal to $a_j$. 
\end{itemize} 

The \emph{$x$-base order of a modular link} $L$ is the $x$-base order of $\mathcal{C}$ up to any order of words and any cyclic permutations of each word. By abuse of notation, we sometimes use one representative of the collection to mean the whole collection.  Similarly, we can define the \emph{$y$-base order of a modular link}. For simplicity, we sometimes omit the outermost curly brackets when the number of link component(s) $c=1$. 
\end{definition} 

 For example, suppose $\mathcal{C} = \{w={x_1}^{10}{y_1}^2{x_2}^5{y_2}^2{x_3}^7{y_3}^6{x_4}^2{y_4}^2{x_5}^5{y_5}^3\}$, the $x$-base order of $\mathcal{C}$ is $(X_{10}, X_{7}, X_{5}, X_{2})=(\{x_1\},\{x_3\},\{x_2,x_5\},\{x_4\})$ and the $y$-base order of $\mathcal{C}$ is $(Y_{2}, Y_{3}, Y_{6}) = (\{y_1, y_2, y_4\},\{y_5\},\{y_3\})$. 

 Let us look at another example with more than one link component. Suppose $\mathcal{C}' = \{w_1, w_2, w_3\}$, where 
 \begin{align*}
     w_1 &= x_1^{100}y_1^{10}x_2^{500}y_2^{50}, \\
     w_2 &= x_3^{200}y_3^{20}x_4^{400}y_4^{40}, \text{ and } \\ 
     w_3 &= x_5^{300}y_5^{30}. 
 \end{align*}
The $x$-base order of $\mathcal{C}'$ is 
\begin{align*}
    \left\{ 
    \begin{array}{cc}
         (X^1_{500}, X^1_{400}, X^1_{300}, X^1_{200}, X^1_{100}), \\
         (X^2_{500}, X^2_{400}, X^2_{300}, X^2_{200}, X^2_{100}), \\
         (X^3_{500}, X^3_{400}, X^3_{300}, X^3_{200}, X^3_{100})
    \end{array}
   \right\}
    = \left\{
    \begin{array}{ccccc}
       (\{x_2\},\{\}, \{\}, \{\}, \{x_1\}), \\
       (\{\},\{x_4\}, \{\}, \{x_3\}, \{\}), \\
       (\{\},\{\}, \{x_5\}, \{\}, \{\})
    \end{array} 
    \right\}.  
\end{align*} 
The $y$-base order of $\mathcal{C}'$ is 
\begin{align*}
    \left\{ 
    \begin{array}{c}
         (Y^1_{50}, Y^1_{40}, Y^1_{30}, Y^1_{20}, Y^1_{10}), \\
         (Y^2_{50}, Y^2_{40}, Y^2_{30}, Y^2_{20}, Y^2_{10}), \\
         (Y^3_{50}, Y^3_{40}, Y^3_{30}, Y^3_{20}, Y^3_{10})
    \end{array}
   \right\}
    = \left\{
    \begin{array}{c}
       (\{y_2\},\{\}, \{\}, \{\}, \{y_1\}), \\
       (\{\},\{y_4\}, \{\}, \{y_3\}, \{\}), \\
       (\{\},\{\}, \{y_5\}, \{\}, \{\}) 
    \end{array} 
    \right\}.  
\end{align*} 
    
    The following proposition shows that the ambient-isotopy class of the parent link of a modular link $L$ in $\SS^3\setminus(\calL_{\textup{SF}}\cup T)$ depends only on the base orders of $L$. 

\begin{proposition} \label{Prop:SameParentLinks}
    Let $L_1$ and $L_2$ be two (possibly distinct) modular links. If the $x$-base order and $y$-base order of $L_1$ equal those of $L_2$, then the parent links of $L_1$ and $L_2$ are ambient isotopic in $\SS^3\setminus(\calL_{\textup{SF}}\cup T)$.  
\end{proposition} 

\begin{proof}
    By \refdef{ParentLink}, each parent link can be constructed uniquely up to ambient isotopy in $\SS^3\setminus(\calL_{\textup{SF}}\cup T)$ based only on the orders within full bunches (for Step~1 in \refdef{ParentLink}), the relative magnitudes (i.e., descending or ascending orders) of the $x$-exponents, and the relative magnitudes of the $y$-exponents (for Step~2 in \refdef{ParentLink}). For example, if we change the exponent of $x^1$ in \reffig{ParentKnotSteps_New} from $10$ to $987$, then all annulus label $A^x_{10}$ will be changed to $A^x_{987}$, but this does not change the orders within full bunches in Step~1 of the construction (\reffig{Step1}) and does not change the horizontal extensions in Step~2a (\reffig{Step2ab} (left)), this results in the same parent knot. Note that the construction of a parent link does not depend directly on the absolute magnitudes of the exponents nor the differences between the exponents.  

    The order within the full $x$-bunch is determined by the ascending order of the $y$-exponents and the descending order of the $x$-exponents, which in turn depend solely on the $x$-base order and $y$-base order of the modular link. Similarly, the order within the full $y$-bunch depends solely on the $x$-base order and $y$-base order of the modular link as well.  Hence, any two modular links with the same base orders share the same parent link up to ambient isotopy. 
\end{proof} 

For example, the words $x_1^2y_1x_2y_2x_3y_3$ and $x_1^3 y_1^5 x_2^2 y_2^5 x_3^2 y_3^5$ represent two different modular knots, both of them have $x$-base order equal $(\{x_1\},\{x_2, x_3\})$ and $y$-base order equal $(\{y_1, y_2, y_3\})$, so by \refprop{SameParentLinks}, their parent knots are the same (up to ambient isotopy) in $\SS^3\setminus(\calL_{\textup{SF}}\cup T)$.

We say two modular link complements belong to the same \emph{class} if they have the same base orders of their corresponding modular links.  

\begin{theorem}[Classification] \label{Thm:Classification}
   All modular link complements can be partitioned into classes. All members in each class share the same parent manifold (up to homeomorphism) and the same upper volume bound $8 \hspace{0.5mm} v_{\textup{tet}} \hspace{0.5mm} \iota(\calP(\calL),\calP(\calL))$, where $\calL$ consists of the trefoil knot $T$, the parent link $L_{\textup{p}}$ and any unknot(s) $V_i$ that encircle(s) two or more vertical line segments of $L_{\textup{p}}$ approaching the same $y$-annulus. 
\end{theorem} 

\begin{proof}
    By \refdef{BaseOrders}, each modular link $L$ with a given collection of words has a unique $x$-base order of $L$ and a unique $y$-base order of $L$. We can thus partition the set of all modular link complements such that members in each class have the same $x$-base order and the same $y$-base order. 

    By \refprop{SameParentLinks}, the parent links of the members in each class are ambient isotopic in $\SS^3\setminus(\calL_{\textup{SF}}\cup T)$. Hence, all members in each class share the same parent link $L_\textup{p}$. The resultant parent manifolds (see \refdef{ParentMfd}) are thus homeomorphic. 

    By \refthm{SelfIntersectionNumber}, the hyperbolic volume of the parent manifold is bounded above by $8 \hspace{0.5mm} v_{\textup{tet}} \hspace{0.5mm} \iota(\calP(\calL),\calP(\calL))$. Note that Dehn filling decreases volume \cite{Thurston:Geom&TopOf3Mfd}.  By \refprop{DehnFillParentMfd}, the hyperbolic volume of each member in the class is bounded above by $8 \hspace{0.5mm} v_{\textup{tet}} \hspace{0.5mm} \iota(\calP(\calL),\calP(\calL))$. 
\end{proof}

For example, the class (or family) of all modular knot complements with
\begin{itemize}
    \item $x$-base order $(\{x_1\},\{x_3\},\{x_2,x_5\},\{x_4\})$ and 
    \item $y$-base order $(\{y_1, y_2, y_4\},\{y_5\},\{y_3\})$
\end{itemize}
have parent knots and parent manifolds the same as the ones in \reffig{SplitTemplateWithLabelsAnnuliKnotRing}. Examples of code words associated with the members of the class are 
\begin{itemize}
    \item ${x_1}^{10}{y_1}^2{x_2}^5{y_2}^2{x_3}^7{y_3}^6{x_4}^2{y_4}^2{x_5}^5{y_5}^3$
    \item ${x_1}^{12}{y_1}^2{x_2}^4{y_2}^2{x_3}^8{y_3}^6{x_4}^1{y_4}^2{x_5}^4{y_5}^3$ 
    \item ${x_1}^{12}{y_1}^{222}{x_2}^4{y_2}^{222}{x_3}^8{y_3}^{1729}{x_4}^1{y_4}^{222}{x_5}^4{y_5}^{345}$
\end{itemize} 

The three code words have the same parent manifolds because the relative magnitudes of their exponents remain the same, i.e., in each of the three code words above, $x_1$ has the maximal exponent among all $x$-bases, $x_3$ has the second largest exponent, $x_2$ and $x_5$ share the third largest exponent, and $x_4$ has the smallest exponent.  

\section{Volume bounds that are linear in word periods} \label{Sec:LinearInPeriod} 

Rodr\'{\i}guez-Migueles showed in \cite[Theorem 1.4]{RodriguezMigueles:PeriodsOfContinuedFractions} that there exists a sequence of modular knot complements with words of the form $x^{k_1} y  x^{k_2} y \ldots x^{k_n} y$, where $k_1 > k_2 > \ldots > k_{n-1} > k_n + 1$, such that an upper volume bound is linear in the word period. By \refthm{Classification}, there exists a sequence of classes of modular knot complements that share the same linear upper bound. In this section, we further generalise such a family of modular knot complements with volume bounds depending linearly on the word period. 

Unless otherwise specified, from now on in this section, let $L$ be a modular link with labelled code words denoted in (\ref{Words}) with $(x_{\sigma(1)}, \ldots, x_{\sigma(\overline{n})})$ and $(y_{\tau(1)}, \ldots, y_{\tau(\overline{n})})$ denoting its orders within the full $x$-bunch and full $y$-bunch respectively. Let $\calP\from M\to B$ be the bundle projection map of the Seifert-fibred space $M\coloneqq \SS^3\setminus \calL_{\textup{SF}}$ of $L$. Note that \refdef{ParentLink} constructs a link projection of the parent link under $\calP$, and the link projection can be considered as lying in four disjoint regions. For the region containing only the $y$-horizontal and $y$-vertical line segments (red) (see \reffig{Step3} for example), we call it the \emph{top-left region}.  For the region containing $x$-horizontal (blue) and $y$-vertical line segments (red), we call it the \emph{bottom-left region}. Similarly, we call the region containing $x$-horizontal and $x$-vertical line segments (blue) the \emph{bottom-right region}, and we call the region containing $y$-horizontal (red) and $x$-vertical (blue) line segments the \emph{top-right region}. 

\begin{theorem}
    If there exists a family of modular links $\{L^{(i)}\}$ such that the number of self-intersections of the parent link projection $\calP({L^{(i)}}_{\textup{p}})$ is linear in the sum of word periods $\overline{n}$, then there exists an upper bound for the hyperbolic volumes $\mathrm{Vol}(\mathrm{T^1}S_{\textup{mod}} \setminus L^{(i)})$ such that the bound depends only on the sum of word periods $\overline{n}$ and is linear in $\overline{n}$. 
\end{theorem} 

\begin{proof}
    By \refprop{DehnFillParentMfd}, each modular link complement $\mathrm{T^1}S_{\textup{mod}}\setminus L^{(i)}$ can be obtained by applying Dehn fillings on its parent manifold $\SS^3\setminus(\calL_{\textup{SF}}\cup \calL_{L^{(i)}})$. 
     
    Note that the number of crossings between the parent link and the union of the trefoil knot $T$ and the unknots $V_i$'s is $(4\overline{n} + 3) + 2\overline{n}$. The number of self-intersections $\iota(\calP(\calL_{L^{(i)}}),\calP(\calL_{L^{(i)}}))$ of the projection of the link $\calL_{L^{(i)}}$ is 
    \[(4\overline{n} + 3) + 2\overline{n} + \iota(\calP({L^{(i)}}_\textup{p}),\calP({L^{(i)}}_\textup{p})),\]
    which is linear in the sum of word periods because of the assumption.

    As Dehn filling decreases volume \cite{Thurston:Geom&TopOf3Mfd}, it follows from \refthm{SelfIntersectionNumber} that the hyperbolic volumes $\mathrm{Vol}(\mathrm{T^1}S_{\textup{mod}} \setminus L^{(i)})$ are bounded above by a polynomial that is linear in $\overline{n}$. 
\end{proof} 

\begin{theorem} \label{Thm:LinearBound}
    Let $n, i\in\ZZ_+$.  Let $K$ be a modular knot with code word $x^{k_1} y^i  x^{k_2} y^i \ldots x^{k_n} y^i$.  If \[k_1 > k_2 > \ldots > k_{n-1} > k_n \text{ \
 or \ } k_1 < k_2 < \ldots < k_{n-1} < k_n, \] then 
 \[\mathrm{Vol}(\mathrm{T^1}S_{\textup{mod}} \setminus K) \leq 8 v_{\textup{tet}} (7n+2),\]
where $v_{\textup{tet}}\approx 1.01494$ is the volume of the regular ideal tetrahedron.
\end{theorem} 

\begin{proof} We consider two cases below. 

\textbf{Case 1:}  Suppose $k_1 > k_2 > \ldots > k_{n-1} > k_n$. By \refprop{OrderWithinFullBs}, the order within full $x$-bunch is $(x_n, x_1, x_2, \ldots, x_{n-1})$, and the order within full $y$-bunch is $(y_n, y_1, y_2, \ldots, y_{n-1})$.  By the construction method of parent link in \refdef{ParentLink}, the set of $n$ blue vertical line segments approach the same $y$-annulus. Hence, together with the adjustment step in \refdef{ParentLink}, the top-right region and the bottom-right region of the projected parent knot of $K$ do not have any self-intersections. For the bottom-left region, the rightmost red vertical line segment intersects all the other horizontal line segment(s) while the other arcs remain parallel to one another, thus the parent knot projection has $(n-1)$ self-intersections. 

\textbf{Case 2:} Suppose $k_1 < k_2 < \ldots < k_{n-1} < k_n$.  By \refprop{OrderWithinFullBs}, the order within full $x$-bunch is $(x_{n-1}, \ldots, x_2, x_1, x_n)$, and the order within full $y$-bunch is $(y_{n-1}, \ldots, y_2, y_1, y_n)$. By a similar argument as in Case 1, the top-right region and the bottom-right region of the projected parent knot of $K$ do not have any self-intersections. For the bottom-left region, the topmost blue horizontal line segment intersects all the other vertical line segment(s) while the other arcs remain parallel to one another, thus the parent knot projection has $(n-1)$ self-intersections. 

In any case, the total number of self-intersections of $\calP(\calL_K)$ is 
\[ (n-1) + (4n+3) + 2n = 7n+2,\] 
where $\calL_k$ consists of the parent knot $K_\textup{p}$, the trefoil knot $T$, and the unknot $V_i$ encircling the $n$ blue vertical line segments. 

By \refthm{SelfIntersectionNumber} (Theorem~1.5 in \cite{Cremaschi-RodriguezMigueles:HypOfLinkCpmInSFSpaces}), we have 
\[\norm{\SS^3\setminus(\calL_{\textup{SF}}\cup \calL_K)} \leq 8 \ \iota(\calP(\calL_K),\calP(\calL_K)) \leq 8 (7n+2).\]

As Dehn filling decreases volume \cite[Proposition~6.5.2]{Thurston:Geom&TopOf3Mfd}, it follows from \refprop{DehnFillParentMfd} that
\[\mathrm{Vol}(\mathrm{T^1}S_{\textup{mod}} \setminus K)=v_{\textup{tet}}\norm{\mathrm{T^1}S_{\textup{mod}} \setminus K}\leq v_{\textup{tet}} \norm{\SS^3\setminus(\calL_{\textup{SF}}\cup \calL_K)}\]  \[ \leq 8 \ v_{\textup{tet}} (7n+2),\] 
which is an upper bound that is linear in the word period. 
\end{proof}

\section{Further discussion} \label{Sec:Discussion}  
\subsection{Comparison of upper volume bounds} 
The upper bound in \cite[Section 3]{Bergeron-Pinsky-Silberman:UpperBound} is proportional to the sum of all word periods and the logarithm of the product of all word exponents.  If the product of all word exponents is $e^{\overline{n}^k}$ for some $k>2$, then such a bound grows faster than quadratically with the sum of word period(s) $\overline{n}$. 

The notation $r_1$ in \cite[Theorem 1.7]{C-F-K-N-P:VolBsForGenTwistedTorusLinks}  for each Lorenz link corresponds to the parameter denoted by $r_{k-1}$ in \cite[Theorem 1]{Birman-Kofman:NewTwistOnLorenzLinks}.  
Consider an overcrossing strand $\hat{x}$ associated with the parameter $r_{k-1}$ in \cite[Theorem 1]{Birman-Kofman:NewTwistOnLorenzLinks}. Note that $\hat{x}$ corresponds to an overcrossing strand with the second largest (negative) slope. Consider words with two consecutive $y$-exponents in their ascending order differing by at least two. By \refthm{UnionBunchesOfTurns} in this paper, there are at least two consecutive overcrossing strands ending at distinct $y$-intervals with integral labels differing by at least two. Thus, we may assume the overcrossing strand $\hat{x}$ starts from the $x$-split line and ends at the positive side of the branch line. 

Let $\hat{y}$ be the $y$-turn starting from the endpoint of $\hat{x}$.
We can cyclic-permute the word such that the first letter corresponds to the $y$-turn $\hat{y}$. As the exponent of the first base in the cyclically permuted word can be made arbitrarily large compared to those of other bases. By \refthm{UnionBunchesOfTurns}, the parameter $r_{k-1}$ can be arbitrarily large as well, depending on the size of the exponent. The upper volume bound in \refcor{QuadraticUpperBound} is independent of the word exponents and serves as a sharper volume upper bound for hyperbolic Lorenz link complements in the above case. 

\subsection{Limitation of the proof techniques}
In this paper, we use the notion of bunches, annular Dehn fillings, and upper volume bounds in terms of self-intersections to prove our main results. We suspect the best upper volume bound for all modular link complements that this technique can prove is one that depends quadratically on the braid index. The reason for this guess is as follows: a sequence of parent links with growing braid indices can be constructed such that each parent link has a half-twist; the number of intersections contributed by half-twists grows quadratically with the braid index. If one wants to show that the upper volume bound grows slower than a quadratic rate, one may need other techniques. 

\subsection{Further questions}
\refthm{QuadraticUpperBound} shows an upper volume bound which is quadratic in the braid index.  We are unsure at the moment whether a quadratic bound is sharp for all modular link complements. It is interesting to know whether there is a family of modular link complements with volumes converging (from below) to a bound that is quadratic in the braid index.  

There are sequences of closed geodesics with word periods growing to infinity and the hyperbolic volumes of the corresponding modular link complements are uniformly bounded (i.e., the upper bound is a real constant not depending on the word period) \cite[Corollary 1.2]{RodriguezMigueles:LowerBound}. We are interested in whether the techniques used in this paper can be used to find uniform bounds for families of modular link complements with growing word periods and highly repeated subwords. It is also interesting to explore whether the bunch notion and results in this paper can be generalised to the complements of the canonical lifts of periodic geodesics in hyperbolic surfaces other than the modular surface.

\bibliographystyle{amsplain}  
\bibliography{biblio_MathSciNet.bib} 
\end{document}

%% file: figures/Template.pdf_tex
\begingroup%
  \makeatletter%
  \providecommand\color[2][]{%
    \errmessage{(Inkscape) Color is used for the text in Inkscape, but the package 'color.sty' is not loaded}%
    \renewcommand\color[2][]{}%
  }%
  \providecommand\transparent[1]{%
    \errmessage{(Inkscape) Transparency is used (non-zero) for the text in Inkscape, but the package 'transparent.sty' is not loaded}%
    \renewcommand\transparent[1]{}%
  }%
  \providecommand\rotatebox[2]{#2}%
  \newcommand*\fsize{\dimexpr\f@size pt\relax}%
  \newcommand*\lineheight[1]{\fontsize{\fsize}{#1\fsize}\selectfont}%
  \ifx\svgwidth\undefined%
    \setlength{\unitlength}{141.44883994bp}%
    \ifx\svgscale\undefined%
      \relax%
    \else%
      \setlength{\unitlength}{\unitlength * \real{\svgscale}}%
    \fi%
  \else%
    \setlength{\unitlength}{\svgwidth}%
  \fi%
  \global\let\svgwidth\undefined%
  \global\let\svgscale\undefined%
  \makeatother%
  \begin{picture}(1,0.46887706)%
    \lineheight{1}%
    \setlength\tabcolsep{0pt}%
    \put(0,0){\includegraphics[width=\unitlength,page=1]{Template.pdf}}%
    \put(0.24443091,0.07011143){\color[rgb]{0.6,0.6,0.6}\makebox(0,0)[lt]{\lineheight{1.25}\smash{\begin{tabular}[t]{l}$x$\end{tabular}}}}%
    \put(0.68548454,0.07368241){\color[rgb]{0.6,0.6,0.6}\makebox(0,0)[lt]{\lineheight{1.25}\smash{\begin{tabular}[t]{l}$y$\end{tabular}}}}%
  \end{picture}%
\endgroup%

%% file: figures/SplitTemplate.pdf_tex
\begingroup%
  \makeatletter%
  \providecommand\color[2][]{%
    \errmessage{(Inkscape) Color is used for the text in Inkscape, but the package 'color.sty' is not loaded}%
    \renewcommand\color[2][]{}%
  }%
  \providecommand\transparent[1]{%
    \errmessage{(Inkscape) Transparency is used (non-zero) for the text in Inkscape, but the package 'transparent.sty' is not loaded}%
    \renewcommand\transparent[1]{}%
  }%
  \providecommand\rotatebox[2]{#2}%
  \newcommand*\fsize{\dimexpr\f@size pt\relax}%
  \newcommand*\lineheight[1]{\fontsize{\fsize}{#1\fsize}\selectfont}%
  \ifx\svgwidth\undefined%
    \setlength{\unitlength}{190.26874123bp}%
    \ifx\svgscale\undefined%
      \relax%
    \else%
      \setlength{\unitlength}{\unitlength * \real{\svgscale}}%
    \fi%
  \else%
    \setlength{\unitlength}{\svgwidth}%
  \fi%
  \global\let\svgwidth\undefined%
  \global\let\svgscale\undefined%
  \makeatother%
  \begin{picture}(1,0.60212669)%
    \lineheight{1}%
    \setlength\tabcolsep{0pt}%
    \put(0,0){\includegraphics[width=\unitlength,page=1]{SplitTemplate.pdf}}%
    \put(0.36099305,0.11529351){\color[rgb]{0,0.50196078,0}\makebox(0,0)[lt]{\lineheight{1.25}\smash{\begin{tabular}[t]{l}branch line\end{tabular}}}}%
    \put(0.19163019,0.48277251){\color[rgb]{0,0.50196078,0}\makebox(0,0)[lt]{\lineheight{1.25}\smash{\begin{tabular}[t]{l}$x$-split line\end{tabular}}}}%
    \put(0.52924814,0.48291689){\color[rgb]{0,0.50196078,0}\makebox(0,0)[lt]{\lineheight{1.25}\smash{\begin{tabular}[t]{l}$y$-split line\end{tabular}}}}%
    \put(0,0){\includegraphics[width=\unitlength,page=2]{SplitTemplate.pdf}}%
    \put(-0.00304873,0.29635698){\color[rgb]{0,0.50196078,0}\makebox(0,0)[lt]{\lineheight{1.25}\smash{\begin{tabular}[t]{l}glue\end{tabular}}}}%
    \put(0.87991376,0.29635698){\color[rgb]{0,0.50196078,0}\makebox(0,0)[lt]{\lineheight{1.25}\smash{\begin{tabular}[t]{l}glue\end{tabular}}}}%
  \end{picture}%
\endgroup%

%% file: figures/WilliamsAlg1.pdf_tex
\begingroup%
  \makeatletter%
  \providecommand\color[2][]{%
    \errmessage{(Inkscape) Color is used for the text in Inkscape, but the package 'color.sty' is not loaded}%
    \renewcommand\color[2][]{}%
  }%
  \providecommand\transparent[1]{%
    \errmessage{(Inkscape) Transparency is used (non-zero) for the text in Inkscape, but the package 'transparent.sty' is not loaded}%
    \renewcommand\transparent[1]{}%
  }%
  \providecommand\rotatebox[2]{#2}%
  \newcommand*\fsize{\dimexpr\f@size pt\relax}%
  \newcommand*\lineheight[1]{\fontsize{\fsize}{#1\fsize}\selectfont}%
  \ifx\svgwidth\undefined%
    \setlength{\unitlength}{404.18805966bp}%
    \ifx\svgscale\undefined%
      \relax%
    \else%
      \setlength{\unitlength}{\unitlength * \real{\svgscale}}%
    \fi%
  \else%
    \setlength{\unitlength}{\svgwidth}%
  \fi%
  \global\let\svgwidth\undefined%
  \global\let\svgscale\undefined%
  \makeatother%
  \begin{picture}(1,0.37789321)%
    \lineheight{1}%
    \setlength\tabcolsep{0pt}%
    \put(0,0){\includegraphics[width=\unitlength,page=1]{WilliamsAlg1.pdf}}%
    \put(0.44153847,0.00035411){\color[rgb]{0,0.50196078,0}\makebox(0,0)[lt]{\lineheight{1.25}\smash{\begin{tabular}[t]{l}branch line\end{tabular}}}}%
    \put(0.75281544,0.35989078){\color[rgb]{0,0.50196078,0}\makebox(0,0)[lt]{\lineheight{1.25}\smash{\begin{tabular}[t]{l}$y$-split line\end{tabular}}}}%
    \put(0,0){\includegraphics[width=\unitlength,page=2]{WilliamsAlg1.pdf}}%
    \put(0.26275116,0.35988818){\color[rgb]{0,0.50196078,0}\makebox(0,0)[lt]{\lineheight{1.25}\smash{\begin{tabular}[t]{l}$x$-split line\end{tabular}}}}%
    \put(0,0){\includegraphics[width=\unitlength,page=3]{WilliamsAlg1.pdf}}%
    \put(0.01105493,0.33131565){\color[rgb]{1,0,1}\makebox(0,0)[lt]{\lineheight{1.25}\smash{\begin{tabular}[t]{l}\fontsize{6pt}{1em}$1$\end{tabular}}}}%
    \put(0.0333856,0.33131565){\color[rgb]{1,0,1}\makebox(0,0)[lt]{\lineheight{1.25}\smash{\begin{tabular}[t]{l}\fontsize{6pt}{1em}$2$\end{tabular}}}}%
    \put(0.05675886,0.33131565){\color[rgb]{1,0,1}\makebox(0,0)[lt]{\lineheight{1.25}\smash{\begin{tabular}[t]{l}\fontsize{6pt}{1em}$3$\end{tabular}}}}%
    \put(0.07908952,0.33131565){\color[rgb]{1,0,1}\makebox(0,0)[lt]{\lineheight{1.25}\smash{\begin{tabular}[t]{l}\fontsize{6pt}{1em}$4$\end{tabular}}}}%
    \put(0.10142021,0.33131565){\color[rgb]{1,0,1}\makebox(0,0)[lt]{\lineheight{1.25}\smash{\begin{tabular}[t]{l}\fontsize{6pt}{1em}$5$\end{tabular}}}}%
    \put(0.125141,0.33096813){\color[rgb]{1,0,1}\makebox(0,0)[lt]{\lineheight{1.25}\smash{\begin{tabular}[t]{l}\fontsize{6pt}{1em}$6$\end{tabular}}}}%
    \put(0.14747166,0.33096813){\color[rgb]{1,0,1}\makebox(0,0)[lt]{\lineheight{1.25}\smash{\begin{tabular}[t]{l}\fontsize{6pt}{1em}$7$\end{tabular}}}}%
    \put(0.16980233,0.33096813){\color[rgb]{1,0,1}\makebox(0,0)[lt]{\lineheight{1.25}\smash{\begin{tabular}[t]{l}\fontsize{6pt}{1em}$8$\end{tabular}}}}%
    \put(0.19213299,0.33096813){\color[rgb]{1,0,1}\makebox(0,0)[lt]{\lineheight{1.25}\smash{\begin{tabular}[t]{l}\fontsize{6pt}{1em}$9$\end{tabular}}}}%
    \put(0.23546656,0.33096813){\color[rgb]{1,0,1}\makebox(0,0)[lt]{\lineheight{1.25}\smash{\begin{tabular}[t]{l}\fontsize{6pt}{1em}$11$\end{tabular}}}}%
    \put(0.25779722,0.33096813){\color[rgb]{1,0,1}\makebox(0,0)[lt]{\lineheight{1.25}\smash{\begin{tabular}[t]{l}\fontsize{6pt}{1em}$12$\end{tabular}}}}%
    \put(0.28012788,0.33096813){\color[rgb]{1,0,1}\makebox(0,0)[lt]{\lineheight{1.25}\smash{\begin{tabular}[t]{l}\fontsize{6pt}{1em}$13$\end{tabular}}}}%
    \put(0.30245855,0.33096813){\color[rgb]{1,0,1}\makebox(0,0)[lt]{\lineheight{1.25}\smash{\begin{tabular}[t]{l}\fontsize{6pt}{1em}$14$\end{tabular}}}}%
    \put(0.32478921,0.33096813){\color[rgb]{1,0,1}\makebox(0,0)[lt]{\lineheight{1.25}\smash{\begin{tabular}[t]{l}\fontsize{6pt}{1em}$15$\end{tabular}}}}%
    \put(0.34781492,0.33096813){\color[rgb]{1,0,1}\makebox(0,0)[lt]{\lineheight{1.25}\smash{\begin{tabular}[t]{l}\fontsize{6pt}{1em}$16$\end{tabular}}}}%
    \put(0.37014559,0.33096813){\color[rgb]{1,0,1}\makebox(0,0)[lt]{\lineheight{1.25}\smash{\begin{tabular}[t]{l}\fontsize{6pt}{1em}$17$\end{tabular}}}}%
    \put(0.39247625,0.33096813){\color[rgb]{1,0,1}\makebox(0,0)[lt]{\lineheight{1.25}\smash{\begin{tabular}[t]{l}\fontsize{6pt}{1em}$18$\end{tabular}}}}%
    \put(0.41619695,0.33131565){\color[rgb]{1,0,1}\makebox(0,0)[lt]{\lineheight{1.25}\smash{\begin{tabular}[t]{l}\fontsize{6pt}{1em}$19$\end{tabular}}}}%
    \put(0.21168341,0.33131565){\color[rgb]{1,0,1}\makebox(0,0)[lt]{\lineheight{1.25}\smash{\begin{tabular}[t]{l}\fontsize{6pt}{1em}$10$\end{tabular}}}}%
    \put(0.46085828,0.33131565){\color[rgb]{1,0,1}\makebox(0,0)[lt]{\lineheight{1.25}\smash{\begin{tabular}[t]{l}\fontsize{6pt}{1em}$21$\end{tabular}}}}%
    \put(0.48318895,0.33131565){\color[rgb]{1,0,1}\makebox(0,0)[lt]{\lineheight{1.25}\smash{\begin{tabular}[t]{l}\fontsize{6pt}{1em}$22$\end{tabular}}}}%
    \put(0.50551956,0.33131565){\color[rgb]{1,0,1}\makebox(0,0)[lt]{\lineheight{1.25}\smash{\begin{tabular}[t]{l}\fontsize{6pt}{1em}$23$\end{tabular}}}}%
    \put(0.52785017,0.33131565){\color[rgb]{1,0,1}\makebox(0,0)[lt]{\lineheight{1.25}\smash{\begin{tabular}[t]{l}\fontsize{6pt}{1em}$24$\end{tabular}}}}%
    \put(0.55018061,0.33131565){\color[rgb]{1,0,1}\makebox(0,0)[lt]{\lineheight{1.25}\smash{\begin{tabular}[t]{l}\fontsize{6pt}{1em}$25$\end{tabular}}}}%
    \put(0.57251096,0.33131565){\color[rgb]{1,0,1}\makebox(0,0)[lt]{\lineheight{1.25}\smash{\begin{tabular}[t]{l}\fontsize{6pt}{1em}$26$\end{tabular}}}}%
    \put(0.5948413,0.33131565){\color[rgb]{1,0,1}\makebox(0,0)[lt]{\lineheight{1.25}\smash{\begin{tabular}[t]{l}\fontsize{6pt}{1em}$27$\end{tabular}}}}%
    \put(0.61717159,0.33131565){\color[rgb]{1,0,1}\makebox(0,0)[lt]{\lineheight{1.25}\smash{\begin{tabular}[t]{l}\fontsize{6pt}{1em}$28$\end{tabular}}}}%
    \put(0.63950204,0.33131565){\color[rgb]{1,0,1}\makebox(0,0)[lt]{\lineheight{1.25}\smash{\begin{tabular}[t]{l}\fontsize{6pt}{1em}$29$\end{tabular}}}}%
    \put(0.68631015,0.33131565){\color[rgb]{1,0,1}\makebox(0,0)[lt]{\lineheight{1.25}\smash{\begin{tabular}[t]{l}\fontsize{6pt}{1em}$31$\end{tabular}}}}%
    \put(0.70968312,0.33131565){\color[rgb]{1,0,1}\makebox(0,0)[lt]{\lineheight{1.25}\smash{\begin{tabular}[t]{l}\fontsize{6pt}{1em}$32$\end{tabular}}}}%
    \put(0.73201346,0.33131565){\color[rgb]{1,0,1}\makebox(0,0)[lt]{\lineheight{1.25}\smash{\begin{tabular}[t]{l}\fontsize{6pt}{1em}$33$\end{tabular}}}}%
    \put(0.7543438,0.33131565){\color[rgb]{1,0,1}\makebox(0,0)[lt]{\lineheight{1.25}\smash{\begin{tabular}[t]{l}\fontsize{6pt}{1em}$34$\end{tabular}}}}%
    \put(0.77667414,0.33131565){\color[rgb]{1,0,1}\makebox(0,0)[lt]{\lineheight{1.25}\smash{\begin{tabular}[t]{l}\fontsize{6pt}{1em}$35$\end{tabular}}}}%
    \put(0.79900449,0.33131565){\color[rgb]{1,0,1}\makebox(0,0)[lt]{\lineheight{1.25}\smash{\begin{tabular}[t]{l}\fontsize{6pt}{1em}$36$\end{tabular}}}}%
    \put(0.82133483,0.33131565){\color[rgb]{1,0,1}\makebox(0,0)[lt]{\lineheight{1.25}\smash{\begin{tabular}[t]{l}\fontsize{6pt}{1em}$37$\end{tabular}}}}%
    \put(0.84370209,0.33131565){\color[rgb]{1,0,1}\makebox(0,0)[lt]{\lineheight{1.25}\smash{\begin{tabular}[t]{l}\fontsize{6pt}{1em}$38$\end{tabular}}}}%
    \put(0.86610632,0.33131566){\color[rgb]{1,0,1}\makebox(0,0)[lt]{\lineheight{1.25}\smash{\begin{tabular}[t]{l}\fontsize{6pt}{1em}$39$\end{tabular}}}}%
    \put(0.91337365,0.33131566){\color[rgb]{1,0,1}\makebox(0,0)[lt]{\lineheight{1.25}\smash{\begin{tabular}[t]{l}\fontsize{6pt}{1em}$41$\end{tabular}}}}%
    \put(0.93575813,0.33131566){\color[rgb]{1,0,1}\makebox(0,0)[lt]{\lineheight{1.25}\smash{\begin{tabular}[t]{l}\fontsize{6pt}{1em}$42$\end{tabular}}}}%
    \put(0.95813986,0.33131566){\color[rgb]{1,0,1}\makebox(0,0)[lt]{\lineheight{1.25}\smash{\begin{tabular}[t]{l}\fontsize{6pt}{1em}$43$\end{tabular}}}}%
    \put(0.98052865,0.33131566){\color[rgb]{1,0,1}\makebox(0,0)[lt]{\lineheight{1.25}\smash{\begin{tabular}[t]{l}\fontsize{6pt}{1em}$44$\end{tabular}}}}%
    \put(0.43852778,0.33131565){\color[rgb]{1,0,1}\makebox(0,0)[lt]{\lineheight{1.25}\smash{\begin{tabular}[t]{l}\fontsize{6pt}{1em}$20$\end{tabular}}}}%
    \put(0.89099198,0.33131565){\color[rgb]{1,0,1}\makebox(0,0)[lt]{\lineheight{1.25}\smash{\begin{tabular}[t]{l}\fontsize{6pt}{1em}$40$\end{tabular}}}}%
    \put(0.66397981,0.33131565){\color[rgb]{1,0,1}\makebox(0,0)[lt]{\lineheight{1.25}\smash{\begin{tabular}[t]{l}\fontsize{6pt}{1em}$30$\end{tabular}}}}%
    \put(0,0){\includegraphics[width=\unitlength,page=4]{WilliamsAlg1.pdf}}%
    \put(0.01105493,0.03442391){\color[rgb]{1,0,1}\makebox(0,0)[lt]{\lineheight{1.25}\smash{\begin{tabular}[t]{l}\fontsize{6pt}{1em}$1$\end{tabular}}}}%
    \put(0.0333856,0.03442391){\color[rgb]{1,0,1}\makebox(0,0)[lt]{\lineheight{1.25}\smash{\begin{tabular}[t]{l}\fontsize{6pt}{1em}$2$\end{tabular}}}}%
    \put(0.05675886,0.03442391){\color[rgb]{1,0,1}\makebox(0,0)[lt]{\lineheight{1.25}\smash{\begin{tabular}[t]{l}\fontsize{6pt}{1em}$3$\end{tabular}}}}%
    \put(0.07908952,0.03442391){\color[rgb]{1,0,1}\makebox(0,0)[lt]{\lineheight{1.25}\smash{\begin{tabular}[t]{l}\fontsize{6pt}{1em}$4$\end{tabular}}}}%
    \put(0.10142021,0.03442391){\color[rgb]{1,0,1}\makebox(0,0)[lt]{\lineheight{1.25}\smash{\begin{tabular}[t]{l}\fontsize{6pt}{1em}$5$\end{tabular}}}}%
    \put(0.125141,0.03407638){\color[rgb]{1,0,1}\makebox(0,0)[lt]{\lineheight{1.25}\smash{\begin{tabular}[t]{l}\fontsize{6pt}{1em}$6$\end{tabular}}}}%
    \put(0.14747166,0.03407638){\color[rgb]{1,0,1}\makebox(0,0)[lt]{\lineheight{1.25}\smash{\begin{tabular}[t]{l}\fontsize{6pt}{1em}$7$\end{tabular}}}}%
    \put(0.16980233,0.03407638){\color[rgb]{1,0,1}\makebox(0,0)[lt]{\lineheight{1.25}\smash{\begin{tabular}[t]{l}\fontsize{6pt}{1em}$8$\end{tabular}}}}%
    \put(0.19213299,0.03407638){\color[rgb]{1,0,1}\makebox(0,0)[lt]{\lineheight{1.25}\smash{\begin{tabular}[t]{l}\fontsize{6pt}{1em}$9$\end{tabular}}}}%
    \put(0.23546656,0.03407638){\color[rgb]{1,0,1}\makebox(0,0)[lt]{\lineheight{1.25}\smash{\begin{tabular}[t]{l}\fontsize{6pt}{1em}$11$\end{tabular}}}}%
    \put(0.25779722,0.03407638){\color[rgb]{1,0,1}\makebox(0,0)[lt]{\lineheight{1.25}\smash{\begin{tabular}[t]{l}\fontsize{6pt}{1em}$12$\end{tabular}}}}%
    \put(0.28012788,0.03407638){\color[rgb]{1,0,1}\makebox(0,0)[lt]{\lineheight{1.25}\smash{\begin{tabular}[t]{l}\fontsize{6pt}{1em}$13$\end{tabular}}}}%
    \put(0.30245855,0.03407638){\color[rgb]{1,0,1}\makebox(0,0)[lt]{\lineheight{1.25}\smash{\begin{tabular}[t]{l}\fontsize{6pt}{1em}$14$\end{tabular}}}}%
    \put(0.32478921,0.03407638){\color[rgb]{1,0,1}\makebox(0,0)[lt]{\lineheight{1.25}\smash{\begin{tabular}[t]{l}\fontsize{6pt}{1em}$15$\end{tabular}}}}%
    \put(0.34781492,0.03407638){\color[rgb]{1,0,1}\makebox(0,0)[lt]{\lineheight{1.25}\smash{\begin{tabular}[t]{l}\fontsize{6pt}{1em}$16$\end{tabular}}}}%
    \put(0.37014559,0.03407638){\color[rgb]{1,0,1}\makebox(0,0)[lt]{\lineheight{1.25}\smash{\begin{tabular}[t]{l}\fontsize{6pt}{1em}$17$\end{tabular}}}}%
    \put(0.39247625,0.03407638){\color[rgb]{1,0,1}\makebox(0,0)[lt]{\lineheight{1.25}\smash{\begin{tabular}[t]{l}\fontsize{6pt}{1em}$18$\end{tabular}}}}%
    \put(0.41619695,0.03442391){\color[rgb]{1,0,1}\makebox(0,0)[lt]{\lineheight{1.25}\smash{\begin{tabular}[t]{l}\fontsize{6pt}{1em}$19$\end{tabular}}}}%
    \put(0.21168341,0.03442391){\color[rgb]{1,0,1}\makebox(0,0)[lt]{\lineheight{1.25}\smash{\begin{tabular}[t]{l}\fontsize{6pt}{1em}$10$\end{tabular}}}}%
    \put(0.46085828,0.03442391){\color[rgb]{1,0,1}\makebox(0,0)[lt]{\lineheight{1.25}\smash{\begin{tabular}[t]{l}\fontsize{6pt}{1em}$21$\end{tabular}}}}%
    \put(0.48318895,0.03442391){\color[rgb]{1,0,1}\makebox(0,0)[lt]{\lineheight{1.25}\smash{\begin{tabular}[t]{l}\fontsize{6pt}{1em}$22$\end{tabular}}}}%
    \put(0.50551956,0.03442391){\color[rgb]{1,0,1}\makebox(0,0)[lt]{\lineheight{1.25}\smash{\begin{tabular}[t]{l}\fontsize{6pt}{1em}$23$\end{tabular}}}}%
    \put(0.52785017,0.03442391){\color[rgb]{1,0,1}\makebox(0,0)[lt]{\lineheight{1.25}\smash{\begin{tabular}[t]{l}\fontsize{6pt}{1em}$24$\end{tabular}}}}%
    \put(0.55018061,0.03442391){\color[rgb]{1,0,1}\makebox(0,0)[lt]{\lineheight{1.25}\smash{\begin{tabular}[t]{l}\fontsize{6pt}{1em}$25$\end{tabular}}}}%
    \put(0.57251096,0.03442391){\color[rgb]{1,0,1}\makebox(0,0)[lt]{\lineheight{1.25}\smash{\begin{tabular}[t]{l}\fontsize{6pt}{1em}$26$\end{tabular}}}}%
    \put(0.5948413,0.03442391){\color[rgb]{1,0,1}\makebox(0,0)[lt]{\lineheight{1.25}\smash{\begin{tabular}[t]{l}\fontsize{6pt}{1em}$27$\end{tabular}}}}%
    \put(0.61717159,0.03442391){\color[rgb]{1,0,1}\makebox(0,0)[lt]{\lineheight{1.25}\smash{\begin{tabular}[t]{l}\fontsize{6pt}{1em}$28$\end{tabular}}}}%
    \put(0.63950204,0.03442391){\color[rgb]{1,0,1}\makebox(0,0)[lt]{\lineheight{1.25}\smash{\begin{tabular}[t]{l}\fontsize{6pt}{1em}$29$\end{tabular}}}}%
    \put(0.68631015,0.03442391){\color[rgb]{1,0,1}\makebox(0,0)[lt]{\lineheight{1.25}\smash{\begin{tabular}[t]{l}\fontsize{6pt}{1em}$31$\end{tabular}}}}%
    \put(0.70968312,0.03442391){\color[rgb]{1,0,1}\makebox(0,0)[lt]{\lineheight{1.25}\smash{\begin{tabular}[t]{l}\fontsize{6pt}{1em}$32$\end{tabular}}}}%
    \put(0.73201346,0.03442391){\color[rgb]{1,0,1}\makebox(0,0)[lt]{\lineheight{1.25}\smash{\begin{tabular}[t]{l}\fontsize{6pt}{1em}$33$\end{tabular}}}}%
    \put(0.7543438,0.03442391){\color[rgb]{1,0,1}\makebox(0,0)[lt]{\lineheight{1.25}\smash{\begin{tabular}[t]{l}\fontsize{6pt}{1em}$34$\end{tabular}}}}%
    \put(0.77667414,0.03442391){\color[rgb]{1,0,1}\makebox(0,0)[lt]{\lineheight{1.25}\smash{\begin{tabular}[t]{l}\fontsize{6pt}{1em}$35$\end{tabular}}}}%
    \put(0.79900449,0.03442391){\color[rgb]{1,0,1}\makebox(0,0)[lt]{\lineheight{1.25}\smash{\begin{tabular}[t]{l}\fontsize{6pt}{1em}$36$\end{tabular}}}}%
    \put(0.82133483,0.03442391){\color[rgb]{1,0,1}\makebox(0,0)[lt]{\lineheight{1.25}\smash{\begin{tabular}[t]{l}\fontsize{6pt}{1em}$37$\end{tabular}}}}%
    \put(0.84370209,0.03442459){\color[rgb]{1,0,1}\makebox(0,0)[lt]{\lineheight{1.25}\smash{\begin{tabular}[t]{l}\fontsize{6pt}{1em}$38$\end{tabular}}}}%
    \put(0.86610633,0.0344238){\color[rgb]{1,0,1}\makebox(0,0)[lt]{\lineheight{1.25}\smash{\begin{tabular}[t]{l}\fontsize{6pt}{1em}$39$\end{tabular}}}}%
    \put(0.91337363,0.03442433){\color[rgb]{1,0,1}\makebox(0,0)[lt]{\lineheight{1.25}\smash{\begin{tabular}[t]{l}\fontsize{6pt}{1em}$41$\end{tabular}}}}%
    \put(0.93575813,0.03442405){\color[rgb]{1,0,1}\makebox(0,0)[lt]{\lineheight{1.25}\smash{\begin{tabular}[t]{l}\fontsize{6pt}{1em}$42$\end{tabular}}}}%
    \put(0.95813986,0.03442397){\color[rgb]{1,0,1}\makebox(0,0)[lt]{\lineheight{1.25}\smash{\begin{tabular}[t]{l}\fontsize{6pt}{1em}$43$\end{tabular}}}}%
    \put(0.98052865,0.03442425){\color[rgb]{1,0,1}\makebox(0,0)[lt]{\lineheight{1.25}\smash{\begin{tabular}[t]{l}\fontsize{6pt}{1em}$44$\end{tabular}}}}%
    \put(0.43852778,0.03442391){\color[rgb]{1,0,1}\makebox(0,0)[lt]{\lineheight{1.25}\smash{\begin{tabular}[t]{l}\fontsize{6pt}{1em}$20$\end{tabular}}}}%
    \put(0.89099198,0.03442473){\color[rgb]{1,0,1}\makebox(0,0)[lt]{\lineheight{1.25}\smash{\begin{tabular}[t]{l}\fontsize{6pt}{1em}$40$\end{tabular}}}}%
    \put(0.66397981,0.03442391){\color[rgb]{1,0,1}\makebox(0,0)[lt]{\lineheight{1.25}\smash{\begin{tabular}[t]{l}\fontsize{6pt}{1em}$30$\end{tabular}}}}%
  \end{picture}%
\endgroup%

%% file: figures/WilliamsAlg2.pdf_tex
\begingroup%
  \makeatletter%
  \providecommand\color[2][]{%
    \errmessage{(Inkscape) Color is used for the text in Inkscape, but the package 'color.sty' is not loaded}%
    \renewcommand\color[2][]{}%
  }%
  \providecommand\transparent[1]{%
    \errmessage{(Inkscape) Transparency is used (non-zero) for the text in Inkscape, but the package 'transparent.sty' is not loaded}%
    \renewcommand\transparent[1]{}%
  }%
  \providecommand\rotatebox[2]{#2}%
  \newcommand*\fsize{\dimexpr\f@size pt\relax}%
  \newcommand*\lineheight[1]{\fontsize{\fsize}{#1\fsize}\selectfont}%
  \ifx\svgwidth\undefined%
    \setlength{\unitlength}{404.18805966bp}%
    \ifx\svgscale\undefined%
      \relax%
    \else%
      \setlength{\unitlength}{\unitlength * \real{\svgscale}}%
    \fi%
  \else%
    \setlength{\unitlength}{\svgwidth}%
  \fi%
  \global\let\svgwidth\undefined%
  \global\let\svgscale\undefined%
  \makeatother%
  \begin{picture}(1,0.37789321)%
    \lineheight{1}%
    \setlength\tabcolsep{0pt}%
    \put(0,0){\includegraphics[width=\unitlength,page=1]{WilliamsAlg2.pdf}}%
    \put(0.44153847,0.00035411){\color[rgb]{0,0.50196078,0}\makebox(0,0)[lt]{\lineheight{1.25}\smash{\begin{tabular}[t]{l}branch line\end{tabular}}}}%
    \put(0.75281544,0.35989078){\color[rgb]{0,0.50196078,0}\makebox(0,0)[lt]{\lineheight{1.25}\smash{\begin{tabular}[t]{l}$y$-split line\end{tabular}}}}%
    \put(0,0){\includegraphics[width=\unitlength,page=2]{WilliamsAlg2.pdf}}%
    \put(0.26275116,0.35988818){\color[rgb]{0,0.50196078,0}\makebox(0,0)[lt]{\lineheight{1.25}\smash{\begin{tabular}[t]{l}$x$-split line\end{tabular}}}}%
    \put(0,0){\includegraphics[width=\unitlength,page=3]{WilliamsAlg2.pdf}}%
    \put(0.01105493,0.33131565){\color[rgb]{1,0,1}\makebox(0,0)[lt]{\lineheight{1.25}\smash{\begin{tabular}[t]{l}\fontsize{6pt}{1em}$1$\end{tabular}}}}%
    \put(0.0333856,0.33131565){\color[rgb]{1,0,1}\makebox(0,0)[lt]{\lineheight{1.25}\smash{\begin{tabular}[t]{l}\fontsize{6pt}{1em}$2$\end{tabular}}}}%
    \put(0.05675886,0.33131565){\color[rgb]{1,0,1}\makebox(0,0)[lt]{\lineheight{1.25}\smash{\begin{tabular}[t]{l}\fontsize{6pt}{1em}$3$\end{tabular}}}}%
    \put(0.07908952,0.33131565){\color[rgb]{1,0,1}\makebox(0,0)[lt]{\lineheight{1.25}\smash{\begin{tabular}[t]{l}\fontsize{6pt}{1em}$4$\end{tabular}}}}%
    \put(0.10142021,0.33131565){\color[rgb]{1,0,1}\makebox(0,0)[lt]{\lineheight{1.25}\smash{\begin{tabular}[t]{l}\fontsize{6pt}{1em}$5$\end{tabular}}}}%
    \put(0.125141,0.33096813){\color[rgb]{1,0,1}\makebox(0,0)[lt]{\lineheight{1.25}\smash{\begin{tabular}[t]{l}\fontsize{6pt}{1em}$6$\end{tabular}}}}%
    \put(0.14747166,0.33096813){\color[rgb]{1,0,1}\makebox(0,0)[lt]{\lineheight{1.25}\smash{\begin{tabular}[t]{l}\fontsize{6pt}{1em}$7$\end{tabular}}}}%
    \put(0.16980233,0.33096813){\color[rgb]{1,0,1}\makebox(0,0)[lt]{\lineheight{1.25}\smash{\begin{tabular}[t]{l}\fontsize{6pt}{1em}$8$\end{tabular}}}}%
    \put(0.19213299,0.33096813){\color[rgb]{1,0,1}\makebox(0,0)[lt]{\lineheight{1.25}\smash{\begin{tabular}[t]{l}\fontsize{6pt}{1em}$9$\end{tabular}}}}%
    \put(0.23546656,0.33096813){\color[rgb]{1,0,1}\makebox(0,0)[lt]{\lineheight{1.25}\smash{\begin{tabular}[t]{l}\fontsize{6pt}{1em}$11$\end{tabular}}}}%
    \put(0.25779722,0.33096813){\color[rgb]{1,0,1}\makebox(0,0)[lt]{\lineheight{1.25}\smash{\begin{tabular}[t]{l}\fontsize{6pt}{1em}$12$\end{tabular}}}}%
    \put(0.28012788,0.33096813){\color[rgb]{1,0,1}\makebox(0,0)[lt]{\lineheight{1.25}\smash{\begin{tabular}[t]{l}\fontsize{6pt}{1em}$13$\end{tabular}}}}%
    \put(0.30245855,0.33096813){\color[rgb]{1,0,1}\makebox(0,0)[lt]{\lineheight{1.25}\smash{\begin{tabular}[t]{l}\fontsize{6pt}{1em}$14$\end{tabular}}}}%
    \put(0.32478921,0.33096813){\color[rgb]{1,0,1}\makebox(0,0)[lt]{\lineheight{1.25}\smash{\begin{tabular}[t]{l}\fontsize{6pt}{1em}$15$\end{tabular}}}}%
    \put(0.34781492,0.33096813){\color[rgb]{1,0,1}\makebox(0,0)[lt]{\lineheight{1.25}\smash{\begin{tabular}[t]{l}\fontsize{6pt}{1em}$16$\end{tabular}}}}%
    \put(0.37014559,0.33096813){\color[rgb]{1,0,1}\makebox(0,0)[lt]{\lineheight{1.25}\smash{\begin{tabular}[t]{l}\fontsize{6pt}{1em}$17$\end{tabular}}}}%
    \put(0.39247625,0.33096813){\color[rgb]{1,0,1}\makebox(0,0)[lt]{\lineheight{1.25}\smash{\begin{tabular}[t]{l}\fontsize{6pt}{1em}$18$\end{tabular}}}}%
    \put(0.41619695,0.33131565){\color[rgb]{1,0,1}\makebox(0,0)[lt]{\lineheight{1.25}\smash{\begin{tabular}[t]{l}\fontsize{6pt}{1em}$19$\end{tabular}}}}%
    \put(0.21168341,0.33131565){\color[rgb]{1,0,1}\makebox(0,0)[lt]{\lineheight{1.25}\smash{\begin{tabular}[t]{l}\fontsize{6pt}{1em}$10$\end{tabular}}}}%
    \put(0.46085828,0.33131565){\color[rgb]{1,0,1}\makebox(0,0)[lt]{\lineheight{1.25}\smash{\begin{tabular}[t]{l}\fontsize{6pt}{1em}$21$\end{tabular}}}}%
    \put(0.48318895,0.33131565){\color[rgb]{1,0,1}\makebox(0,0)[lt]{\lineheight{1.25}\smash{\begin{tabular}[t]{l}\fontsize{6pt}{1em}$22$\end{tabular}}}}%
    \put(0.50551956,0.33131565){\color[rgb]{1,0,1}\makebox(0,0)[lt]{\lineheight{1.25}\smash{\begin{tabular}[t]{l}\fontsize{6pt}{1em}$23$\end{tabular}}}}%
    \put(0.52785017,0.33131565){\color[rgb]{1,0,1}\makebox(0,0)[lt]{\lineheight{1.25}\smash{\begin{tabular}[t]{l}\fontsize{6pt}{1em}$24$\end{tabular}}}}%
    \put(0.55018061,0.33131565){\color[rgb]{1,0,1}\makebox(0,0)[lt]{\lineheight{1.25}\smash{\begin{tabular}[t]{l}\fontsize{6pt}{1em}$25$\end{tabular}}}}%
    \put(0.57251096,0.33131565){\color[rgb]{1,0,1}\makebox(0,0)[lt]{\lineheight{1.25}\smash{\begin{tabular}[t]{l}\fontsize{6pt}{1em}$26$\end{tabular}}}}%
    \put(0.5948413,0.33131565){\color[rgb]{1,0,1}\makebox(0,0)[lt]{\lineheight{1.25}\smash{\begin{tabular}[t]{l}\fontsize{6pt}{1em}$27$\end{tabular}}}}%
    \put(0.61717159,0.33131565){\color[rgb]{1,0,1}\makebox(0,0)[lt]{\lineheight{1.25}\smash{\begin{tabular}[t]{l}\fontsize{6pt}{1em}$28$\end{tabular}}}}%
    \put(0.63950204,0.33131565){\color[rgb]{1,0,1}\makebox(0,0)[lt]{\lineheight{1.25}\smash{\begin{tabular}[t]{l}\fontsize{6pt}{1em}$29$\end{tabular}}}}%
    \put(0.68631015,0.33131565){\color[rgb]{1,0,1}\makebox(0,0)[lt]{\lineheight{1.25}\smash{\begin{tabular}[t]{l}\fontsize{6pt}{1em}$31$\end{tabular}}}}%
    \put(0.70968312,0.33131565){\color[rgb]{1,0,1}\makebox(0,0)[lt]{\lineheight{1.25}\smash{\begin{tabular}[t]{l}\fontsize{6pt}{1em}$32$\end{tabular}}}}%
    \put(0.73201346,0.33131565){\color[rgb]{1,0,1}\makebox(0,0)[lt]{\lineheight{1.25}\smash{\begin{tabular}[t]{l}\fontsize{6pt}{1em}$33$\end{tabular}}}}%
    \put(0.7543438,0.33131565){\color[rgb]{1,0,1}\makebox(0,0)[lt]{\lineheight{1.25}\smash{\begin{tabular}[t]{l}\fontsize{6pt}{1em}$34$\end{tabular}}}}%
    \put(0.77667414,0.33131565){\color[rgb]{1,0,1}\makebox(0,0)[lt]{\lineheight{1.25}\smash{\begin{tabular}[t]{l}\fontsize{6pt}{1em}$35$\end{tabular}}}}%
    \put(0.79900449,0.33131565){\color[rgb]{1,0,1}\makebox(0,0)[lt]{\lineheight{1.25}\smash{\begin{tabular}[t]{l}\fontsize{6pt}{1em}$36$\end{tabular}}}}%
    \put(0.82133483,0.33131565){\color[rgb]{1,0,1}\makebox(0,0)[lt]{\lineheight{1.25}\smash{\begin{tabular}[t]{l}\fontsize{6pt}{1em}$37$\end{tabular}}}}%
    \put(0.84370209,0.33131565){\color[rgb]{1,0,1}\makebox(0,0)[lt]{\lineheight{1.25}\smash{\begin{tabular}[t]{l}\fontsize{6pt}{1em}$38$\end{tabular}}}}%
    \put(0.86610632,0.33131566){\color[rgb]{1,0,1}\makebox(0,0)[lt]{\lineheight{1.25}\smash{\begin{tabular}[t]{l}\fontsize{6pt}{1em}$39$\end{tabular}}}}%
    \put(0.91337365,0.33131566){\color[rgb]{1,0,1}\makebox(0,0)[lt]{\lineheight{1.25}\smash{\begin{tabular}[t]{l}\fontsize{6pt}{1em}$41$\end{tabular}}}}%
    \put(0.93575813,0.33131566){\color[rgb]{1,0,1}\makebox(0,0)[lt]{\lineheight{1.25}\smash{\begin{tabular}[t]{l}\fontsize{6pt}{1em}$42$\end{tabular}}}}%
    \put(0.95813986,0.33131566){\color[rgb]{1,0,1}\makebox(0,0)[lt]{\lineheight{1.25}\smash{\begin{tabular}[t]{l}\fontsize{6pt}{1em}$43$\end{tabular}}}}%
    \put(0.98052865,0.33131566){\color[rgb]{1,0,1}\makebox(0,0)[lt]{\lineheight{1.25}\smash{\begin{tabular}[t]{l}\fontsize{6pt}{1em}$44$\end{tabular}}}}%
    \put(0.43852778,0.33131565){\color[rgb]{1,0,1}\makebox(0,0)[lt]{\lineheight{1.25}\smash{\begin{tabular}[t]{l}\fontsize{6pt}{1em}$20$\end{tabular}}}}%
    \put(0.89099198,0.33131565){\color[rgb]{1,0,1}\makebox(0,0)[lt]{\lineheight{1.25}\smash{\begin{tabular}[t]{l}\fontsize{6pt}{1em}$40$\end{tabular}}}}%
    \put(0.66397981,0.33131565){\color[rgb]{1,0,1}\makebox(0,0)[lt]{\lineheight{1.25}\smash{\begin{tabular}[t]{l}\fontsize{6pt}{1em}$30$\end{tabular}}}}%
    \put(0,0){\includegraphics[width=\unitlength,page=4]{WilliamsAlg2.pdf}}%
    \put(0.01105493,0.03442391){\color[rgb]{1,0,1}\makebox(0,0)[lt]{\lineheight{1.25}\smash{\begin{tabular}[t]{l}\fontsize{6pt}{1em}$1$\end{tabular}}}}%
    \put(0.0333856,0.03442391){\color[rgb]{1,0,1}\makebox(0,0)[lt]{\lineheight{1.25}\smash{\begin{tabular}[t]{l}\fontsize{6pt}{1em}$2$\end{tabular}}}}%
    \put(0.05675886,0.03442391){\color[rgb]{1,0,1}\makebox(0,0)[lt]{\lineheight{1.25}\smash{\begin{tabular}[t]{l}\fontsize{6pt}{1em}$3$\end{tabular}}}}%
    \put(0.07908952,0.03442391){\color[rgb]{1,0,1}\makebox(0,0)[lt]{\lineheight{1.25}\smash{\begin{tabular}[t]{l}\fontsize{6pt}{1em}$4$\end{tabular}}}}%
    \put(0.10142021,0.03442391){\color[rgb]{1,0,1}\makebox(0,0)[lt]{\lineheight{1.25}\smash{\begin{tabular}[t]{l}\fontsize{6pt}{1em}$5$\end{tabular}}}}%
    \put(0.125141,0.03407638){\color[rgb]{1,0,1}\makebox(0,0)[lt]{\lineheight{1.25}\smash{\begin{tabular}[t]{l}\fontsize{6pt}{1em}$6$\end{tabular}}}}%
    \put(0.14747166,0.03407638){\color[rgb]{1,0,1}\makebox(0,0)[lt]{\lineheight{1.25}\smash{\begin{tabular}[t]{l}\fontsize{6pt}{1em}$7$\end{tabular}}}}%
    \put(0.16980233,0.03407638){\color[rgb]{1,0,1}\makebox(0,0)[lt]{\lineheight{1.25}\smash{\begin{tabular}[t]{l}\fontsize{6pt}{1em}$8$\end{tabular}}}}%
    \put(0.19213299,0.03407638){\color[rgb]{1,0,1}\makebox(0,0)[lt]{\lineheight{1.25}\smash{\begin{tabular}[t]{l}\fontsize{6pt}{1em}$9$\end{tabular}}}}%
    \put(0.23546656,0.03407638){\color[rgb]{1,0,1}\makebox(0,0)[lt]{\lineheight{1.25}\smash{\begin{tabular}[t]{l}\fontsize{6pt}{1em}$11$\end{tabular}}}}%
    \put(0.25779722,0.03407638){\color[rgb]{1,0,1}\makebox(0,0)[lt]{\lineheight{1.25}\smash{\begin{tabular}[t]{l}\fontsize{6pt}{1em}$12$\end{tabular}}}}%
    \put(0.28012788,0.03407638){\color[rgb]{1,0,1}\makebox(0,0)[lt]{\lineheight{1.25}\smash{\begin{tabular}[t]{l}\fontsize{6pt}{1em}$13$\end{tabular}}}}%
    \put(0.30245855,0.03407638){\color[rgb]{1,0,1}\makebox(0,0)[lt]{\lineheight{1.25}\smash{\begin{tabular}[t]{l}\fontsize{6pt}{1em}$14$\end{tabular}}}}%
    \put(0.32478921,0.03407638){\color[rgb]{1,0,1}\makebox(0,0)[lt]{\lineheight{1.25}\smash{\begin{tabular}[t]{l}\fontsize{6pt}{1em}$15$\end{tabular}}}}%
    \put(0.34781492,0.03407638){\color[rgb]{1,0,1}\makebox(0,0)[lt]{\lineheight{1.25}\smash{\begin{tabular}[t]{l}\fontsize{6pt}{1em}$16$\end{tabular}}}}%
    \put(0.37014559,0.03407638){\color[rgb]{1,0,1}\makebox(0,0)[lt]{\lineheight{1.25}\smash{\begin{tabular}[t]{l}\fontsize{6pt}{1em}$17$\end{tabular}}}}%
    \put(0.39247625,0.03407638){\color[rgb]{1,0,1}\makebox(0,0)[lt]{\lineheight{1.25}\smash{\begin{tabular}[t]{l}\fontsize{6pt}{1em}$18$\end{tabular}}}}%
    \put(0.41619695,0.03442391){\color[rgb]{1,0,1}\makebox(0,0)[lt]{\lineheight{1.25}\smash{\begin{tabular}[t]{l}\fontsize{6pt}{1em}$19$\end{tabular}}}}%
    \put(0.21168341,0.03442391){\color[rgb]{1,0,1}\makebox(0,0)[lt]{\lineheight{1.25}\smash{\begin{tabular}[t]{l}\fontsize{6pt}{1em}$10$\end{tabular}}}}%
    \put(0.46085828,0.03442391){\color[rgb]{1,0,1}\makebox(0,0)[lt]{\lineheight{1.25}\smash{\begin{tabular}[t]{l}\fontsize{6pt}{1em}$21$\end{tabular}}}}%
    \put(0.48318895,0.03442391){\color[rgb]{1,0,1}\makebox(0,0)[lt]{\lineheight{1.25}\smash{\begin{tabular}[t]{l}\fontsize{6pt}{1em}$22$\end{tabular}}}}%
    \put(0.50551956,0.03442391){\color[rgb]{1,0,1}\makebox(0,0)[lt]{\lineheight{1.25}\smash{\begin{tabular}[t]{l}\fontsize{6pt}{1em}$23$\end{tabular}}}}%
    \put(0.52785017,0.03442391){\color[rgb]{1,0,1}\makebox(0,0)[lt]{\lineheight{1.25}\smash{\begin{tabular}[t]{l}\fontsize{6pt}{1em}$24$\end{tabular}}}}%
    \put(0.55018061,0.03442391){\color[rgb]{1,0,1}\makebox(0,0)[lt]{\lineheight{1.25}\smash{\begin{tabular}[t]{l}\fontsize{6pt}{1em}$25$\end{tabular}}}}%
    \put(0.57251096,0.03442391){\color[rgb]{1,0,1}\makebox(0,0)[lt]{\lineheight{1.25}\smash{\begin{tabular}[t]{l}\fontsize{6pt}{1em}$26$\end{tabular}}}}%
    \put(0.5948413,0.03442391){\color[rgb]{1,0,1}\makebox(0,0)[lt]{\lineheight{1.25}\smash{\begin{tabular}[t]{l}\fontsize{6pt}{1em}$27$\end{tabular}}}}%
    \put(0.61717159,0.03442391){\color[rgb]{1,0,1}\makebox(0,0)[lt]{\lineheight{1.25}\smash{\begin{tabular}[t]{l}\fontsize{6pt}{1em}$28$\end{tabular}}}}%
    \put(0.63950204,0.03442391){\color[rgb]{1,0,1}\makebox(0,0)[lt]{\lineheight{1.25}\smash{\begin{tabular}[t]{l}\fontsize{6pt}{1em}$29$\end{tabular}}}}%
    \put(0.68631015,0.03442391){\color[rgb]{1,0,1}\makebox(0,0)[lt]{\lineheight{1.25}\smash{\begin{tabular}[t]{l}\fontsize{6pt}{1em}$31$\end{tabular}}}}%
    \put(0.70968312,0.03442391){\color[rgb]{1,0,1}\makebox(0,0)[lt]{\lineheight{1.25}\smash{\begin{tabular}[t]{l}\fontsize{6pt}{1em}$32$\end{tabular}}}}%
    \put(0.73201346,0.03442391){\color[rgb]{1,0,1}\makebox(0,0)[lt]{\lineheight{1.25}\smash{\begin{tabular}[t]{l}\fontsize{6pt}{1em}$33$\end{tabular}}}}%
    \put(0.7543438,0.03442391){\color[rgb]{1,0,1}\makebox(0,0)[lt]{\lineheight{1.25}\smash{\begin{tabular}[t]{l}\fontsize{6pt}{1em}$34$\end{tabular}}}}%
    \put(0.77667414,0.03442391){\color[rgb]{1,0,1}\makebox(0,0)[lt]{\lineheight{1.25}\smash{\begin{tabular}[t]{l}\fontsize{6pt}{1em}$35$\end{tabular}}}}%
    \put(0.79900449,0.03442391){\color[rgb]{1,0,1}\makebox(0,0)[lt]{\lineheight{1.25}\smash{\begin{tabular}[t]{l}\fontsize{6pt}{1em}$36$\end{tabular}}}}%
    \put(0.82133483,0.03442391){\color[rgb]{1,0,1}\makebox(0,0)[lt]{\lineheight{1.25}\smash{\begin{tabular}[t]{l}\fontsize{6pt}{1em}$37$\end{tabular}}}}%
    \put(0.84370209,0.03442459){\color[rgb]{1,0,1}\makebox(0,0)[lt]{\lineheight{1.25}\smash{\begin{tabular}[t]{l}\fontsize{6pt}{1em}$38$\end{tabular}}}}%
    \put(0.86610633,0.0344238){\color[rgb]{1,0,1}\makebox(0,0)[lt]{\lineheight{1.25}\smash{\begin{tabular}[t]{l}\fontsize{6pt}{1em}$39$\end{tabular}}}}%
    \put(0.91337363,0.03442433){\color[rgb]{1,0,1}\makebox(0,0)[lt]{\lineheight{1.25}\smash{\begin{tabular}[t]{l}\fontsize{6pt}{1em}$41$\end{tabular}}}}%
    \put(0.93575813,0.03442405){\color[rgb]{1,0,1}\makebox(0,0)[lt]{\lineheight{1.25}\smash{\begin{tabular}[t]{l}\fontsize{6pt}{1em}$42$\end{tabular}}}}%
    \put(0.95813986,0.03442397){\color[rgb]{1,0,1}\makebox(0,0)[lt]{\lineheight{1.25}\smash{\begin{tabular}[t]{l}\fontsize{6pt}{1em}$43$\end{tabular}}}}%
    \put(0.98052865,0.03442425){\color[rgb]{1,0,1}\makebox(0,0)[lt]{\lineheight{1.25}\smash{\begin{tabular}[t]{l}\fontsize{6pt}{1em}$44$\end{tabular}}}}%
    \put(0.43852778,0.03442391){\color[rgb]{1,0,1}\makebox(0,0)[lt]{\lineheight{1.25}\smash{\begin{tabular}[t]{l}\fontsize{6pt}{1em}$20$\end{tabular}}}}%
    \put(0.89099198,0.03442473){\color[rgb]{1,0,1}\makebox(0,0)[lt]{\lineheight{1.25}\smash{\begin{tabular}[t]{l}\fontsize{6pt}{1em}$40$\end{tabular}}}}%
    \put(0.66397981,0.03442391){\color[rgb]{1,0,1}\makebox(0,0)[lt]{\lineheight{1.25}\smash{\begin{tabular}[t]{l}\fontsize{6pt}{1em}$30$\end{tabular}}}}%
    \put(0,0){\includegraphics[width=\unitlength,page=5]{WilliamsAlg2.pdf}}%
  \end{picture}%
\endgroup%

%% file: figures/WilliamsAlg2_coloured.pdf_tex
\begingroup%
  \makeatletter%
  \providecommand\color[2][]{%
    \errmessage{(Inkscape) Color is used for the text in Inkscape, but the package 'color.sty' is not loaded}%
    \renewcommand\color[2][]{}%
  }%
  \providecommand\transparent[1]{%
    \errmessage{(Inkscape) Transparency is used (non-zero) for the text in Inkscape, but the package 'transparent.sty' is not loaded}%
    \renewcommand\transparent[1]{}%
  }%
  \providecommand\rotatebox[2]{#2}%
  \newcommand*\fsize{\dimexpr\f@size pt\relax}%
  \newcommand*\lineheight[1]{\fontsize{\fsize}{#1\fsize}\selectfont}%
  \ifx\svgwidth\undefined%
    \setlength{\unitlength}{404.18805966bp}%
    \ifx\svgscale\undefined%
      \relax%
    \else%
      \setlength{\unitlength}{\unitlength * \real{\svgscale}}%
    \fi%
  \else%
    \setlength{\unitlength}{\svgwidth}%
  \fi%
  \global\let\svgwidth\undefined%
  \global\let\svgscale\undefined%
  \makeatother%
  \begin{picture}(1,0.41770459)%
    \lineheight{1}%
    \setlength\tabcolsep{0pt}%
    \put(0,0){\includegraphics[width=\unitlength,page=1]{WilliamsAlg2_coloured.pdf}}%
    \put(0.44153847,0.04016549){\color[rgb]{0,0.50196078,0}\makebox(0,0)[lt]{\lineheight{1.25}\smash{\begin{tabular}[t]{l}branch line\end{tabular}}}}%
    \put(0.75281544,0.39970216){\color[rgb]{0,0.50196078,0}\makebox(0,0)[lt]{\lineheight{1.25}\smash{\begin{tabular}[t]{l}$y$-split line\end{tabular}}}}%
    \put(0,0){\includegraphics[width=\unitlength,page=2]{WilliamsAlg2_coloured.pdf}}%
    \put(0.26275116,0.39969956){\color[rgb]{0,0.50196078,0}\makebox(0,0)[lt]{\lineheight{1.25}\smash{\begin{tabular}[t]{l}$x$-split line\end{tabular}}}}%
    \put(0,0){\includegraphics[width=\unitlength,page=3]{WilliamsAlg2_coloured.pdf}}%
    \put(0.29898663,0.00547491){\color[rgb]{0,0,0}\makebox(0,0)[lt]{\lineheight{1.25}\smash{\begin{tabular}[t]{l}$w=$\end{tabular}}}}%
    \put(0.40418719,0.00547491){\color[rgb]{0.33333333,0,0}\makebox(0,0)[lt]{\lineheight{1.25}\smash{\begin{tabular}[t]{l}$y^{2}$\end{tabular}}}}%
    \put(0.43925316,0.00547491){\color[rgb]{0,0.33333333,0.83137255}\makebox(0,0)[lt]{\lineheight{1.25}\smash{\begin{tabular}[t]{l}$x^{5}$\end{tabular}}}}%
    \put(0.47431909,0.00547491){\color[rgb]{0.66666667,0,0}\makebox(0,0)[lt]{\lineheight{1.25}\smash{\begin{tabular}[t]{l}$y^{2}$\end{tabular}}}}%
    \put(0.50938501,0.00547491){\color[rgb]{0.16470588,0.49803922,1}\makebox(0,0)[lt]{\lineheight{1.25}\smash{\begin{tabular}[t]{l}$x^{7}$\end{tabular}}}}%
    \put(0.54445094,0.00547491){\color[rgb]{1,0,0}\makebox(0,0)[lt]{\lineheight{1.25}\smash{\begin{tabular}[t]{l}$y^{6}$\end{tabular}}}}%
    \put(0.57951687,0.00547491){\color[rgb]{0.33333333,0.6,1}\makebox(0,0)[lt]{\lineheight{1.25}\smash{\begin{tabular}[t]{l}$x^{2}$\end{tabular}}}}%
    \put(0.61458279,0.00547491){\color[rgb]{1,0.50196078,0.50196078}\makebox(0,0)[lt]{\lineheight{1.25}\smash{\begin{tabular}[t]{l}$y^{2}$\end{tabular}}}}%
    \put(0.64964877,0.00547491){\color[rgb]{0.83529412,0.89803922,1}\makebox(0,0)[lt]{\lineheight{1.25}\smash{\begin{tabular}[t]{l}$x^{5}$\end{tabular}}}}%
    \put(0.6847147,0.00547491){\color[rgb]{1,0.83529412,0.83529412}\makebox(0,0)[lt]{\lineheight{1.25}\smash{\begin{tabular}[t]{l}$y^{3}$\end{tabular}}}}%
    \put(0.29898663,0.00547491){\color[rgb]{0,0,0}\makebox(0,0)[lt]{\lineheight{1.25}\smash{\begin{tabular}[t]{l}$w=$\end{tabular}}}}%
    \put(0.35509489,0.00547769){\color[rgb]{0,0.2,0.50196078}\makebox(0,0)[lt]{\lineheight{1.25}\smash{\begin{tabular}[t]{l}$x_1^{10}$\end{tabular}}}}%
    \put(0.40418719,0.00547491){\color[rgb]{0.33333333,0,0}\makebox(0,0)[lt]{\lineheight{1.25}\smash{\begin{tabular}[t]{l}$y_1^{2}$\end{tabular}}}}%
    \put(0.47431909,0.00547769){\color[rgb]{0.66666667,0,0}\makebox(0,0)[lt]{\lineheight{1.25}\smash{\begin{tabular}[t]{l}$y_2^{2}$\end{tabular}}}}%
    \put(0.50938501,0.00547769){\color[rgb]{0.16470588,0.49803922,1}\makebox(0,0)[lt]{\lineheight{1.25}\smash{\begin{tabular}[t]{l}$x_3^{7}$\end{tabular}}}}%
    \put(0.54445094,0.00547769){\color[rgb]{1,0,0}\makebox(0,0)[lt]{\lineheight{1.25}\smash{\begin{tabular}[t]{l}$y_3^{6}$\end{tabular}}}}%
    \put(0.57951687,0.00547769){\color[rgb]{0.33333333,0.6,1}\makebox(0,0)[lt]{\lineheight{1.25}\smash{\begin{tabular}[t]{l}$x_4^{2}$\end{tabular}}}}%
    \put(0.61458279,0.00547769){\color[rgb]{1,0.50196078,0.50196078}\makebox(0,0)[lt]{\lineheight{1.25}\smash{\begin{tabular}[t]{l}$y_4^{2}$\end{tabular}}}}%
    \put(0.64964877,0.00547769){\color[rgb]{0.83529412,0.89803922,1}\makebox(0,0)[lt]{\lineheight{1.25}\smash{\begin{tabular}[t]{l}$x_5^{5}$\end{tabular}}}}%
    \put(0.6847147,0.00547491){\color[rgb]{1,0.83529412,0.83529412}\makebox(0,0)[lt]{\lineheight{1.25}\smash{\begin{tabular}[t]{l}$y_5^{3}$\end{tabular}}}}%
    \put(0.43925316,0.00547769){\color[rgb]{0,0.33333333,0.83137255}\makebox(0,0)[lt]{\lineheight{1.25}\smash{\begin{tabular}[t]{l}$x_2^{5}$\end{tabular}}}}%
    \put(0.01105493,0.37112703){\color[rgb]{1,0,1}\makebox(0,0)[lt]{\lineheight{1.25}\smash{\begin{tabular}[t]{l}\fontsize{6pt}{1em}$1$\end{tabular}}}}%
    \put(0.0333856,0.37112703){\color[rgb]{1,0,1}\makebox(0,0)[lt]{\lineheight{1.25}\smash{\begin{tabular}[t]{l}\fontsize{6pt}{1em}$2$\end{tabular}}}}%
    \put(0.05675886,0.37112703){\color[rgb]{1,0,1}\makebox(0,0)[lt]{\lineheight{1.25}\smash{\begin{tabular}[t]{l}\fontsize{6pt}{1em}$3$\end{tabular}}}}%
    \put(0.07908952,0.37112703){\color[rgb]{1,0,1}\makebox(0,0)[lt]{\lineheight{1.25}\smash{\begin{tabular}[t]{l}\fontsize{6pt}{1em}$4$\end{tabular}}}}%
    \put(0.10142021,0.37112703){\color[rgb]{1,0,1}\makebox(0,0)[lt]{\lineheight{1.25}\smash{\begin{tabular}[t]{l}\fontsize{6pt}{1em}$5$\end{tabular}}}}%
    \put(0.125141,0.37077951){\color[rgb]{1,0,1}\makebox(0,0)[lt]{\lineheight{1.25}\smash{\begin{tabular}[t]{l}\fontsize{6pt}{1em}$6$\end{tabular}}}}%
    \put(0.14747166,0.37077951){\color[rgb]{1,0,1}\makebox(0,0)[lt]{\lineheight{1.25}\smash{\begin{tabular}[t]{l}\fontsize{6pt}{1em}$7$\end{tabular}}}}%
    \put(0.16980233,0.37077951){\color[rgb]{1,0,1}\makebox(0,0)[lt]{\lineheight{1.25}\smash{\begin{tabular}[t]{l}\fontsize{6pt}{1em}$8$\end{tabular}}}}%
    \put(0.19213299,0.37077951){\color[rgb]{1,0,1}\makebox(0,0)[lt]{\lineheight{1.25}\smash{\begin{tabular}[t]{l}\fontsize{6pt}{1em}$9$\end{tabular}}}}%
    \put(0.23546656,0.37077951){\color[rgb]{1,0,1}\makebox(0,0)[lt]{\lineheight{1.25}\smash{\begin{tabular}[t]{l}\fontsize{6pt}{1em}$11$\end{tabular}}}}%
    \put(0.25779722,0.37077951){\color[rgb]{1,0,1}\makebox(0,0)[lt]{\lineheight{1.25}\smash{\begin{tabular}[t]{l}\fontsize{6pt}{1em}$12$\end{tabular}}}}%
    \put(0.28012788,0.37077951){\color[rgb]{1,0,1}\makebox(0,0)[lt]{\lineheight{1.25}\smash{\begin{tabular}[t]{l}\fontsize{6pt}{1em}$13$\end{tabular}}}}%
    \put(0.30245855,0.37077951){\color[rgb]{1,0,1}\makebox(0,0)[lt]{\lineheight{1.25}\smash{\begin{tabular}[t]{l}\fontsize{6pt}{1em}$14$\end{tabular}}}}%
    \put(0.32478921,0.37077951){\color[rgb]{1,0,1}\makebox(0,0)[lt]{\lineheight{1.25}\smash{\begin{tabular}[t]{l}\fontsize{6pt}{1em}$15$\end{tabular}}}}%
    \put(0.34781492,0.37077951){\color[rgb]{1,0,1}\makebox(0,0)[lt]{\lineheight{1.25}\smash{\begin{tabular}[t]{l}\fontsize{6pt}{1em}$16$\end{tabular}}}}%
    \put(0.37014559,0.37077951){\color[rgb]{1,0,1}\makebox(0,0)[lt]{\lineheight{1.25}\smash{\begin{tabular}[t]{l}\fontsize{6pt}{1em}$17$\end{tabular}}}}%
    \put(0.39247625,0.37077951){\color[rgb]{1,0,1}\makebox(0,0)[lt]{\lineheight{1.25}\smash{\begin{tabular}[t]{l}\fontsize{6pt}{1em}$18$\end{tabular}}}}%
    \put(0.41619695,0.37112703){\color[rgb]{1,0,1}\makebox(0,0)[lt]{\lineheight{1.25}\smash{\begin{tabular}[t]{l}\fontsize{6pt}{1em}$19$\end{tabular}}}}%
    \put(0.21168341,0.37112703){\color[rgb]{1,0,1}\makebox(0,0)[lt]{\lineheight{1.25}\smash{\begin{tabular}[t]{l}\fontsize{6pt}{1em}$10$\end{tabular}}}}%
    \put(0.46085828,0.37112703){\color[rgb]{1,0,1}\makebox(0,0)[lt]{\lineheight{1.25}\smash{\begin{tabular}[t]{l}\fontsize{6pt}{1em}$21$\end{tabular}}}}%
    \put(0.48318895,0.37112703){\color[rgb]{1,0,1}\makebox(0,0)[lt]{\lineheight{1.25}\smash{\begin{tabular}[t]{l}\fontsize{6pt}{1em}$22$\end{tabular}}}}%
    \put(0.50551956,0.37112703){\color[rgb]{1,0,1}\makebox(0,0)[lt]{\lineheight{1.25}\smash{\begin{tabular}[t]{l}\fontsize{6pt}{1em}$23$\end{tabular}}}}%
    \put(0.52785017,0.37112703){\color[rgb]{1,0,1}\makebox(0,0)[lt]{\lineheight{1.25}\smash{\begin{tabular}[t]{l}\fontsize{6pt}{1em}$24$\end{tabular}}}}%
    \put(0.55018061,0.37112703){\color[rgb]{1,0,1}\makebox(0,0)[lt]{\lineheight{1.25}\smash{\begin{tabular}[t]{l}\fontsize{6pt}{1em}$25$\end{tabular}}}}%
    \put(0.57251096,0.37112703){\color[rgb]{1,0,1}\makebox(0,0)[lt]{\lineheight{1.25}\smash{\begin{tabular}[t]{l}\fontsize{6pt}{1em}$26$\end{tabular}}}}%
    \put(0.5948413,0.37112703){\color[rgb]{1,0,1}\makebox(0,0)[lt]{\lineheight{1.25}\smash{\begin{tabular}[t]{l}\fontsize{6pt}{1em}$27$\end{tabular}}}}%
    \put(0.61717159,0.37112703){\color[rgb]{1,0,1}\makebox(0,0)[lt]{\lineheight{1.25}\smash{\begin{tabular}[t]{l}\fontsize{6pt}{1em}$28$\end{tabular}}}}%
    \put(0.63950204,0.37112703){\color[rgb]{1,0,1}\makebox(0,0)[lt]{\lineheight{1.25}\smash{\begin{tabular}[t]{l}\fontsize{6pt}{1em}$29$\end{tabular}}}}%
    \put(0.68631015,0.37112703){\color[rgb]{1,0,1}\makebox(0,0)[lt]{\lineheight{1.25}\smash{\begin{tabular}[t]{l}\fontsize{6pt}{1em}$31$\end{tabular}}}}%
    \put(0.70968312,0.37112703){\color[rgb]{1,0,1}\makebox(0,0)[lt]{\lineheight{1.25}\smash{\begin{tabular}[t]{l}\fontsize{6pt}{1em}$32$\end{tabular}}}}%
    \put(0.73201346,0.37112703){\color[rgb]{1,0,1}\makebox(0,0)[lt]{\lineheight{1.25}\smash{\begin{tabular}[t]{l}\fontsize{6pt}{1em}$33$\end{tabular}}}}%
    \put(0.7543438,0.37112703){\color[rgb]{1,0,1}\makebox(0,0)[lt]{\lineheight{1.25}\smash{\begin{tabular}[t]{l}\fontsize{6pt}{1em}$34$\end{tabular}}}}%
    \put(0.77667414,0.37112703){\color[rgb]{1,0,1}\makebox(0,0)[lt]{\lineheight{1.25}\smash{\begin{tabular}[t]{l}\fontsize{6pt}{1em}$35$\end{tabular}}}}%
    \put(0.79900449,0.37112703){\color[rgb]{1,0,1}\makebox(0,0)[lt]{\lineheight{1.25}\smash{\begin{tabular}[t]{l}\fontsize{6pt}{1em}$36$\end{tabular}}}}%
    \put(0.82133483,0.37112703){\color[rgb]{1,0,1}\makebox(0,0)[lt]{\lineheight{1.25}\smash{\begin{tabular}[t]{l}\fontsize{6pt}{1em}$37$\end{tabular}}}}%
    \put(0.84370209,0.37112703){\color[rgb]{1,0,1}\makebox(0,0)[lt]{\lineheight{1.25}\smash{\begin{tabular}[t]{l}\fontsize{6pt}{1em}$38$\end{tabular}}}}%
    \put(0.86610632,0.37112704){\color[rgb]{1,0,1}\makebox(0,0)[lt]{\lineheight{1.25}\smash{\begin{tabular}[t]{l}\fontsize{6pt}{1em}$39$\end{tabular}}}}%
    \put(0.91337365,0.37112703){\color[rgb]{1,0,1}\makebox(0,0)[lt]{\lineheight{1.25}\smash{\begin{tabular}[t]{l}\fontsize{6pt}{1em}$41$\end{tabular}}}}%
    \put(0.93575813,0.37112703){\color[rgb]{1,0,1}\makebox(0,0)[lt]{\lineheight{1.25}\smash{\begin{tabular}[t]{l}\fontsize{6pt}{1em}$42$\end{tabular}}}}%
    \put(0.95813986,0.37112703){\color[rgb]{1,0,1}\makebox(0,0)[lt]{\lineheight{1.25}\smash{\begin{tabular}[t]{l}\fontsize{6pt}{1em}$43$\end{tabular}}}}%
    \put(0.98052865,0.37112704){\color[rgb]{1,0,1}\makebox(0,0)[lt]{\lineheight{1.25}\smash{\begin{tabular}[t]{l}\fontsize{6pt}{1em}$44$\end{tabular}}}}%
    \put(0.43852778,0.37112703){\color[rgb]{1,0,1}\makebox(0,0)[lt]{\lineheight{1.25}\smash{\begin{tabular}[t]{l}\fontsize{6pt}{1em}$20$\end{tabular}}}}%
    \put(0.89099198,0.37112703){\color[rgb]{1,0,1}\makebox(0,0)[lt]{\lineheight{1.25}\smash{\begin{tabular}[t]{l}\fontsize{6pt}{1em}$40$\end{tabular}}}}%
    \put(0.66397981,0.37112703){\color[rgb]{1,0,1}\makebox(0,0)[lt]{\lineheight{1.25}\smash{\begin{tabular}[t]{l}\fontsize{6pt}{1em}$30$\end{tabular}}}}%
    \put(0,0){\includegraphics[width=\unitlength,page=4]{WilliamsAlg2_coloured.pdf}}%
    \put(0.01105493,0.07423529){\color[rgb]{1,0,1}\makebox(0,0)[lt]{\lineheight{1.25}\smash{\begin{tabular}[t]{l}\fontsize{6pt}{1em}$1$\end{tabular}}}}%
    \put(0.0333856,0.07423529){\color[rgb]{1,0,1}\makebox(0,0)[lt]{\lineheight{1.25}\smash{\begin{tabular}[t]{l}\fontsize{6pt}{1em}$2$\end{tabular}}}}%
    \put(0.05675886,0.07423529){\color[rgb]{1,0,1}\makebox(0,0)[lt]{\lineheight{1.25}\smash{\begin{tabular}[t]{l}\fontsize{6pt}{1em}$3$\end{tabular}}}}%
    \put(0.07908952,0.07423529){\color[rgb]{1,0,1}\makebox(0,0)[lt]{\lineheight{1.25}\smash{\begin{tabular}[t]{l}\fontsize{6pt}{1em}$4$\end{tabular}}}}%
    \put(0.10142021,0.07423529){\color[rgb]{1,0,1}\makebox(0,0)[lt]{\lineheight{1.25}\smash{\begin{tabular}[t]{l}\fontsize{6pt}{1em}$5$\end{tabular}}}}%
    \put(0.125141,0.07388776){\color[rgb]{1,0,1}\makebox(0,0)[lt]{\lineheight{1.25}\smash{\begin{tabular}[t]{l}\fontsize{6pt}{1em}$6$\end{tabular}}}}%
    \put(0.14747166,0.07388776){\color[rgb]{1,0,1}\makebox(0,0)[lt]{\lineheight{1.25}\smash{\begin{tabular}[t]{l}\fontsize{6pt}{1em}$7$\end{tabular}}}}%
    \put(0.16980233,0.07388776){\color[rgb]{1,0,1}\makebox(0,0)[lt]{\lineheight{1.25}\smash{\begin{tabular}[t]{l}\fontsize{6pt}{1em}$8$\end{tabular}}}}%
    \put(0.19213299,0.07388776){\color[rgb]{1,0,1}\makebox(0,0)[lt]{\lineheight{1.25}\smash{\begin{tabular}[t]{l}\fontsize{6pt}{1em}$9$\end{tabular}}}}%
    \put(0.23546656,0.07388776){\color[rgb]{1,0,1}\makebox(0,0)[lt]{\lineheight{1.25}\smash{\begin{tabular}[t]{l}\fontsize{6pt}{1em}$11$\end{tabular}}}}%
    \put(0.25779722,0.07388776){\color[rgb]{1,0,1}\makebox(0,0)[lt]{\lineheight{1.25}\smash{\begin{tabular}[t]{l}\fontsize{6pt}{1em}$12$\end{tabular}}}}%
    \put(0.28012788,0.07388776){\color[rgb]{1,0,1}\makebox(0,0)[lt]{\lineheight{1.25}\smash{\begin{tabular}[t]{l}\fontsize{6pt}{1em}$13$\end{tabular}}}}%
    \put(0.30245855,0.07388776){\color[rgb]{1,0,1}\makebox(0,0)[lt]{\lineheight{1.25}\smash{\begin{tabular}[t]{l}\fontsize{6pt}{1em}$14$\end{tabular}}}}%
    \put(0.32478921,0.07388776){\color[rgb]{1,0,1}\makebox(0,0)[lt]{\lineheight{1.25}\smash{\begin{tabular}[t]{l}\fontsize{6pt}{1em}$15$\end{tabular}}}}%
    \put(0.34781492,0.07388776){\color[rgb]{1,0,1}\makebox(0,0)[lt]{\lineheight{1.25}\smash{\begin{tabular}[t]{l}\fontsize{6pt}{1em}$16$\end{tabular}}}}%
    \put(0.37014559,0.07388776){\color[rgb]{1,0,1}\makebox(0,0)[lt]{\lineheight{1.25}\smash{\begin{tabular}[t]{l}\fontsize{6pt}{1em}$17$\end{tabular}}}}%
    \put(0.39247625,0.07388776){\color[rgb]{1,0,1}\makebox(0,0)[lt]{\lineheight{1.25}\smash{\begin{tabular}[t]{l}\fontsize{6pt}{1em}$18$\end{tabular}}}}%
    \put(0.41619695,0.07423529){\color[rgb]{1,0,1}\makebox(0,0)[lt]{\lineheight{1.25}\smash{\begin{tabular}[t]{l}\fontsize{6pt}{1em}$19$\end{tabular}}}}%
    \put(0.21168341,0.07423529){\color[rgb]{1,0,1}\makebox(0,0)[lt]{\lineheight{1.25}\smash{\begin{tabular}[t]{l}\fontsize{6pt}{1em}$10$\end{tabular}}}}%
    \put(0.46085828,0.07423529){\color[rgb]{1,0,1}\makebox(0,0)[lt]{\lineheight{1.25}\smash{\begin{tabular}[t]{l}\fontsize{6pt}{1em}$21$\end{tabular}}}}%
    \put(0.48318895,0.07423529){\color[rgb]{1,0,1}\makebox(0,0)[lt]{\lineheight{1.25}\smash{\begin{tabular}[t]{l}\fontsize{6pt}{1em}$22$\end{tabular}}}}%
    \put(0.50551956,0.07423529){\color[rgb]{1,0,1}\makebox(0,0)[lt]{\lineheight{1.25}\smash{\begin{tabular}[t]{l}\fontsize{6pt}{1em}$23$\end{tabular}}}}%
    \put(0.52785017,0.07423529){\color[rgb]{1,0,1}\makebox(0,0)[lt]{\lineheight{1.25}\smash{\begin{tabular}[t]{l}\fontsize{6pt}{1em}$24$\end{tabular}}}}%
    \put(0.55018061,0.07423529){\color[rgb]{1,0,1}\makebox(0,0)[lt]{\lineheight{1.25}\smash{\begin{tabular}[t]{l}\fontsize{6pt}{1em}$25$\end{tabular}}}}%
    \put(0.57251096,0.07423529){\color[rgb]{1,0,1}\makebox(0,0)[lt]{\lineheight{1.25}\smash{\begin{tabular}[t]{l}\fontsize{6pt}{1em}$26$\end{tabular}}}}%
    \put(0.5948413,0.07423529){\color[rgb]{1,0,1}\makebox(0,0)[lt]{\lineheight{1.25}\smash{\begin{tabular}[t]{l}\fontsize{6pt}{1em}$27$\end{tabular}}}}%
    \put(0.61717159,0.07423529){\color[rgb]{1,0,1}\makebox(0,0)[lt]{\lineheight{1.25}\smash{\begin{tabular}[t]{l}\fontsize{6pt}{1em}$28$\end{tabular}}}}%
    \put(0.63950204,0.07423529){\color[rgb]{1,0,1}\makebox(0,0)[lt]{\lineheight{1.25}\smash{\begin{tabular}[t]{l}\fontsize{6pt}{1em}$29$\end{tabular}}}}%
    \put(0.68631015,0.07423529){\color[rgb]{1,0,1}\makebox(0,0)[lt]{\lineheight{1.25}\smash{\begin{tabular}[t]{l}\fontsize{6pt}{1em}$31$\end{tabular}}}}%
    \put(0.70968312,0.07423529){\color[rgb]{1,0,1}\makebox(0,0)[lt]{\lineheight{1.25}\smash{\begin{tabular}[t]{l}\fontsize{6pt}{1em}$32$\end{tabular}}}}%
    \put(0.73201346,0.07423529){\color[rgb]{1,0,1}\makebox(0,0)[lt]{\lineheight{1.25}\smash{\begin{tabular}[t]{l}\fontsize{6pt}{1em}$33$\end{tabular}}}}%
    \put(0.7543438,0.07423529){\color[rgb]{1,0,1}\makebox(0,0)[lt]{\lineheight{1.25}\smash{\begin{tabular}[t]{l}\fontsize{6pt}{1em}$34$\end{tabular}}}}%
    \put(0.77667414,0.07423529){\color[rgb]{1,0,1}\makebox(0,0)[lt]{\lineheight{1.25}\smash{\begin{tabular}[t]{l}\fontsize{6pt}{1em}$35$\end{tabular}}}}%
    \put(0.79900449,0.07423529){\color[rgb]{1,0,1}\makebox(0,0)[lt]{\lineheight{1.25}\smash{\begin{tabular}[t]{l}\fontsize{6pt}{1em}$36$\end{tabular}}}}%
    \put(0.82133483,0.07423529){\color[rgb]{1,0,1}\makebox(0,0)[lt]{\lineheight{1.25}\smash{\begin{tabular}[t]{l}\fontsize{6pt}{1em}$37$\end{tabular}}}}%
    \put(0.84370209,0.07423597){\color[rgb]{1,0,1}\makebox(0,0)[lt]{\lineheight{1.25}\smash{\begin{tabular}[t]{l}\fontsize{6pt}{1em}$38$\end{tabular}}}}%
    \put(0.86610633,0.07423518){\color[rgb]{1,0,1}\makebox(0,0)[lt]{\lineheight{1.25}\smash{\begin{tabular}[t]{l}\fontsize{6pt}{1em}$39$\end{tabular}}}}%
    \put(0.91337363,0.07423571){\color[rgb]{1,0,1}\makebox(0,0)[lt]{\lineheight{1.25}\smash{\begin{tabular}[t]{l}\fontsize{6pt}{1em}$41$\end{tabular}}}}%
    \put(0.93575813,0.07423543){\color[rgb]{1,0,1}\makebox(0,0)[lt]{\lineheight{1.25}\smash{\begin{tabular}[t]{l}\fontsize{6pt}{1em}$42$\end{tabular}}}}%
    \put(0.95813986,0.07423535){\color[rgb]{1,0,1}\makebox(0,0)[lt]{\lineheight{1.25}\smash{\begin{tabular}[t]{l}\fontsize{6pt}{1em}$43$\end{tabular}}}}%
    \put(0.98052865,0.07423563){\color[rgb]{1,0,1}\makebox(0,0)[lt]{\lineheight{1.25}\smash{\begin{tabular}[t]{l}\fontsize{6pt}{1em}$44$\end{tabular}}}}%
    \put(0.43852778,0.07423529){\color[rgb]{1,0,1}\makebox(0,0)[lt]{\lineheight{1.25}\smash{\begin{tabular}[t]{l}\fontsize{6pt}{1em}$20$\end{tabular}}}}%
    \put(0.89099198,0.07423611){\color[rgb]{1,0,1}\makebox(0,0)[lt]{\lineheight{1.25}\smash{\begin{tabular}[t]{l}\fontsize{6pt}{1em}$40$\end{tabular}}}}%
    \put(0.66397981,0.07423529){\color[rgb]{1,0,1}\makebox(0,0)[lt]{\lineheight{1.25}\smash{\begin{tabular}[t]{l}\fontsize{6pt}{1em}$30$\end{tabular}}}}%
    \put(0,0){\includegraphics[width=\unitlength,page=5]{WilliamsAlg2_coloured.pdf}}%
  \end{picture}%
\endgroup%

%% file: figures/SplitTemplateWithLabels.pdf_tex
\begingroup%
  \makeatletter%
  \providecommand\color[2][]{%
    \errmessage{(Inkscape) Color is used for the text in Inkscape, but the package 'color.sty' is not loaded}%
    \renewcommand\color[2][]{}%
  }%
  \providecommand\transparent[1]{%
    \errmessage{(Inkscape) Transparency is used (non-zero) for the text in Inkscape, but the package 'transparent.sty' is not loaded}%
    \renewcommand\transparent[1]{}%
  }%
  \providecommand\rotatebox[2]{#2}%
  \newcommand*\fsize{\dimexpr\f@size pt\relax}%
  \newcommand*\lineheight[1]{\fontsize{\fsize}{#1\fsize}\selectfont}%
  \ifx\svgwidth\undefined%
    \setlength{\unitlength}{223.90890236bp}%
    \ifx\svgscale\undefined%
      \relax%
    \else%
      \setlength{\unitlength}{\unitlength * \real{\svgscale}}%
    \fi%
  \else%
    \setlength{\unitlength}{\svgwidth}%
  \fi%
  \global\let\svgwidth\undefined%
  \global\let\svgscale\undefined%
  \makeatother%
  \begin{picture}(1,0.45415764)%
    \lineheight{1}%
    \setlength\tabcolsep{0pt}%
    \put(0,0){\includegraphics[width=\unitlength,page=1]{SplitTemplateWithLabels.pdf}}%
    \put(0.45015252,0.00604572){\color[rgb]{0,0.50196078,0}\makebox(0,0)[lt]{\lineheight{1.25}\smash{\begin{tabular}[t]{l}\fontsize{9pt}{1em}\selectfont branch line\end{tabular}}}}%
    \put(0.72493599,0.43103129){\color[rgb]{0,0.50196078,0}\makebox(0,0)[lt]{\lineheight{1.25}\smash{\begin{tabular}[t]{l}\fontsize{9pt}{1em}\selectfont $y$-split line\end{tabular}}}}%
    \put(0.60872093,0.04164152){\color[rgb]{0.4,0.4,0.4}\makebox(0,0)[lt]{\lineheight{1.25}\smash{\begin{tabular}[t]{l}\fontsize{6pt}{1em}$0$\end{tabular}}}}%
    \put(0.65705844,0.041638){\color[rgb]{0.4,0.4,0.4}\makebox(0,0)[lt]{\lineheight{1.25}\smash{\begin{tabular}[t]{l}\fontsize{6pt}{1em}$+1$\end{tabular}}}}%
    \put(0,0){\includegraphics[width=\unitlength,page=2]{SplitTemplateWithLabels.pdf}}%
    \put(0.54089309,0.041638){\color[rgb]{0.4,0.4,0.4}\makebox(0,0)[lt]{\lineheight{1.25}\smash{\begin{tabular}[t]{l}\fontsize{6pt}{1em}$-1$\end{tabular}}}}%
    \put(0.48281046,0.041638){\color[rgb]{0.4,0.4,0.4}\makebox(0,0)[lt]{\lineheight{1.25}\smash{\begin{tabular}[t]{l}\fontsize{6pt}{1em}$-2$\end{tabular}}}}%
    \put(0.42472783,0.041638){\color[rgb]{0.4,0.4,0.4}\makebox(0,0)[lt]{\lineheight{1.25}\smash{\begin{tabular}[t]{l}\fontsize{6pt}{1em}$-3$\end{tabular}}}}%
    \put(0.36664511,0.041638){\color[rgb]{0.4,0.4,0.4}\makebox(0,0)[lt]{\lineheight{1.25}\smash{\begin{tabular}[t]{l}\fontsize{6pt}{1em}$-4$\end{tabular}}}}%
    \put(0.30856239,0.041638){\color[rgb]{0.4,0.4,0.4}\makebox(0,0)[lt]{\lineheight{1.25}\smash{\begin{tabular}[t]{l}\fontsize{6pt}{1em}$-5$\end{tabular}}}}%
    \put(0.25047966,0.041638){\color[rgb]{0.4,0.4,0.4}\makebox(0,0)[lt]{\lineheight{1.25}\smash{\begin{tabular}[t]{l}\fontsize{6pt}{1em}$-6$\end{tabular}}}}%
    \put(0.19239694,0.041638){\color[rgb]{0.4,0.4,0.4}\makebox(0,0)[lt]{\lineheight{1.25}\smash{\begin{tabular}[t]{l}\fontsize{6pt}{1em}$-7$\end{tabular}}}}%
    \put(0.13431426,0.041638){\color[rgb]{0.4,0.4,0.4}\makebox(0,0)[lt]{\lineheight{1.25}\smash{\begin{tabular}[t]{l}\fontsize{6pt}{1em}$-8$\end{tabular}}}}%
    \put(0.07623154,0.041638){\color[rgb]{0.4,0.4,0.4}\makebox(0,0)[lt]{\lineheight{1.25}\smash{\begin{tabular}[t]{l}\fontsize{6pt}{1em}$-9$\end{tabular}}}}%
    \put(0.00585463,0.041638){\color[rgb]{0.4,0.4,0.4}\makebox(0,0)[lt]{\lineheight{1.25}\smash{\begin{tabular}[t]{l}\fontsize{6pt}{1em}$-10$\end{tabular}}}}%
    \put(0.21794872,0.43102796){\color[rgb]{0,0.50196078,0}\makebox(0,0)[lt]{\lineheight{1.25}\smash{\begin{tabular}[t]{l}\fontsize{9pt}{1em}\selectfont $x$-split line\end{tabular}}}}%
    \put(0,0){\includegraphics[width=\unitlength,page=3]{SplitTemplateWithLabels.pdf}}%
    \put(0.94818682,0.04163797){\color[rgb]{0.4,0.4,0.4}\makebox(0,0)[lt]{\lineheight{1.25}\smash{\begin{tabular}[t]{l}\fontsize{6pt}{1em}$+6$\end{tabular}}}}%
    \put(0.89000312,0.041638){\color[rgb]{0.4,0.4,0.4}\makebox(0,0)[lt]{\lineheight{1.25}\smash{\begin{tabular}[t]{l}\fontsize{6pt}{1em}$+5$\end{tabular}}}}%
    \put(0.83173415,0.041638){\color[rgb]{0.4,0.4,0.4}\makebox(0,0)[lt]{\lineheight{1.25}\smash{\begin{tabular}[t]{l}\fontsize{6pt}{1em}$+4$\end{tabular}}}}%
    \put(0.77349339,0.04163804){\color[rgb]{0.4,0.4,0.4}\makebox(0,0)[lt]{\lineheight{1.25}\smash{\begin{tabular}[t]{l}\fontsize{6pt}{1em}$+3$\end{tabular}}}}%
    \put(0.71529171,0.04163798){\color[rgb]{0.4,0.4,0.4}\makebox(0,0)[lt]{\lineheight{1.25}\smash{\begin{tabular}[t]{l}\fontsize{6pt}{1em}$+2$\end{tabular}}}}%
    \put(0,0){\includegraphics[width=\unitlength,page=4]{SplitTemplateWithLabels.pdf}}%
  \end{picture}%
\endgroup%

%% file: figures/TemplateWithx10y6.pdf_tex
\begingroup%
  \makeatletter%
  \providecommand\color[2][]{%
    \errmessage{(Inkscape) Color is used for the text in Inkscape, but the package 'color.sty' is not loaded}%
    \renewcommand\color[2][]{}%
  }%
  \providecommand\transparent[1]{%
    \errmessage{(Inkscape) Transparency is used (non-zero) for the text in Inkscape, but the package 'transparent.sty' is not loaded}%
    \renewcommand\transparent[1]{}%
  }%
  \providecommand\rotatebox[2]{#2}%
  \newcommand*\fsize{\dimexpr\f@size pt\relax}%
  \newcommand*\lineheight[1]{\fontsize{\fsize}{#1\fsize}\selectfont}%
  \ifx\svgwidth\undefined%
    \setlength{\unitlength}{239.41924719bp}%
    \ifx\svgscale\undefined%
      \relax%
    \else%
      \setlength{\unitlength}{\unitlength * \real{\svgscale}}%
    \fi%
  \else%
    \setlength{\unitlength}{\svgwidth}%
  \fi%
  \global\let\svgwidth\undefined%
  \global\let\svgscale\undefined%
  \makeatother%
  \begin{picture}(1,0.46887009)%
    \lineheight{1}%
    \setlength\tabcolsep{0pt}%
    \put(0,0){\includegraphics[width=\unitlength,page=1]{TemplateWithx10y6.pdf}}%
  \end{picture}%
\endgroup%

%% file: figures/SplitTemplateWithLabelsTurns.pdf_tex
\begingroup%
  \makeatletter%
  \providecommand\color[2][]{%
    \errmessage{(Inkscape) Color is used for the text in Inkscape, but the package 'color.sty' is not loaded}%
    \renewcommand\color[2][]{}%
  }%
  \providecommand\transparent[1]{%
    \errmessage{(Inkscape) Transparency is used (non-zero) for the text in Inkscape, but the package 'transparent.sty' is not loaded}%
    \renewcommand\transparent[1]{}%
  }%
  \providecommand\rotatebox[2]{#2}%
  \newcommand*\fsize{\dimexpr\f@size pt\relax}%
  \newcommand*\lineheight[1]{\fontsize{\fsize}{#1\fsize}\selectfont}%
  \ifx\svgwidth\undefined%
    \setlength{\unitlength}{218.62911742bp}%
    \ifx\svgscale\undefined%
      \relax%
    \else%
      \setlength{\unitlength}{\unitlength * \real{\svgscale}}%
    \fi%
  \else%
    \setlength{\unitlength}{\svgwidth}%
  \fi%
  \global\let\svgwidth\undefined%
  \global\let\svgscale\undefined%
  \makeatother%
  \begin{picture}(1,0.45349718)%
    \lineheight{1}%
    \setlength\tabcolsep{0pt}%
    \put(0,0){\includegraphics[width=\unitlength,page=1]{SplitTemplateWithLabelsTurns.pdf}}%
    \put(0.44949408,0.00603692){\color[rgb]{0,0.50196078,0}\makebox(0,0)[lt]{\lineheight{1.25}\smash{\begin{tabular}[t]{l}\fontsize{9pt}{1em}\selectfont branch line\end{tabular}}}}%
    \put(0.72388381,0.43040448){\color[rgb]{0,0.50196078,0}\makebox(0,0)[lt]{\lineheight{1.25}\smash{\begin{tabular}[t]{l}\fontsize{9pt}{1em}\selectfont $y$-split line\end{tabular}}}}%
    \put(0.60795494,0.04158103){\color[rgb]{0.70196078,0.70196078,0.70196078}\makebox(0,0)[lt]{\lineheight{1.25}\smash{\begin{tabular}[t]{l}\fontsize{6pt}{1em}\selectfont $0$\end{tabular}}}}%
    \put(0.65635738,0.04157745){\color[rgb]{0.70196078,0.70196078,0.70196078}\makebox(0,0)[lt]{\lineheight{1.25}\smash{\begin{tabular}[t]{l}\fontsize{6pt}{1em}\selectfont $+1$\end{tabular}}}}%
    \put(0.71446039,0.04157745){\color[rgb]{0.70196078,0.70196078,0.70196078}\makebox(0,0)[lt]{\lineheight{1.25}\smash{\begin{tabular}[t]{l}\fontsize{6pt}{1em}\selectfont $+2$\end{tabular}}}}%
    \put(0,0){\includegraphics[width=\unitlength,page=2]{SplitTemplateWithLabelsTurns.pdf}}%
    \put(0.772552,0.04157744){\color[rgb]{0.70196078,0.70196078,0.70196078}\makebox(0,0)[lt]{\lineheight{1.25}\smash{\begin{tabular}[t]{l}\fontsize{6pt}{1em}\selectfont $+3$\end{tabular}}}}%
    \put(0.83066571,0.04157745){\color[rgb]{0.70196078,0.70196078,0.70196078}\makebox(0,0)[lt]{\lineheight{1.25}\smash{\begin{tabular}[t]{l}\fontsize{6pt}{1em}\selectfont $+4$\end{tabular}}}}%
    \put(0.88876395,0.04157743){\color[rgb]{0.70196078,0.70196078,0.70196078}\makebox(0,0)[lt]{\lineheight{1.25}\smash{\begin{tabular}[t]{l}\fontsize{6pt}{1em}\selectfont $+5$\end{tabular}}}}%
    \put(0.94684267,0.04157746){\color[rgb]{0.70196078,0.70196078,0.70196078}\makebox(0,0)[lt]{\lineheight{1.25}\smash{\begin{tabular}[t]{l}\fontsize{6pt}{1em}\selectfont $+6$\end{tabular}}}}%
    \put(0.54010643,0.04157743){\color[rgb]{0.70196078,0.70196078,0.70196078}\makebox(0,0)[lt]{\lineheight{1.25}\smash{\begin{tabular}[t]{l}\fontsize{6pt}{1em}\selectfont $-1$\end{tabular}}}}%
    \put(0.48210829,0.04157743){\color[rgb]{0.70196078,0.70196078,0.70196078}\makebox(0,0)[lt]{\lineheight{1.25}\smash{\begin{tabular}[t]{l}\fontsize{6pt}{1em}\selectfont $-2$\end{tabular}}}}%
    \put(0.42411015,0.04157743){\color[rgb]{0.70196078,0.70196078,0.70196078}\makebox(0,0)[lt]{\lineheight{1.25}\smash{\begin{tabular}[t]{l}\fontsize{6pt}{1em}\selectfont $-3$\end{tabular}}}}%
    \put(0.36611181,0.04157743){\color[rgb]{0.70196078,0.70196078,0.70196078}\makebox(0,0)[lt]{\lineheight{1.25}\smash{\begin{tabular}[t]{l}\fontsize{6pt}{1em}\selectfont $-4$\end{tabular}}}}%
    \put(0.30811367,0.04157743){\color[rgb]{0.70196078,0.70196078,0.70196078}\makebox(0,0)[lt]{\lineheight{1.25}\smash{\begin{tabular}[t]{l}\fontsize{6pt}{1em}\selectfont $-5$\end{tabular}}}}%
    \put(0.25011538,0.04157743){\color[rgb]{0.70196078,0.70196078,0.70196078}\makebox(0,0)[lt]{\lineheight{1.25}\smash{\begin{tabular}[t]{l}\fontsize{6pt}{1em}\selectfont $-6$\end{tabular}}}}%
    \put(0.19211714,0.04157743){\color[rgb]{0.70196078,0.70196078,0.70196078}\makebox(0,0)[lt]{\lineheight{1.25}\smash{\begin{tabular}[t]{l}\fontsize{6pt}{1em}\selectfont $-7$\end{tabular}}}}%
    \put(0.13411895,0.04157743){\color[rgb]{0.70196078,0.70196078,0.70196078}\makebox(0,0)[lt]{\lineheight{1.25}\smash{\begin{tabular}[t]{l}\fontsize{6pt}{1em}\selectfont $-8$\end{tabular}}}}%
    \put(0.07612066,0.04157743){\color[rgb]{0.70196078,0.70196078,0.70196078}\makebox(0,0)[lt]{\lineheight{1.25}\smash{\begin{tabular}[t]{l}\fontsize{6pt}{1em}\selectfont $-9$\end{tabular}}}}%
    \put(0.00584614,0.04157743){\color[rgb]{0.70196078,0.70196078,0.70196078}\makebox(0,0)[lt]{\lineheight{1.25}\smash{\begin{tabular}[t]{l}\fontsize{6pt}{1em}\selectfont $-10$\end{tabular}}}}%
    \put(0.21763173,0.43040114){\color[rgb]{0,0.50196078,0}\makebox(0,0)[lt]{\lineheight{1.25}\smash{\begin{tabular}[t]{l}\fontsize{9pt}{1em}\selectfont $x$-split line\end{tabular}}}}%
    \put(0,0){\includegraphics[width=\unitlength,page=3]{SplitTemplateWithLabelsTurns.pdf}}%
  \end{picture}%
\endgroup%

%% file: figures/SplitTemplateWithLabelsFullBs.pdf_tex
\begingroup%
  \makeatletter%
  \providecommand\color[2][]{%
    \errmessage{(Inkscape) Color is used for the text in Inkscape, but the package 'color.sty' is not loaded}%
    \renewcommand\color[2][]{}%
  }%
  \providecommand\transparent[1]{%
    \errmessage{(Inkscape) Transparency is used (non-zero) for the text in Inkscape, but the package 'transparent.sty' is not loaded}%
    \renewcommand\transparent[1]{}%
  }%
  \providecommand\rotatebox[2]{#2}%
  \newcommand*\fsize{\dimexpr\f@size pt\relax}%
  \newcommand*\lineheight[1]{\fontsize{\fsize}{#1\fsize}\selectfont}%
  \ifx\svgwidth\undefined%
    \setlength{\unitlength}{370.70674668bp}%
    \ifx\svgscale\undefined%
      \relax%
    \else%
      \setlength{\unitlength}{\unitlength * \real{\svgscale}}%
    \fi%
  \else%
    \setlength{\unitlength}{\svgwidth}%
  \fi%
  \global\let\svgwidth\undefined%
  \global\let\svgscale\undefined%
  \makeatother%
  \begin{picture}(1,0.48902989)%
    \lineheight{1}%
    \setlength\tabcolsep{0pt}%
    \put(0,0){\includegraphics[width=\unitlength,page=1]{SplitTemplateWithLabelsFullBs.pdf}}%
    \put(0.444292,0.04657485){\color[rgb]{0,0.50196078,0}\makebox(0,0)[lt]{\lineheight{1.25}\smash{\begin{tabular}[t]{l}branch line\end{tabular}}}}%
    \put(0.71516782,0.46619552){\color[rgb]{0,0.50196078,0}\makebox(0,0)[lt]{\lineheight{1.25}\smash{\begin{tabular}[t]{l}$y$-split line\end{tabular}}}}%
    \put(0.60103654,0.08172131){\color[rgb]{0.70196078,0.70196078,0.70196078}\makebox(0,0)[lt]{\lineheight{1.25}\smash{\begin{tabular}[t]{l}$0$\end{tabular}}}}%
    \put(0.6487639,0.08171781){\color[rgb]{0.70196078,0.70196078,0.70196078}\makebox(0,0)[lt]{\lineheight{1.25}\smash{\begin{tabular}[t]{l}$+1$\end{tabular}}}}%
    \put(0.70611374,0.08171781){\color[rgb]{0.70196078,0.70196078,0.70196078}\makebox(0,0)[lt]{\lineheight{1.25}\smash{\begin{tabular}[t]{l}$+2$\end{tabular}}}}%
    \put(0,0){\includegraphics[width=\unitlength,page=2]{SplitTemplateWithLabelsFullBs.pdf}}%
    \put(0.76346323,0.08171781){\color[rgb]{0.70196078,0.70196078,0.70196078}\makebox(0,0)[lt]{\lineheight{1.25}\smash{\begin{tabular}[t]{l}$+3$\end{tabular}}}}%
    \put(0.82081272,0.08171781){\color[rgb]{0.70196078,0.70196078,0.70196078}\makebox(0,0)[lt]{\lineheight{1.25}\smash{\begin{tabular}[t]{l}$+4$\end{tabular}}}}%
    \put(0.87816221,0.08171781){\color[rgb]{0.70196078,0.70196078,0.70196078}\makebox(0,0)[lt]{\lineheight{1.25}\smash{\begin{tabular}[t]{l}$+5$\end{tabular}}}}%
    \put(0.93551169,0.08171781){\color[rgb]{0.70196078,0.70196078,0.70196078}\makebox(0,0)[lt]{\lineheight{1.25}\smash{\begin{tabular}[t]{l}$+6$\end{tabular}}}}%
    \put(0.53406492,0.08171781){\color[rgb]{0.70196078,0.70196078,0.70196078}\makebox(0,0)[lt]{\lineheight{1.25}\smash{\begin{tabular}[t]{l}$-1$\end{tabular}}}}%
    \put(0.47671555,0.08171781){\color[rgb]{0.70196078,0.70196078,0.70196078}\makebox(0,0)[lt]{\lineheight{1.25}\smash{\begin{tabular}[t]{l}$-2$\end{tabular}}}}%
    \put(0.41936612,0.08171781){\color[rgb]{0.70196078,0.70196078,0.70196078}\makebox(0,0)[lt]{\lineheight{1.25}\smash{\begin{tabular}[t]{l}$-3$\end{tabular}}}}%
    \put(0.36201663,0.08171781){\color[rgb]{0.70196078,0.70196078,0.70196078}\makebox(0,0)[lt]{\lineheight{1.25}\smash{\begin{tabular}[t]{l}$-4$\end{tabular}}}}%
    \put(0.30466714,0.08171781){\color[rgb]{0.70196078,0.70196078,0.70196078}\makebox(0,0)[lt]{\lineheight{1.25}\smash{\begin{tabular}[t]{l}$-5$\end{tabular}}}}%
    \put(0.24731765,0.08171781){\color[rgb]{0.70196078,0.70196078,0.70196078}\makebox(0,0)[lt]{\lineheight{1.25}\smash{\begin{tabular}[t]{l}$-6$\end{tabular}}}}%
    \put(0.18996816,0.08171781){\color[rgb]{0.70196078,0.70196078,0.70196078}\makebox(0,0)[lt]{\lineheight{1.25}\smash{\begin{tabular}[t]{l}$-7$\end{tabular}}}}%
    \put(0.13261867,0.08171781){\color[rgb]{0.70196078,0.70196078,0.70196078}\makebox(0,0)[lt]{\lineheight{1.25}\smash{\begin{tabular}[t]{l}$-8$\end{tabular}}}}%
    \put(0.07526918,0.08171781){\color[rgb]{0.70196078,0.70196078,0.70196078}\makebox(0,0)[lt]{\lineheight{1.25}\smash{\begin{tabular}[t]{l}$-9$\end{tabular}}}}%
    \put(0.00578071,0.08171781){\color[rgb]{0.70196078,0.70196078,0.70196078}\makebox(0,0)[lt]{\lineheight{1.25}\smash{\begin{tabular}[t]{l}$-10$\end{tabular}}}}%
    \put(0.21519735,0.46619223){\color[rgb]{0,0.50196078,0}\makebox(0,0)[lt]{\lineheight{1.25}\smash{\begin{tabular}[t]{l}$x$-split line\end{tabular}}}}%
    \put(0,0){\includegraphics[width=\unitlength,page=3]{SplitTemplateWithLabelsFullBs.pdf}}%
    \put(0.26158,0.00596939){\color[rgb]{0,0,0}\makebox(0,0)[lt]{\lineheight{1.25}\smash{\begin{tabular}[t]{l}$w=$\end{tabular}}}}%
    \put(0.33039997,0.00596939){\color[rgb]{0,0.2,0.50196078}\makebox(0,0)[lt]{\lineheight{1.25}\smash{\begin{tabular}[t]{l}$x_1^{10}$\end{tabular}}}}%
    \put(0.37628201,0.00596939){\color[rgb]{0.33333333,0,0}\makebox(0,0)[lt]{\lineheight{1.25}\smash{\begin{tabular}[t]{l}$y_1^{2}$\end{tabular}}}}%
    \put(0.414515,0.00596939){\color[rgb]{0,0.33333333,0.83137255}\makebox(0,0)[lt]{\lineheight{1.25}\smash{\begin{tabular}[t]{l}$x_2^{5}$\end{tabular}}}}%
    \put(0.452748,0.00596939){\color[rgb]{0.66666667,0,0}\makebox(0,0)[lt]{\lineheight{1.25}\smash{\begin{tabular}[t]{l}$y_2^{2}$\end{tabular}}}}%
    \put(0.49098105,0.00596939){\color[rgb]{0.16470588,0.49803922,1}\makebox(0,0)[lt]{\lineheight{1.25}\smash{\begin{tabular}[t]{l}$x_3^{7}$\end{tabular}}}}%
    \put(0.52921404,0.00596939){\color[rgb]{1,0,0}\makebox(0,0)[lt]{\lineheight{1.25}\smash{\begin{tabular}[t]{l}$y_3^{6}$\end{tabular}}}}%
    \put(0.56744703,0.00596939){\color[rgb]{0.33333333,0.6,1}\makebox(0,0)[lt]{\lineheight{1.25}\smash{\begin{tabular}[t]{l}$x_4^{2}$\end{tabular}}}}%
    \put(0.60568003,0.00596939){\color[rgb]{1,0.50196078,0.50196078}\makebox(0,0)[lt]{\lineheight{1.25}\smash{\begin{tabular}[t]{l}$y_4^{2}$\end{tabular}}}}%
    \put(0.64391302,0.00596939){\color[rgb]{0.83529412,0.89803922,1}\makebox(0,0)[lt]{\lineheight{1.25}\smash{\begin{tabular}[t]{l}$x_5^{5}$\end{tabular}}}}%
    \put(0.68214601,0.00596939){\color[rgb]{1,0.83529412,0.83529412}\makebox(0,0)[lt]{\lineheight{1.25}\smash{\begin{tabular}[t]{l}$y_5^{3}$\end{tabular}}}}%
    \put(0,0){\includegraphics[width=\unitlength,page=4]{SplitTemplateWithLabelsFullBs.pdf}}%
  \end{picture}%
\endgroup%

%% file: figures/SplitTemplateWithLabelsWhole.pdf_tex
\begingroup%
  \makeatletter%
  \providecommand\color[2][]{%
    \errmessage{(Inkscape) Color is used for the text in Inkscape, but the package 'color.sty' is not loaded}%
    \renewcommand\color[2][]{}%
  }%
  \providecommand\transparent[1]{%
    \errmessage{(Inkscape) Transparency is used (non-zero) for the text in Inkscape, but the package 'transparent.sty' is not loaded}%
    \renewcommand\transparent[1]{}%
  }%
  \providecommand\rotatebox[2]{#2}%
  \newcommand*\fsize{\dimexpr\f@size pt\relax}%
  \newcommand*\lineheight[1]{\fontsize{\fsize}{#1\fsize}\selectfont}%
  \ifx\svgwidth\undefined%
    \setlength{\unitlength}{370.70674668bp}%
    \ifx\svgscale\undefined%
      \relax%
    \else%
      \setlength{\unitlength}{\unitlength * \real{\svgscale}}%
    \fi%
  \else%
    \setlength{\unitlength}{\svgwidth}%
  \fi%
  \global\let\svgwidth\undefined%
  \global\let\svgscale\undefined%
  \makeatother%
  \begin{picture}(1,0.49042297)%
    \lineheight{1}%
    \setlength\tabcolsep{0pt}%
    \put(0,0){\includegraphics[width=\unitlength,page=1]{SplitTemplateWithLabelsWhole.pdf}}%
    \put(0.444292,0.04796793){\color[rgb]{0,0.50196078,0}\makebox(0,0)[lt]{\lineheight{1.25}\smash{\begin{tabular}[t]{l}branch line\end{tabular}}}}%
    \put(0.71516782,0.4675886){\color[rgb]{0,0.50196078,0}\makebox(0,0)[lt]{\lineheight{1.25}\smash{\begin{tabular}[t]{l}$y$-split line\end{tabular}}}}%
    \put(0.60103654,0.08311438){\color[rgb]{0.70196078,0.70196078,0.70196078}\makebox(0,0)[lt]{\lineheight{1.25}\smash{\begin{tabular}[t]{l}$0$\end{tabular}}}}%
    \put(0.6487639,0.08311088){\color[rgb]{0.70196078,0.70196078,0.70196078}\makebox(0,0)[lt]{\lineheight{1.25}\smash{\begin{tabular}[t]{l}$+1$\end{tabular}}}}%
    \put(0.70611374,0.08311088){\color[rgb]{0.70196078,0.70196078,0.70196078}\makebox(0,0)[lt]{\lineheight{1.25}\smash{\begin{tabular}[t]{l}$+2$\end{tabular}}}}%
    \put(0,0){\includegraphics[width=\unitlength,page=2]{SplitTemplateWithLabelsWhole.pdf}}%
    \put(0.76346323,0.08311088){\color[rgb]{0.70196078,0.70196078,0.70196078}\makebox(0,0)[lt]{\lineheight{1.25}\smash{\begin{tabular}[t]{l}$+3$\end{tabular}}}}%
    \put(0.82081272,0.08311088){\color[rgb]{0.70196078,0.70196078,0.70196078}\makebox(0,0)[lt]{\lineheight{1.25}\smash{\begin{tabular}[t]{l}$+4$\end{tabular}}}}%
    \put(0.87816221,0.08311088){\color[rgb]{0.70196078,0.70196078,0.70196078}\makebox(0,0)[lt]{\lineheight{1.25}\smash{\begin{tabular}[t]{l}$+5$\end{tabular}}}}%
    \put(0.93551169,0.08311088){\color[rgb]{0.70196078,0.70196078,0.70196078}\makebox(0,0)[lt]{\lineheight{1.25}\smash{\begin{tabular}[t]{l}$+6$\end{tabular}}}}%
    \put(0.53406492,0.08311088){\color[rgb]{0.70196078,0.70196078,0.70196078}\makebox(0,0)[lt]{\lineheight{1.25}\smash{\begin{tabular}[t]{l}$-1$\end{tabular}}}}%
    \put(0.47671555,0.08311088){\color[rgb]{0.70196078,0.70196078,0.70196078}\makebox(0,0)[lt]{\lineheight{1.25}\smash{\begin{tabular}[t]{l}$-2$\end{tabular}}}}%
    \put(0.41936612,0.08311088){\color[rgb]{0.70196078,0.70196078,0.70196078}\makebox(0,0)[lt]{\lineheight{1.25}\smash{\begin{tabular}[t]{l}$-3$\end{tabular}}}}%
    \put(0.36201663,0.08311088){\color[rgb]{0.70196078,0.70196078,0.70196078}\makebox(0,0)[lt]{\lineheight{1.25}\smash{\begin{tabular}[t]{l}$-4$\end{tabular}}}}%
    \put(0.30466714,0.08311088){\color[rgb]{0.70196078,0.70196078,0.70196078}\makebox(0,0)[lt]{\lineheight{1.25}\smash{\begin{tabular}[t]{l}$-5$\end{tabular}}}}%
    \put(0.24731765,0.08311088){\color[rgb]{0.70196078,0.70196078,0.70196078}\makebox(0,0)[lt]{\lineheight{1.25}\smash{\begin{tabular}[t]{l}$-6$\end{tabular}}}}%
    \put(0.18996816,0.08311088){\color[rgb]{0.70196078,0.70196078,0.70196078}\makebox(0,0)[lt]{\lineheight{1.25}\smash{\begin{tabular}[t]{l}$-7$\end{tabular}}}}%
    \put(0.13261867,0.08311088){\color[rgb]{0.70196078,0.70196078,0.70196078}\makebox(0,0)[lt]{\lineheight{1.25}\smash{\begin{tabular}[t]{l}$-8$\end{tabular}}}}%
    \put(0.07526918,0.08311088){\color[rgb]{0.70196078,0.70196078,0.70196078}\makebox(0,0)[lt]{\lineheight{1.25}\smash{\begin{tabular}[t]{l}$-9$\end{tabular}}}}%
    \put(0.00578071,0.08311088){\color[rgb]{0.70196078,0.70196078,0.70196078}\makebox(0,0)[lt]{\lineheight{1.25}\smash{\begin{tabular}[t]{l}$-10$\end{tabular}}}}%
    \put(0.21519735,0.4675853){\color[rgb]{0,0.50196078,0}\makebox(0,0)[lt]{\lineheight{1.25}\smash{\begin{tabular}[t]{l}$x$-split line\end{tabular}}}}%
    \put(0,0){\includegraphics[width=\unitlength,page=3]{SplitTemplateWithLabelsWhole.pdf}}%
    \put(0.27365971,0.00596939){\color[rgb]{0,0,0}\makebox(0,0)[lt]{\lineheight{1.25}\smash{\begin{tabular}[t]{l}$w=$\end{tabular}}}}%
    \put(0.38836173,0.00596939){\color[rgb]{0.33333333,0,0}\makebox(0,0)[lt]{\lineheight{1.25}\smash{\begin{tabular}[t]{l}$y^{2}$\end{tabular}}}}%
    \put(0.42659472,0.00596939){\color[rgb]{0,0.33333333,0.83137255}\makebox(0,0)[lt]{\lineheight{1.25}\smash{\begin{tabular}[t]{l}$x^{5}$\end{tabular}}}}%
    \put(0.46482771,0.00596939){\color[rgb]{0.66666667,0,0}\makebox(0,0)[lt]{\lineheight{1.25}\smash{\begin{tabular}[t]{l}$y^{2}$\end{tabular}}}}%
    \put(0.50306076,0.00596939){\color[rgb]{0.16470588,0.49803922,1}\makebox(0,0)[lt]{\lineheight{1.25}\smash{\begin{tabular}[t]{l}$x^{7}$\end{tabular}}}}%
    \put(0.54129375,0.00596939){\color[rgb]{1,0,0}\makebox(0,0)[lt]{\lineheight{1.25}\smash{\begin{tabular}[t]{l}$y^{6}$\end{tabular}}}}%
    \put(0.57952675,0.00596939){\color[rgb]{0.33333333,0.6,1}\makebox(0,0)[lt]{\lineheight{1.25}\smash{\begin{tabular}[t]{l}$x^{2}$\end{tabular}}}}%
    \put(0.61775974,0.00596939){\color[rgb]{1,0.50196078,0.50196078}\makebox(0,0)[lt]{\lineheight{1.25}\smash{\begin{tabular}[t]{l}$y^{2}$\end{tabular}}}}%
    \put(0.65599273,0.00596939){\color[rgb]{0.83529412,0.89803922,1}\makebox(0,0)[lt]{\lineheight{1.25}\smash{\begin{tabular}[t]{l}$x^{5}$\end{tabular}}}}%
    \put(0.69422567,0.00596939){\color[rgb]{1,0.83529412,0.83529412}\makebox(0,0)[lt]{\lineheight{1.25}\smash{\begin{tabular}[t]{l}$y^{3}$\end{tabular}}}}%
    \put(0.27365971,0.00596939){\color[rgb]{0,0,0}\makebox(0,0)[lt]{\lineheight{1.25}\smash{\begin{tabular}[t]{l}$w=$\end{tabular}}}}%
    \put(0.33483554,0.00597242){\color[rgb]{0,0.2,0.50196078}\makebox(0,0)[lt]{\lineheight{1.25}\smash{\begin{tabular}[t]{l}$x_1^{10}$\end{tabular}}}}%
    \put(0.38836173,0.00596939){\color[rgb]{0.33333333,0,0}\makebox(0,0)[lt]{\lineheight{1.25}\smash{\begin{tabular}[t]{l}$y_1^{2}$\end{tabular}}}}%
    \put(0.46482771,0.00597242){\color[rgb]{0.66666667,0,0}\makebox(0,0)[lt]{\lineheight{1.25}\smash{\begin{tabular}[t]{l}$y_2^{2}$\end{tabular}}}}%
    \put(0.50306076,0.00597242){\color[rgb]{0.16470588,0.49803922,1}\makebox(0,0)[lt]{\lineheight{1.25}\smash{\begin{tabular}[t]{l}$x_3^{7}$\end{tabular}}}}%
    \put(0.54129375,0.00597242){\color[rgb]{1,0,0}\makebox(0,0)[lt]{\lineheight{1.25}\smash{\begin{tabular}[t]{l}$y_3^{6}$\end{tabular}}}}%
    \put(0.57952675,0.00597242){\color[rgb]{0.33333333,0.6,1}\makebox(0,0)[lt]{\lineheight{1.25}\smash{\begin{tabular}[t]{l}$x_4^{2}$\end{tabular}}}}%
    \put(0.61775974,0.00597242){\color[rgb]{1,0.50196078,0.50196078}\makebox(0,0)[lt]{\lineheight{1.25}\smash{\begin{tabular}[t]{l}$y_4^{2}$\end{tabular}}}}%
    \put(0.65599273,0.00597242){\color[rgb]{0.83529412,0.89803922,1}\makebox(0,0)[lt]{\lineheight{1.25}\smash{\begin{tabular}[t]{l}$x_5^{5}$\end{tabular}}}}%
    \put(0.69422567,0.00596939){\color[rgb]{1,0.83529412,0.83529412}\makebox(0,0)[lt]{\lineheight{1.25}\smash{\begin{tabular}[t]{l}$y_5^{3}$\end{tabular}}}}%
    \put(0.42659472,0.00597242){\color[rgb]{0,0.33333333,0.83137255}\makebox(0,0)[lt]{\lineheight{1.25}\smash{\begin{tabular}[t]{l}$x_2^{5}$\end{tabular}}}}%
    \put(0,0){\includegraphics[width=\unitlength,page=4]{SplitTemplateWithLabelsWhole.pdf}}%
  \end{picture}%
\endgroup%

%% file: figures/AnnularDehnFilling.pdf_tex
\begingroup%
  \makeatletter%
  \providecommand\color[2][]{%
    \errmessage{(Inkscape) Color is used for the text in Inkscape, but the package 'color.sty' is not loaded}%
    \renewcommand\color[2][]{}%
  }%
  \providecommand\transparent[1]{%
    \errmessage{(Inkscape) Transparency is used (non-zero) for the text in Inkscape, but the package 'transparent.sty' is not loaded}%
    \renewcommand\transparent[1]{}%
  }%
  \providecommand\rotatebox[2]{#2}%
  \newcommand*\fsize{\dimexpr\f@size pt\relax}%
  \newcommand*\lineheight[1]{\fontsize{\fsize}{#1\fsize}\selectfont}%
  \ifx\svgwidth\undefined%
    \setlength{\unitlength}{303.22492639bp}%
    \ifx\svgscale\undefined%
      \relax%
    \else%
      \setlength{\unitlength}{\unitlength * \real{\svgscale}}%
    \fi%
  \else%
    \setlength{\unitlength}{\svgwidth}%
  \fi%
  \global\let\svgwidth\undefined%
  \global\let\svgscale\undefined%
  \makeatother%
  \begin{picture}(1,0.35986736)%
    \lineheight{1}%
    \setlength\tabcolsep{0pt}%
    \put(0,0){\includegraphics[width=\unitlength,page=1]{AnnularDehnFilling.pdf}}%
    \put(0.26800953,0.03883489){\makebox(0,0)[lt]{\lineheight{1.25}\smash{\begin{tabular}[t]{l}\fontsize{6pt}{1em}$+\frac{1}{3}$\end{tabular}}}}%
    \put(0.16768331,0.16793528){\makebox(0,0)[lt]{\lineheight{1.25}\smash{\begin{tabular}[t]{l}\fontsize{6pt}{1em}$-\frac{1}{3}$\end{tabular}}}}%
    \put(0.37464111,0.13982814){\color[rgb]{0,0.2,0.50196078}\makebox(0,0)[lt]{\lineheight{1.25}\smash{\begin{tabular}[t]{l}$R$\end{tabular}}}}%
    \put(0.15358408,0.00499392){\color[rgb]{0.82745098,0.7372549,0.37254902}\makebox(0,0)[lt]{\lineheight{1.25}\smash{\begin{tabular}[t]{l}$A$\end{tabular}}}}%
    \put(0.94625741,0.13982814){\color[rgb]{0,0.2,0.50196078}\makebox(0,0)[lt]{\lineheight{1.25}\smash{\begin{tabular}[t]{l}$R$\end{tabular}}}}%
    \put(0.72520187,0.00499392){\color[rgb]{0.82745098,0.7372549,0.37254902}\makebox(0,0)[lt]{\lineheight{1.25}\smash{\begin{tabular}[t]{l}$A$\end{tabular}}}}%
    \put(0,0){\includegraphics[width=\unitlength,page=2]{AnnularDehnFilling.pdf}}%
    \put(0.46027217,0.33806563){\makebox(0,0)[t]{\lineheight{1.25}\smash{\begin{tabular}[t]{c}\fontsize{8pt}{1em} $(+\frac{1}{3})\textup{-annular}$ \\\fontsize{8pt}{1em} $\textup{Dehn filling}$\end{tabular}}}}%
  \end{picture}%
\endgroup%

%% file: figures/LinkCompleInSFSpaceResized_JustSF.pdf_tex
\begingroup%
  \makeatletter%
  \providecommand\color[2][]{%
    \errmessage{(Inkscape) Color is used for the text in Inkscape, but the package 'color.sty' is not loaded}%
    \renewcommand\color[2][]{}%
  }%
  \providecommand\transparent[1]{%
    \errmessage{(Inkscape) Transparency is used (non-zero) for the text in Inkscape, but the package 'transparent.sty' is not loaded}%
    \renewcommand\transparent[1]{}%
  }%
  \providecommand\rotatebox[2]{#2}%
  \newcommand*\fsize{\dimexpr\f@size pt\relax}%
  \newcommand*\lineheight[1]{\fontsize{\fsize}{#1\fsize}\selectfont}%
  \ifx\svgwidth\undefined%
    \setlength{\unitlength}{373.49191101bp}%
    \ifx\svgscale\undefined%
      \relax%
    \else%
      \setlength{\unitlength}{\unitlength * \real{\svgscale}}%
    \fi%
  \else%
    \setlength{\unitlength}{\svgwidth}%
  \fi%
  \global\let\svgwidth\undefined%
  \global\let\svgscale\undefined%
  \makeatother%
  \begin{picture}(1,0.42628393)%
    \lineheight{1}%
    \setlength\tabcolsep{0pt}%
    \put(0,0){\includegraphics[width=\unitlength,page=1]{LinkCompleInSFSpaceResized_JustSF.pdf}}%
    \put(0.67453834,0.1005098){\color[rgb]{0.50588235,0.50588235,0.50588235}\makebox(0,0)[lt]{\lineheight{1.25}\smash{\begin{tabular}[t]{l}$U$\end{tabular}}}}%
  \end{picture}%
\endgroup%

%% file: figures/ParentKnotSteps1ab.pdf_tex
\begingroup%
  \makeatletter%
  \providecommand\color[2][]{%
    \errmessage{(Inkscape) Color is used for the text in Inkscape, but the package 'color.sty' is not loaded}%
    \renewcommand\color[2][]{}%
  }%
  \providecommand\transparent[1]{%
    \errmessage{(Inkscape) Transparency is used (non-zero) for the text in Inkscape, but the package 'transparent.sty' is not loaded}%
    \renewcommand\transparent[1]{}%
  }%
  \providecommand\rotatebox[2]{#2}%
  \newcommand*\fsize{\dimexpr\f@size pt\relax}%
  \newcommand*\lineheight[1]{\fontsize{\fsize}{#1\fsize}\selectfont}%
  \ifx\svgwidth\undefined%
    \setlength{\unitlength}{226.78515638bp}%
    \ifx\svgscale\undefined%
      \relax%
    \else%
      \setlength{\unitlength}{\unitlength * \real{\svgscale}}%
    \fi%
  \else%
    \setlength{\unitlength}{\svgwidth}%
  \fi%
  \global\let\svgwidth\undefined%
  \global\let\svgscale\undefined%
  \makeatother%
  \begin{picture}(1,0.56046543)%
    \lineheight{1}%
    \setlength\tabcolsep{0pt}%
    \put(0,0){\includegraphics[width=\unitlength,page=1]{ParentKnotSteps1ab.pdf}}%
    \put(0.89308055,0.00368709){\color[rgb]{0.47058824,0.40392157,0.12941176}\transparent{0.69999999}\makebox(0,0)[lt]{\lineheight{1.25}\smash{\begin{tabular}[t]{l}\fontsize{8pt}{1em}$A^y_6$\end{tabular}}}}%
    \put(0.78130712,0.00368708){\color[rgb]{0.47058824,0.40392157,0.12941176}\transparent{0.69999999}\makebox(0,0)[lt]{\lineheight{1.25}\smash{\begin{tabular}[t]{l}\fontsize{8pt}{1em}$A^y_3$\end{tabular}}}}%
    \put(0.69027977,0.00368705){\color[rgb]{0.47058824,0.40392157,0.12941176}\transparent{0.69999999}\makebox(0,0)[lt]{\lineheight{1.25}\smash{\begin{tabular}[t]{l}\fontsize{8pt}{1em}$A^y_2$\end{tabular}}}}%
    \put(0,0){\includegraphics[width=\unitlength,page=2]{ParentKnotSteps1ab.pdf}}%
    \put(0.16597489,0.52313503){\color[rgb]{0,0,0}\makebox(0,0)[lt]{\lineheight{1.25}\smash{\begin{tabular}[t]{l}$w=$\end{tabular}}}}%
    \put(0.3380409,0.52313503){\color[rgb]{0.33333333,0,0}\makebox(0,0)[lt]{\lineheight{1.25}\smash{\begin{tabular}[t]{l}$y^{2}$\end{tabular}}}}%
    \put(0.39539468,0.52313503){\color[rgb]{0,0.33333333,0.83137255}\makebox(0,0)[lt]{\lineheight{1.25}\smash{\begin{tabular}[t]{l}$x^{5}$\end{tabular}}}}%
    \put(0.45274846,0.52313503){\color[rgb]{0.66666667,0,0}\makebox(0,0)[lt]{\lineheight{1.25}\smash{\begin{tabular}[t]{l}$y^{2}$\end{tabular}}}}%
    \put(0.51010252,0.52313503){\color[rgb]{0.16470588,0.49803922,1}\makebox(0,0)[lt]{\lineheight{1.25}\smash{\begin{tabular}[t]{l}$x^{7}$\end{tabular}}}}%
    \put(0.56745649,0.52313503){\color[rgb]{1,0,0}\makebox(0,0)[lt]{\lineheight{1.25}\smash{\begin{tabular}[t]{l}$y^{6}$\end{tabular}}}}%
    \put(0.62481036,0.52313503){\color[rgb]{0.33333333,0.6,1}\makebox(0,0)[lt]{\lineheight{1.25}\smash{\begin{tabular}[t]{l}$x^{2}$\end{tabular}}}}%
    \put(0.68216415,0.52313503){\color[rgb]{1,0.50196078,0.50196078}\makebox(0,0)[lt]{\lineheight{1.25}\smash{\begin{tabular}[t]{l}$y^{2}$\end{tabular}}}}%
    \put(0.73951821,0.52313503){\color[rgb]{0.83529412,0.89803922,1}\makebox(0,0)[lt]{\lineheight{1.25}\smash{\begin{tabular}[t]{l}$x^{5}$\end{tabular}}}}%
    \put(0.79687208,0.52313503){\color[rgb]{1,0.83529412,0.83529412}\makebox(0,0)[lt]{\lineheight{1.25}\smash{\begin{tabular}[t]{l}$y^{3}$\end{tabular}}}}%
    \put(0.16597489,0.52313503){\color[rgb]{0,0,0}\makebox(0,0)[lt]{\lineheight{1.25}\smash{\begin{tabular}[t]{l}$w=$\end{tabular}}}}%
    \put(0.25774548,0.52313999){\color[rgb]{0,0.2,0.50196078}\makebox(0,0)[lt]{\lineheight{1.25}\smash{\begin{tabular}[t]{l}$x_1^{10}$\end{tabular}}}}%
    \put(0.3380409,0.52313503){\color[rgb]{0.33333333,0,0}\makebox(0,0)[lt]{\lineheight{1.25}\smash{\begin{tabular}[t]{l}$y_1^{2}$\end{tabular}}}}%
    \put(0.45274846,0.52313999){\color[rgb]{0.66666667,0,0}\makebox(0,0)[lt]{\lineheight{1.25}\smash{\begin{tabular}[t]{l}$y_2^{2}$\end{tabular}}}}%
    \put(0.51010252,0.52313999){\color[rgb]{0.16470588,0.49803922,1}\makebox(0,0)[lt]{\lineheight{1.25}\smash{\begin{tabular}[t]{l}$x_3^{7}$\end{tabular}}}}%
    \put(0.56745649,0.52313999){\color[rgb]{1,0,0}\makebox(0,0)[lt]{\lineheight{1.25}\smash{\begin{tabular}[t]{l}$y_3^{6}$\end{tabular}}}}%
    \put(0.62481036,0.52313999){\color[rgb]{0.33333333,0.6,1}\makebox(0,0)[lt]{\lineheight{1.25}\smash{\begin{tabular}[t]{l}$x_4^{2}$\end{tabular}}}}%
    \put(0.68216415,0.52313999){\color[rgb]{1,0.50196078,0.50196078}\makebox(0,0)[lt]{\lineheight{1.25}\smash{\begin{tabular}[t]{l}$y_4^{2}$\end{tabular}}}}%
    \put(0.73951821,0.52313999){\color[rgb]{0.83529412,0.89803922,1}\makebox(0,0)[lt]{\lineheight{1.25}\smash{\begin{tabular}[t]{l}$x_5^{5}$\end{tabular}}}}%
    \put(0.79783584,0.52313505){\color[rgb]{1,0.83529412,0.83529412}\makebox(0,0)[lt]{\lineheight{1.25}\smash{\begin{tabular}[t]{l}$y_5^{3}$\end{tabular}}}}%
    \put(0.39539468,0.52313999){\color[rgb]{0,0.33333333,0.83137255}\makebox(0,0)[lt]{\lineheight{1.25}\smash{\begin{tabular}[t]{l}$x_2^{5}$\end{tabular}}}}%
    \put(0,0){\includegraphics[width=\unitlength,page=3]{ParentKnotSteps1ab.pdf}}%
    \put(0.50441536,0.00368705){\color[rgb]{0.47058824,0.40392157,0.12941176}\transparent{0.69999999}\makebox(0,0)[lt]{\lineheight{1.25}\smash{\begin{tabular}[t]{l}\fontsize{8pt}{1em}$A^x_2$\end{tabular}}}}%
    \put(0.06247104,0.00368705){\color[rgb]{0.47058824,0.40392157,0.12941176}\transparent{0.69999999}\makebox(0,0)[lt]{\lineheight{1.25}\smash{\begin{tabular}[t]{l}\fontsize{8pt}{1em}$A^x_{10}$\end{tabular}}}}%
    \put(0.21942326,0.00368705){\color[rgb]{0.47058824,0.40392157,0.12941176}\transparent{0.69999999}\makebox(0,0)[lt]{\lineheight{1.25}\smash{\begin{tabular}[t]{l}\fontsize{8pt}{1em}$A^x_7$\end{tabular}}}}%
    \put(0.36811481,0.00368705){\color[rgb]{0.47058824,0.40392157,0.12941176}\transparent{0.69999999}\makebox(0,0)[lt]{\lineheight{1.25}\smash{\begin{tabular}[t]{l}\fontsize{8pt}{1em}$A^x_5$\end{tabular}}}}%
    \put(0,0){\includegraphics[width=\unitlength,page=4]{ParentKnotSteps1ab.pdf}}%
  \end{picture}%
\endgroup%

%% file: figures/LinkCompleInSFSpaceResized_Step1.pdf_tex
\begingroup%
  \makeatletter%
  \providecommand\color[2][]{%
    \errmessage{(Inkscape) Color is used for the text in Inkscape, but the package 'color.sty' is not loaded}%
    \renewcommand\color[2][]{}%
  }%
  \providecommand\transparent[1]{%
    \errmessage{(Inkscape) Transparency is used (non-zero) for the text in Inkscape, but the package 'transparent.sty' is not loaded}%
    \renewcommand\transparent[1]{}%
  }%
  \providecommand\rotatebox[2]{#2}%
  \newcommand*\fsize{\dimexpr\f@size pt\relax}%
  \newcommand*\lineheight[1]{\fontsize{\fsize}{#1\fsize}\selectfont}%
  \ifx\svgwidth\undefined%
    \setlength{\unitlength}{232.8887902bp}%
    \ifx\svgscale\undefined%
      \relax%
    \else%
      \setlength{\unitlength}{\unitlength * \real{\svgscale}}%
    \fi%
  \else%
    \setlength{\unitlength}{\svgwidth}%
  \fi%
  \global\let\svgwidth\undefined%
  \global\let\svgscale\undefined%
  \makeatother%
  \begin{picture}(1,0.42643002)%
    \lineheight{1}%
    \setlength\tabcolsep{0pt}%
    \put(0,0){\includegraphics[width=\unitlength,page=1]{LinkCompleInSFSpaceResized_Step1.pdf}}%
    \put(0.67658886,0.08430942){\color[rgb]{0.50588235,0.50588235,0.50588235}\makebox(0,0)[lt]{\lineheight{1.25}\smash{\begin{tabular}[t]{l}$U$\end{tabular}}}}%
  \end{picture}%
\endgroup%

%% file: figures/ParentKnotSteps2a.pdf_tex
\begingroup%
  \makeatletter%
  \providecommand\color[2][]{%
    \errmessage{(Inkscape) Color is used for the text in Inkscape, but the package 'color.sty' is not loaded}%
    \renewcommand\color[2][]{}%
  }%
  \providecommand\transparent[1]{%
    \errmessage{(Inkscape) Transparency is used (non-zero) for the text in Inkscape, but the package 'transparent.sty' is not loaded}%
    \renewcommand\transparent[1]{}%
  }%
  \providecommand\rotatebox[2]{#2}%
  \newcommand*\fsize{\dimexpr\f@size pt\relax}%
  \newcommand*\lineheight[1]{\fontsize{\fsize}{#1\fsize}\selectfont}%
  \ifx\svgwidth\undefined%
    \setlength{\unitlength}{226.78515638bp}%
    \ifx\svgscale\undefined%
      \relax%
    \else%
      \setlength{\unitlength}{\unitlength * \real{\svgscale}}%
    \fi%
  \else%
    \setlength{\unitlength}{\svgwidth}%
  \fi%
  \global\let\svgwidth\undefined%
  \global\let\svgscale\undefined%
  \makeatother%
  \begin{picture}(1,0.49214524)%
    \lineheight{1}%
    \setlength\tabcolsep{0pt}%
    \put(0,0){\includegraphics[width=\unitlength,page=1]{ParentKnotSteps2a.pdf}}%
    \put(0.89308794,0.00368707){\color[rgb]{0.47058824,0.40392157,0.12941176}\transparent{0.69999999}\makebox(0,0)[lt]{\lineheight{1.25}\smash{\begin{tabular}[t]{l}\fontsize{8pt}{1em}$A^y_6$\end{tabular}}}}%
    \put(0.78131541,0.00368712){\color[rgb]{0.47058824,0.40392157,0.12941176}\transparent{0.69999999}\makebox(0,0)[lt]{\lineheight{1.25}\smash{\begin{tabular}[t]{l}\fontsize{8pt}{1em}$A^y_3$\end{tabular}}}}%
    \put(0.69027977,0.00368711){\color[rgb]{0.47058824,0.40392157,0.12941176}\transparent{0.69999999}\makebox(0,0)[lt]{\lineheight{1.25}\smash{\begin{tabular}[t]{l}\fontsize{8pt}{1em}$A^y_2$\end{tabular}}}}%
    \put(0,0){\includegraphics[width=\unitlength,page=2]{ParentKnotSteps2a.pdf}}%
    \put(0.50441536,0.00368711){\color[rgb]{0.47058824,0.40392157,0.12941176}\transparent{0.69999999}\makebox(0,0)[lt]{\lineheight{1.25}\smash{\begin{tabular}[t]{l}\fontsize{8pt}{1em}$A^x_2$\end{tabular}}}}%
    \put(0.06247104,0.00368711){\color[rgb]{0.47058824,0.40392157,0.12941176}\transparent{0.69999999}\makebox(0,0)[lt]{\lineheight{1.25}\smash{\begin{tabular}[t]{l}\fontsize{8pt}{1em}$A^x_{10}$\end{tabular}}}}%
    \put(0.21942326,0.00368711){\color[rgb]{0.47058824,0.40392157,0.12941176}\transparent{0.69999999}\makebox(0,0)[lt]{\lineheight{1.25}\smash{\begin{tabular}[t]{l}\fontsize{8pt}{1em}$A^x_7$\end{tabular}}}}%
    \put(0.36811481,0.00368711){\color[rgb]{0.47058824,0.40392157,0.12941176}\transparent{0.69999999}\makebox(0,0)[lt]{\lineheight{1.25}\smash{\begin{tabular}[t]{l}\fontsize{8pt}{1em}$A^x_5$\end{tabular}}}}%
    \put(0,0){\includegraphics[width=\unitlength,page=3]{ParentKnotSteps2a.pdf}}%
  \end{picture}%
\endgroup%

%% file: figures/ParentKnotSteps2ab.pdf_tex
\begingroup%
  \makeatletter%
  \providecommand\color[2][]{%
    \errmessage{(Inkscape) Color is used for the text in Inkscape, but the package 'color.sty' is not loaded}%
    \renewcommand\color[2][]{}%
  }%
  \providecommand\transparent[1]{%
    \errmessage{(Inkscape) Transparency is used (non-zero) for the text in Inkscape, but the package 'transparent.sty' is not loaded}%
    \renewcommand\transparent[1]{}%
  }%
  \providecommand\rotatebox[2]{#2}%
  \newcommand*\fsize{\dimexpr\f@size pt\relax}%
  \newcommand*\lineheight[1]{\fontsize{\fsize}{#1\fsize}\selectfont}%
  \ifx\svgwidth\undefined%
    \setlength{\unitlength}{226.78515638bp}%
    \ifx\svgscale\undefined%
      \relax%
    \else%
      \setlength{\unitlength}{\unitlength * \real{\svgscale}}%
    \fi%
  \else%
    \setlength{\unitlength}{\svgwidth}%
  \fi%
  \global\let\svgwidth\undefined%
  \global\let\svgscale\undefined%
  \makeatother%
  \begin{picture}(1,0.49214524)%
    \lineheight{1}%
    \setlength\tabcolsep{0pt}%
    \put(0,0){\includegraphics[width=\unitlength,page=1]{ParentKnotSteps2ab.pdf}}%
    \put(0.89308794,0.00368707){\color[rgb]{0.47058824,0.40392157,0.12941176}\transparent{0.69999999}\makebox(0,0)[lt]{\lineheight{1.25}\smash{\begin{tabular}[t]{l}\fontsize{8pt}{1em}$A^y_6$\end{tabular}}}}%
    \put(0.78131541,0.00368712){\color[rgb]{0.47058824,0.40392157,0.12941176}\transparent{0.69999999}\makebox(0,0)[lt]{\lineheight{1.25}\smash{\begin{tabular}[t]{l}\fontsize{8pt}{1em}$A^y_3$\end{tabular}}}}%
    \put(0.69027977,0.00368711){\color[rgb]{0.47058824,0.40392157,0.12941176}\transparent{0.69999999}\makebox(0,0)[lt]{\lineheight{1.25}\smash{\begin{tabular}[t]{l}\fontsize{8pt}{1em}$A^y_2$\end{tabular}}}}%
    \put(0,0){\includegraphics[width=\unitlength,page=2]{ParentKnotSteps2ab.pdf}}%
    \put(0.50441536,0.00368711){\color[rgb]{0.47058824,0.40392157,0.12941176}\transparent{0.69999999}\makebox(0,0)[lt]{\lineheight{1.25}\smash{\begin{tabular}[t]{l}\fontsize{8pt}{1em}$A^x_2$\end{tabular}}}}%
    \put(0.06247104,0.00368711){\color[rgb]{0.47058824,0.40392157,0.12941176}\transparent{0.69999999}\makebox(0,0)[lt]{\lineheight{1.25}\smash{\begin{tabular}[t]{l}\fontsize{8pt}{1em}$A^x_{10}$\end{tabular}}}}%
    \put(0.21942326,0.00368711){\color[rgb]{0.47058824,0.40392157,0.12941176}\transparent{0.69999999}\makebox(0,0)[lt]{\lineheight{1.25}\smash{\begin{tabular}[t]{l}\fontsize{8pt}{1em}$A^x_7$\end{tabular}}}}%
    \put(0.36811481,0.00368711){\color[rgb]{0.47058824,0.40392157,0.12941176}\transparent{0.69999999}\makebox(0,0)[lt]{\lineheight{1.25}\smash{\begin{tabular}[t]{l}\fontsize{8pt}{1em}$A^x_5$\end{tabular}}}}%
    \put(0,0){\includegraphics[width=\unitlength,page=3]{ParentKnotSteps2ab.pdf}}%
  \end{picture}%
\endgroup%

%% file: figures/ParentKnotSteps3-4.pdf_tex
\begingroup%
  \makeatletter%
  \providecommand\color[2][]{%
    \errmessage{(Inkscape) Color is used for the text in Inkscape, but the package 'color.sty' is not loaded}%
    \renewcommand\color[2][]{}%
  }%
  \providecommand\transparent[1]{%
    \errmessage{(Inkscape) Transparency is used (non-zero) for the text in Inkscape, but the package 'transparent.sty' is not loaded}%
    \renewcommand\transparent[1]{}%
  }%
  \providecommand\rotatebox[2]{#2}%
  \newcommand*\fsize{\dimexpr\f@size pt\relax}%
  \newcommand*\lineheight[1]{\fontsize{\fsize}{#1\fsize}\selectfont}%
  \ifx\svgwidth\undefined%
    \setlength{\unitlength}{226.78515626bp}%
    \ifx\svgscale\undefined%
      \relax%
    \else%
      \setlength{\unitlength}{\unitlength * \real{\svgscale}}%
    \fi%
  \else%
    \setlength{\unitlength}{\svgwidth}%
  \fi%
  \global\let\svgwidth\undefined%
  \global\let\svgscale\undefined%
  \makeatother%
  \begin{picture}(1,0.49214527)%
    \lineheight{1}%
    \setlength\tabcolsep{0pt}%
    \put(0,0){\includegraphics[width=\unitlength,page=1]{ParentKnotSteps3-4.pdf}}%
    \put(0.89970184,0.00368712){\color[rgb]{0.47058824,0.40392157,0.12941176}\transparent{0.69999999}\makebox(0,0)[lt]{\lineheight{1.25}\smash{\begin{tabular}[t]{l}\fontsize{8pt}{1em}$A^y_6$\end{tabular}}}}%
    \put(0.78793525,0.00368706){\color[rgb]{0.47058824,0.40392157,0.12941176}\transparent{0.69999999}\makebox(0,0)[lt]{\lineheight{1.25}\smash{\begin{tabular}[t]{l}\fontsize{8pt}{1em}$A^y_3$\end{tabular}}}}%
    \put(0.69689397,0.00368714){\color[rgb]{0.47058824,0.40392157,0.12941176}\transparent{0.69999999}\makebox(0,0)[lt]{\lineheight{1.25}\smash{\begin{tabular}[t]{l}\fontsize{8pt}{1em}$A^y_2$\end{tabular}}}}%
    \put(0,0){\includegraphics[width=\unitlength,page=2]{ParentKnotSteps3-4.pdf}}%
    \put(0.51102956,0.00368714){\color[rgb]{0.47058824,0.40392157,0.12941176}\transparent{0.69999999}\makebox(0,0)[lt]{\lineheight{1.25}\smash{\begin{tabular}[t]{l}\fontsize{8pt}{1em}$A^x_2$\end{tabular}}}}%
    \put(0.06908524,0.00368714){\color[rgb]{0.47058824,0.40392157,0.12941176}\transparent{0.69999999}\makebox(0,0)[lt]{\lineheight{1.25}\smash{\begin{tabular}[t]{l}\fontsize{8pt}{1em}$A^x_{10}$\end{tabular}}}}%
    \put(0.22603746,0.00368714){\color[rgb]{0.47058824,0.40392157,0.12941176}\transparent{0.69999999}\makebox(0,0)[lt]{\lineheight{1.25}\smash{\begin{tabular}[t]{l}\fontsize{8pt}{1em}$A^x_7$\end{tabular}}}}%
    \put(0.37472902,0.00368714){\color[rgb]{0.47058824,0.40392157,0.12941176}\transparent{0.69999999}\makebox(0,0)[lt]{\lineheight{1.25}\smash{\begin{tabular}[t]{l}\fontsize{8pt}{1em}$A^x_5$\end{tabular}}}}%
    \put(0,0){\includegraphics[width=\unitlength,page=3]{ParentKnotSteps3-4.pdf}}%
  \end{picture}%
\endgroup%

%% file: figures/LinkCompleInSFSpaceResized.pdf_tex
\begingroup%
  \makeatletter%
  \providecommand\color[2][]{%
    \errmessage{(Inkscape) Color is used for the text in Inkscape, but the package 'color.sty' is not loaded}%
    \renewcommand\color[2][]{}%
  }%
  \providecommand\transparent[1]{%
    \errmessage{(Inkscape) Transparency is used (non-zero) for the text in Inkscape, but the package 'transparent.sty' is not loaded}%
    \renewcommand\transparent[1]{}%
  }%
  \providecommand\rotatebox[2]{#2}%
  \newcommand*\fsize{\dimexpr\f@size pt\relax}%
  \newcommand*\lineheight[1]{\fontsize{\fsize}{#1\fsize}\selectfont}%
  \ifx\svgwidth\undefined%
    \setlength{\unitlength}{374.7297126bp}%
    \ifx\svgscale\undefined%
      \relax%
    \else%
      \setlength{\unitlength}{\unitlength * \real{\svgscale}}%
    \fi%
  \else%
    \setlength{\unitlength}{\svgwidth}%
  \fi%
  \global\let\svgwidth\undefined%
  \global\let\svgscale\undefined%
  \makeatother%
  \begin{picture}(1,0.42487583)%
    \lineheight{1}%
    \setlength\tabcolsep{0pt}%
    \put(0,0){\includegraphics[width=\unitlength,page=1]{LinkCompleInSFSpaceResized.pdf}}%
    \put(0.71957315,0.21132254){\color[rgb]{1,0,1}\makebox(0,0)[lt]{\lineheight{1.25}\smash{\begin{tabular}[t]{l}\fontsize{5pt}{1em}$V_2$\end{tabular}}}}%
    \put(0.75310614,0.21933034){\color[rgb]{1,0,1}\makebox(0,0)[lt]{\lineheight{1.25}\smash{\begin{tabular}[t]{l}\fontsize{5pt}{1em}$V_3$\end{tabular}}}}%
    \put(0.81830144,0.22733654){\color[rgb]{1,0,1}\makebox(0,0)[lt]{\lineheight{1.25}\smash{\begin{tabular}[t]{l}\fontsize{5pt}{1em}$V_6$\end{tabular}}}}%
    \put(0,0){\includegraphics[width=\unitlength,page=2]{LinkCompleInSFSpaceResized.pdf}}%
    \put(0.67231021,0.1001778){\color[rgb]{0.50588235,0.50588235,0.50588235}\makebox(0,0)[lt]{\lineheight{1.25}\smash{\begin{tabular}[t]{l}$U$\end{tabular}}}}%
  \end{picture}%
\endgroup%

%% file: figures/SplitTemplateWithLabelsAnnuliKnotRing.pdf_tex
\begingroup%
  \makeatletter%
  \providecommand\color[2][]{%
    \errmessage{(Inkscape) Color is used for the text in Inkscape, but the package 'color.sty' is not loaded}%
    \renewcommand\color[2][]{}%
  }%
  \providecommand\transparent[1]{%
    \errmessage{(Inkscape) Transparency is used (non-zero) for the text in Inkscape, but the package 'transparent.sty' is not loaded}%
    \renewcommand\transparent[1]{}%
  }%
  \providecommand\rotatebox[2]{#2}%
  \newcommand*\fsize{\dimexpr\f@size pt\relax}%
  \newcommand*\lineheight[1]{\fontsize{\fsize}{#1\fsize}\selectfont}%
  \ifx\svgwidth\undefined%
    \setlength{\unitlength}{431.2819456bp}%
    \ifx\svgscale\undefined%
      \relax%
    \else%
      \setlength{\unitlength}{\unitlength * \real{\svgscale}}%
    \fi%
  \else%
    \setlength{\unitlength}{\svgwidth}%
  \fi%
  \global\let\svgwidth\undefined%
  \global\let\svgscale\undefined%
  \makeatother%
  \begin{picture}(1,0.48222826)%
    \lineheight{1}%
    \setlength\tabcolsep{0pt}%
    \put(0,0){\includegraphics[width=\unitlength,page=1]{SplitTemplateWithLabelsAnnuliKnotRing.pdf}}%
    \put(0.61471977,0.46260107){\color[rgb]{0,0.50196078,0}\makebox(0,0)[lt]{\lineheight{1.25}\smash{\begin{tabular}[t]{l}$y$-split line\end{tabular}}}}%
    \put(0.51661866,0.10430367){\color[rgb]{0,0.50196078,0}\makebox(0,0)[lt]{\lineheight{1.25}\smash{\begin{tabular}[t]{l}$0$\end{tabular}}}}%
    \put(0.55764253,0.10430066){\color[rgb]{0,0.50196078,0}\makebox(0,0)[lt]{\lineheight{1.25}\smash{\begin{tabular}[t]{l}$+1$\end{tabular}}}}%
    \put(0.60693736,0.10430066){\color[rgb]{0,0.50196078,0}\makebox(0,0)[lt]{\lineheight{1.25}\smash{\begin{tabular}[t]{l}$+2$\end{tabular}}}}%
    \put(0,0){\includegraphics[width=\unitlength,page=2]{SplitTemplateWithLabelsAnnuliKnotRing.pdf}}%
    \put(0.6562319,0.10430066){\color[rgb]{0,0.50196078,0}\makebox(0,0)[lt]{\lineheight{1.25}\smash{\begin{tabular}[t]{l}$+3$\end{tabular}}}}%
    \put(0.70552643,0.10430066){\color[rgb]{0,0.50196078,0}\makebox(0,0)[lt]{\lineheight{1.25}\smash{\begin{tabular}[t]{l}$+4$\end{tabular}}}}%
    \put(0.75482096,0.10430066){\color[rgb]{0,0.50196078,0}\makebox(0,0)[lt]{\lineheight{1.25}\smash{\begin{tabular}[t]{l}$+5$\end{tabular}}}}%
    \put(0.8041155,0.10430066){\color[rgb]{0,0.50196078,0}\makebox(0,0)[lt]{\lineheight{1.25}\smash{\begin{tabular}[t]{l}$+6$\end{tabular}}}}%
    \put(0.45905346,0.10430066){\color[rgb]{0,0.50196078,0}\makebox(0,0)[lt]{\lineheight{1.25}\smash{\begin{tabular}[t]{l}$-1$\end{tabular}}}}%
    \put(0.40975903,0.10430066){\color[rgb]{0,0.50196078,0}\makebox(0,0)[lt]{\lineheight{1.25}\smash{\begin{tabular}[t]{l}$-2$\end{tabular}}}}%
    \put(0.36046454,0.10430066){\color[rgb]{0,0.50196078,0}\makebox(0,0)[lt]{\lineheight{1.25}\smash{\begin{tabular}[t]{l}$-3$\end{tabular}}}}%
    \put(0.31117001,0.10430066){\color[rgb]{0,0.50196078,0}\makebox(0,0)[lt]{\lineheight{1.25}\smash{\begin{tabular}[t]{l}$-4$\end{tabular}}}}%
    \put(0.26187547,0.10430066){\color[rgb]{0,0.50196078,0}\makebox(0,0)[lt]{\lineheight{1.25}\smash{\begin{tabular}[t]{l}$-5$\end{tabular}}}}%
    \put(0.21258094,0.10430066){\color[rgb]{0,0.50196078,0}\makebox(0,0)[lt]{\lineheight{1.25}\smash{\begin{tabular}[t]{l}$-6$\end{tabular}}}}%
    \put(0.16328641,0.10430066){\color[rgb]{0,0.50196078,0}\makebox(0,0)[lt]{\lineheight{1.25}\smash{\begin{tabular}[t]{l}$-7$\end{tabular}}}}%
    \put(0.11399187,0.10430066){\color[rgb]{0,0.50196078,0}\makebox(0,0)[lt]{\lineheight{1.25}\smash{\begin{tabular}[t]{l}$-8$\end{tabular}}}}%
    \put(0.06469734,0.10430066){\color[rgb]{0,0.50196078,0}\makebox(0,0)[lt]{\lineheight{1.25}\smash{\begin{tabular}[t]{l}$-9$\end{tabular}}}}%
    \put(0.00496878,0.10430066){\color[rgb]{0,0.50196078,0}\makebox(0,0)[lt]{\lineheight{1.25}\smash{\begin{tabular}[t]{l}$-10$\end{tabular}}}}%
    \put(0.18497206,0.46259823){\color[rgb]{0,0.50196078,0}\makebox(0,0)[lt]{\lineheight{1.25}\smash{\begin{tabular}[t]{l}$x$-split line\end{tabular}}}}%
    \put(0,0){\includegraphics[width=\unitlength,page=3]{SplitTemplateWithLabelsAnnuliKnotRing.pdf}}%
    \put(0.58846325,0.04738314){\color[rgb]{0.47058824,0.40392157,0.12941176}\transparent{0.69999999}\makebox(0,0)[lt]{\lineheight{1.25}\smash{\begin{tabular}[t]{l}\fontsize{8pt}{1em}$A^y_2$\end{tabular}}}}%
    \put(0.66497829,0.04738314){\color[rgb]{0.47058824,0.40392157,0.12941176}\transparent{0.69999999}\makebox(0,0)[lt]{\lineheight{1.25}\smash{\begin{tabular}[t]{l}\fontsize{8pt}{1em}$A^y_3$\end{tabular}}}}%
    \put(0.75888311,0.04738314){\color[rgb]{0.47058824,0.40392157,0.12941176}\transparent{0.69999999}\makebox(0,0)[lt]{\lineheight{1.25}\smash{\begin{tabular}[t]{l}\fontsize{8pt}{1em}$A^y_6$\end{tabular}}}}%
    \put(0,0){\includegraphics[width=\unitlength,page=4]{SplitTemplateWithLabelsAnnuliKnotRing.pdf}}%
    \put(0.23522316,0.00513097){\color[rgb]{0,0,0}\makebox(0,0)[lt]{\lineheight{1.25}\smash{\begin{tabular}[t]{l}$w=$\end{tabular}}}}%
    \put(0.33381484,0.00513097){\color[rgb]{0.33333333,0,0}\makebox(0,0)[lt]{\lineheight{1.25}\smash{\begin{tabular}[t]{l}$y^{2}$\end{tabular}}}}%
    \put(0.36667781,0.00513097){\color[rgb]{0,0.33333333,0.83137255}\makebox(0,0)[lt]{\lineheight{1.25}\smash{\begin{tabular}[t]{l}$x^{5}$\end{tabular}}}}%
    \put(0.39954083,0.00513097){\color[rgb]{0.66666667,0,0}\makebox(0,0)[lt]{\lineheight{1.25}\smash{\begin{tabular}[t]{l}$y^{2}$\end{tabular}}}}%
    \put(0.43240395,0.00513097){\color[rgb]{0.16470588,0.49803922,1}\makebox(0,0)[lt]{\lineheight{1.25}\smash{\begin{tabular}[t]{l}$x^{7}$\end{tabular}}}}%
    \put(0.46526698,0.00513097){\color[rgb]{1,0,0}\makebox(0,0)[lt]{\lineheight{1.25}\smash{\begin{tabular}[t]{l}$y^{6}$\end{tabular}}}}%
    \put(0.49813,0.00513097){\color[rgb]{0.33333333,0.6,1}\makebox(0,0)[lt]{\lineheight{1.25}\smash{\begin{tabular}[t]{l}$x^{2}$\end{tabular}}}}%
    \put(0.53099302,0.00513097){\color[rgb]{1,0.50196078,0.50196078}\makebox(0,0)[lt]{\lineheight{1.25}\smash{\begin{tabular}[t]{l}$y^{2}$\end{tabular}}}}%
    \put(0.56385604,0.00513097){\color[rgb]{0.83529412,0.89803922,1}\makebox(0,0)[lt]{\lineheight{1.25}\smash{\begin{tabular}[t]{l}$x^{5}$\end{tabular}}}}%
    \put(0.59671902,0.00513097){\color[rgb]{1,0.83529412,0.83529412}\makebox(0,0)[lt]{\lineheight{1.25}\smash{\begin{tabular}[t]{l}$y^{3}$\end{tabular}}}}%
    \put(0.23522316,0.00513097){\color[rgb]{0,0,0}\makebox(0,0)[lt]{\lineheight{1.25}\smash{\begin{tabular}[t]{l}$w=$\end{tabular}}}}%
    \put(0.2878066,0.00513357){\color[rgb]{0,0.2,0.50196078}\makebox(0,0)[lt]{\lineheight{1.25}\smash{\begin{tabular}[t]{l}$x_1^{10}$\end{tabular}}}}%
    \put(0.33381484,0.00513097){\color[rgb]{0.33333333,0,0}\makebox(0,0)[lt]{\lineheight{1.25}\smash{\begin{tabular}[t]{l}$y_1^{2}$\end{tabular}}}}%
    \put(0.39954083,0.00513357){\color[rgb]{0.66666667,0,0}\makebox(0,0)[lt]{\lineheight{1.25}\smash{\begin{tabular}[t]{l}$y_2^{2}$\end{tabular}}}}%
    \put(0.43240395,0.00513357){\color[rgb]{0.16470588,0.49803922,1}\makebox(0,0)[lt]{\lineheight{1.25}\smash{\begin{tabular}[t]{l}$x_3^{7}$\end{tabular}}}}%
    \put(0.46526698,0.00513357){\color[rgb]{1,0,0}\makebox(0,0)[lt]{\lineheight{1.25}\smash{\begin{tabular}[t]{l}$y_3^{6}$\end{tabular}}}}%
    \put(0.49813,0.00513357){\color[rgb]{0.33333333,0.6,1}\makebox(0,0)[lt]{\lineheight{1.25}\smash{\begin{tabular}[t]{l}$x_4^{2}$\end{tabular}}}}%
    \put(0.53099302,0.00513357){\color[rgb]{1,0.50196078,0.50196078}\makebox(0,0)[lt]{\lineheight{1.25}\smash{\begin{tabular}[t]{l}$y_4^{2}$\end{tabular}}}}%
    \put(0.56385604,0.00513357){\color[rgb]{0.83529412,0.89803922,1}\makebox(0,0)[lt]{\lineheight{1.25}\smash{\begin{tabular}[t]{l}$x_5^{5}$\end{tabular}}}}%
    \put(0.59671902,0.00513097){\color[rgb]{1,0.83529412,0.83529412}\makebox(0,0)[lt]{\lineheight{1.25}\smash{\begin{tabular}[t]{l}$y_5^{3}$\end{tabular}}}}%
    \put(0.36667781,0.00513357){\color[rgb]{0,0.33333333,0.83137255}\makebox(0,0)[lt]{\lineheight{1.25}\smash{\begin{tabular}[t]{l}$x_2^{5}$\end{tabular}}}}%
    \put(0,0){\includegraphics[width=\unitlength,page=5]{SplitTemplateWithLabelsAnnuliKnotRing.pdf}}%
    \put(0.43195305,0.04738314){\color[rgb]{0.47058824,0.40392157,0.12941176}\transparent{0.69999999}\makebox(0,0)[lt]{\lineheight{1.25}\smash{\begin{tabular}[t]{l}\fontsize{8pt}{1em}$A^x_2$\end{tabular}}}}%
    \put(0.0598064,0.04738314){\color[rgb]{0.47058824,0.40392157,0.12941176}\transparent{0.69999999}\makebox(0,0)[lt]{\lineheight{1.25}\smash{\begin{tabular}[t]{l}\fontsize{8pt}{1em}$A^x_{10}$\end{tabular}}}}%
    \put(0.19197065,0.04738314){\color[rgb]{0.47058824,0.40392157,0.12941176}\transparent{0.69999999}\makebox(0,0)[lt]{\lineheight{1.25}\smash{\begin{tabular}[t]{l}\fontsize{8pt}{1em}$A^x_7$\end{tabular}}}}%
    \put(0.31717888,0.04738314){\color[rgb]{0.47058824,0.40392157,0.12941176}\transparent{0.69999999}\makebox(0,0)[lt]{\lineheight{1.25}\smash{\begin{tabular}[t]{l}\fontsize{8pt}{1em}$A^x_5$\end{tabular}}}}%
    \put(0,0){\includegraphics[width=\unitlength,page=6]{SplitTemplateWithLabelsAnnuliKnotRing.pdf}}%
    \put(0.58302873,0.26634478){\color[rgb]{1,0,1}\makebox(0,0)[lt]{\lineheight{1.25}\smash{\begin{tabular}[t]{l}\fontsize{8pt}{1em}$V_2$\end{tabular}}}}%
    \put(0.65243592,0.26930615){\color[rgb]{1,0,1}\makebox(0,0)[lt]{\lineheight{1.25}\smash{\begin{tabular}[t]{l}\fontsize{8pt}{1em}$V_3$\end{tabular}}}}%
    \put(0.78548032,0.27124749){\color[rgb]{1,0,1}\makebox(0,0)[lt]{\lineheight{1.25}\smash{\begin{tabular}[t]{l}\fontsize{8pt}{1em}$V_6$\end{tabular}}}}%
  \end{picture}%
\endgroup%

%% file: figures/SplitTemplateWithLabelsAnnuliKnotRing_AfterDF.pdf_tex
\begingroup%
  \makeatletter%
  \providecommand\color[2][]{%
    \errmessage{(Inkscape) Color is used for the text in Inkscape, but the package 'color.sty' is not loaded}%
    \renewcommand\color[2][]{}%
  }%
  \providecommand\transparent[1]{%
    \errmessage{(Inkscape) Transparency is used (non-zero) for the text in Inkscape, but the package 'transparent.sty' is not loaded}%
    \renewcommand\transparent[1]{}%
  }%
  \providecommand\rotatebox[2]{#2}%
  \newcommand*\fsize{\dimexpr\f@size pt\relax}%
  \newcommand*\lineheight[1]{\fontsize{\fsize}{#1\fsize}\selectfont}%
  \ifx\svgwidth\undefined%
    \setlength{\unitlength}{429.1948244bp}%
    \ifx\svgscale\undefined%
      \relax%
    \else%
      \setlength{\unitlength}{\unitlength * \real{\svgscale}}%
    \fi%
  \else%
    \setlength{\unitlength}{\svgwidth}%
  \fi%
  \global\let\svgwidth\undefined%
  \global\let\svgscale\undefined%
  \makeatother%
  \begin{picture}(1,0.48457328)%
    \lineheight{1}%
    \setlength\tabcolsep{0pt}%
    \put(0,0){\includegraphics[width=\unitlength,page=1]{SplitTemplateWithLabelsAnnuliKnotRing_AfterDF.pdf}}%
    \put(0.61770907,0.46485064){\color[rgb]{0,0.50196078,0}\makebox(0,0)[lt]{\lineheight{1.25}\smash{\begin{tabular}[t]{l}$y$-split line\end{tabular}}}}%
    \put(0.51913091,0.10481089){\color[rgb]{0,0.50196078,0}\makebox(0,0)[lt]{\lineheight{1.25}\smash{\begin{tabular}[t]{l}$0$\end{tabular}}}}%
    \put(0.56035427,0.10480786){\color[rgb]{0,0.50196078,0}\makebox(0,0)[lt]{\lineheight{1.25}\smash{\begin{tabular}[t]{l}$+1$\end{tabular}}}}%
    \put(0.60988882,0.10480786){\color[rgb]{0,0.50196078,0}\makebox(0,0)[lt]{\lineheight{1.25}\smash{\begin{tabular}[t]{l}$+2$\end{tabular}}}}%
    \put(0,0){\includegraphics[width=\unitlength,page=2]{SplitTemplateWithLabelsAnnuliKnotRing_AfterDF.pdf}}%
    \put(0.65942307,0.10480786){\color[rgb]{0,0.50196078,0}\makebox(0,0)[lt]{\lineheight{1.25}\smash{\begin{tabular}[t]{l}$+3$\end{tabular}}}}%
    \put(0.70895732,0.10480786){\color[rgb]{0,0.50196078,0}\makebox(0,0)[lt]{\lineheight{1.25}\smash{\begin{tabular}[t]{l}$+4$\end{tabular}}}}%
    \put(0.75849157,0.10480786){\color[rgb]{0,0.50196078,0}\makebox(0,0)[lt]{\lineheight{1.25}\smash{\begin{tabular}[t]{l}$+5$\end{tabular}}}}%
    \put(0.80802581,0.10480786){\color[rgb]{0,0.50196078,0}\makebox(0,0)[lt]{\lineheight{1.25}\smash{\begin{tabular}[t]{l}$+6$\end{tabular}}}}%
    \put(0.46128578,0.10480786){\color[rgb]{0,0.50196078,0}\makebox(0,0)[lt]{\lineheight{1.25}\smash{\begin{tabular}[t]{l}$-1$\end{tabular}}}}%
    \put(0.41175163,0.10480786){\color[rgb]{0,0.50196078,0}\makebox(0,0)[lt]{\lineheight{1.25}\smash{\begin{tabular}[t]{l}$-2$\end{tabular}}}}%
    \put(0.36221744,0.10480786){\color[rgb]{0,0.50196078,0}\makebox(0,0)[lt]{\lineheight{1.25}\smash{\begin{tabular}[t]{l}$-3$\end{tabular}}}}%
    \put(0.31268319,0.10480786){\color[rgb]{0,0.50196078,0}\makebox(0,0)[lt]{\lineheight{1.25}\smash{\begin{tabular}[t]{l}$-4$\end{tabular}}}}%
    \put(0.26314894,0.10480786){\color[rgb]{0,0.50196078,0}\makebox(0,0)[lt]{\lineheight{1.25}\smash{\begin{tabular}[t]{l}$-5$\end{tabular}}}}%
    \put(0.2136147,0.10480786){\color[rgb]{0,0.50196078,0}\makebox(0,0)[lt]{\lineheight{1.25}\smash{\begin{tabular}[t]{l}$-6$\end{tabular}}}}%
    \put(0.16408045,0.10480786){\color[rgb]{0,0.50196078,0}\makebox(0,0)[lt]{\lineheight{1.25}\smash{\begin{tabular}[t]{l}$-7$\end{tabular}}}}%
    \put(0.1145462,0.10480786){\color[rgb]{0,0.50196078,0}\makebox(0,0)[lt]{\lineheight{1.25}\smash{\begin{tabular}[t]{l}$-8$\end{tabular}}}}%
    \put(0.06501195,0.10480786){\color[rgb]{0,0.50196078,0}\makebox(0,0)[lt]{\lineheight{1.25}\smash{\begin{tabular}[t]{l}$-9$\end{tabular}}}}%
    \put(0.00499295,0.10480786){\color[rgb]{0,0.50196078,0}\makebox(0,0)[lt]{\lineheight{1.25}\smash{\begin{tabular}[t]{l}$-10$\end{tabular}}}}%
    \put(0.18587156,0.46484779){\color[rgb]{0,0.50196078,0}\makebox(0,0)[lt]{\lineheight{1.25}\smash{\begin{tabular}[t]{l}$x$-split line\end{tabular}}}}%
    \put(0,0){\includegraphics[width=\unitlength,page=3]{SplitTemplateWithLabelsAnnuliKnotRing_AfterDF.pdf}}%
    \put(0.57385037,0.04761356){\color[rgb]{0.47058824,0.40392157,0.12941176}\transparent{0.69999999}\makebox(0,0)[lt]{\lineheight{1.25}\smash{\begin{tabular}[t]{l}(\fontsize{8pt}{1em}$A^y_2$)\end{tabular}}}}%
    \put(0.66821199,0.04761356){\color[rgb]{0.47058824,0.40392157,0.12941176}\transparent{0.69999999}\makebox(0,0)[lt]{\lineheight{1.25}\smash{\begin{tabular}[t]{l}\fontsize{8pt}{1em}$A^y_3$\end{tabular}}}}%
    \put(0.76257346,0.04761356){\color[rgb]{0.47058824,0.40392157,0.12941176}\transparent{0.69999999}\makebox(0,0)[lt]{\lineheight{1.25}\smash{\begin{tabular}[t]{l}\fontsize{8pt}{1em}$A^y_6$\end{tabular}}}}%
    \put(0,0){\includegraphics[width=\unitlength,page=4]{SplitTemplateWithLabelsAnnuliKnotRing_AfterDF.pdf}}%
    \put(0.23636702,0.00515592){\color[rgb]{0,0,0}\makebox(0,0)[lt]{\lineheight{1.25}\smash{\begin{tabular}[t]{l}$w=$\end{tabular}}}}%
    \put(0.33543814,0.00515592){\color[rgb]{0.33333333,0,0}\makebox(0,0)[lt]{\lineheight{1.25}\smash{\begin{tabular}[t]{l}$y^{2}$\end{tabular}}}}%
    \put(0.36846092,0.00515592){\color[rgb]{0,0.33333333,0.83137255}\makebox(0,0)[lt]{\lineheight{1.25}\smash{\begin{tabular}[t]{l}$x^{5}$\end{tabular}}}}%
    \put(0.40148375,0.00515592){\color[rgb]{0.66666667,0,0}\makebox(0,0)[lt]{\lineheight{1.25}\smash{\begin{tabular}[t]{l}$y^{2}$\end{tabular}}}}%
    \put(0.43450668,0.00515592){\color[rgb]{0.16470588,0.49803922,1}\makebox(0,0)[lt]{\lineheight{1.25}\smash{\begin{tabular}[t]{l}$x^{7}$\end{tabular}}}}%
    \put(0.46752951,0.00515592){\color[rgb]{1,0,0}\makebox(0,0)[lt]{\lineheight{1.25}\smash{\begin{tabular}[t]{l}$y^{6}$\end{tabular}}}}%
    \put(0.50055234,0.00515592){\color[rgb]{0.33333333,0.6,1}\makebox(0,0)[lt]{\lineheight{1.25}\smash{\begin{tabular}[t]{l}$x^{2}$\end{tabular}}}}%
    \put(0.53357517,0.00515592){\color[rgb]{1,0.50196078,0.50196078}\makebox(0,0)[lt]{\lineheight{1.25}\smash{\begin{tabular}[t]{l}$y^{2}$\end{tabular}}}}%
    \put(0.56659801,0.00515592){\color[rgb]{0.83529412,0.89803922,1}\makebox(0,0)[lt]{\lineheight{1.25}\smash{\begin{tabular}[t]{l}$x^{5}$\end{tabular}}}}%
    \put(0.59962079,0.00515592){\color[rgb]{1,0.83529412,0.83529412}\makebox(0,0)[lt]{\lineheight{1.25}\smash{\begin{tabular}[t]{l}$y^{3}$\end{tabular}}}}%
    \put(0.23636702,0.00515592){\color[rgb]{0,0,0}\makebox(0,0)[lt]{\lineheight{1.25}\smash{\begin{tabular}[t]{l}$w=$\end{tabular}}}}%
    \put(0.28920617,0.00515854){\color[rgb]{0,0.2,0.50196078}\makebox(0,0)[lt]{\lineheight{1.25}\smash{\begin{tabular}[t]{l}$x_1^{10}$\end{tabular}}}}%
    \put(0.33543814,0.00515592){\color[rgb]{0.33333333,0,0}\makebox(0,0)[lt]{\lineheight{1.25}\smash{\begin{tabular}[t]{l}$y_1^{2}$\end{tabular}}}}%
    \put(0.40148375,0.00515854){\color[rgb]{0.66666667,0,0}\makebox(0,0)[lt]{\lineheight{1.25}\smash{\begin{tabular}[t]{l}$y_2^{2}$\end{tabular}}}}%
    \put(0.43450668,0.00515854){\color[rgb]{0.16470588,0.49803922,1}\makebox(0,0)[lt]{\lineheight{1.25}\smash{\begin{tabular}[t]{l}$x_3^{7}$\end{tabular}}}}%
    \put(0.46752951,0.00515854){\color[rgb]{1,0,0}\makebox(0,0)[lt]{\lineheight{1.25}\smash{\begin{tabular}[t]{l}$y_3^{6}$\end{tabular}}}}%
    \put(0.50055234,0.00515854){\color[rgb]{0.33333333,0.6,1}\makebox(0,0)[lt]{\lineheight{1.25}\smash{\begin{tabular}[t]{l}$x_4^{2}$\end{tabular}}}}%
    \put(0.53357517,0.00515854){\color[rgb]{1,0.50196078,0.50196078}\makebox(0,0)[lt]{\lineheight{1.25}\smash{\begin{tabular}[t]{l}$y_4^{2}$\end{tabular}}}}%
    \put(0.56659801,0.00515854){\color[rgb]{0.83529412,0.89803922,1}\makebox(0,0)[lt]{\lineheight{1.25}\smash{\begin{tabular}[t]{l}$x_5^{5}$\end{tabular}}}}%
    \put(0.59962079,0.00515592){\color[rgb]{1,0.83529412,0.83529412}\makebox(0,0)[lt]{\lineheight{1.25}\smash{\begin{tabular}[t]{l}$y_5^{3}$\end{tabular}}}}%
    \put(0.36846092,0.00515854){\color[rgb]{0,0.33333333,0.83137255}\makebox(0,0)[lt]{\lineheight{1.25}\smash{\begin{tabular}[t]{l}$x_2^{5}$\end{tabular}}}}%
    \put(0,0){\includegraphics[width=\unitlength,page=5]{SplitTemplateWithLabelsAnnuliKnotRing_AfterDF.pdf}}%
    \put(0.43405358,0.04761356){\color[rgb]{0.47058824,0.40392157,0.12941176}\transparent{0.69999999}\makebox(0,0)[lt]{\lineheight{1.25}\smash{\begin{tabular}[t]{l}\fontsize{8pt}{1em}$A^x_2$\end{tabular}}}}%
    \put(0.06009723,0.04761356){\color[rgb]{0.47058824,0.40392157,0.12941176}\transparent{0.69999999}\makebox(0,0)[lt]{\lineheight{1.25}\smash{\begin{tabular}[t]{l}(\fontsize{8pt}{1em}$A^x_{10}$)\end{tabular}}}}%
    \put(0.19290418,0.04761356){\color[rgb]{0.47058824,0.40392157,0.12941176}\transparent{0.69999999}\makebox(0,0)[lt]{\lineheight{1.25}\smash{\begin{tabular}[t]{l}\fontsize{8pt}{1em}$A^x_7$\end{tabular}}}}%
    \put(0.31872128,0.04761356){\color[rgb]{0.47058824,0.40392157,0.12941176}\transparent{0.69999999}\makebox(0,0)[lt]{\lineheight{1.25}\smash{\begin{tabular}[t]{l}\fontsize{8pt}{1em}$A^x_5$\end{tabular}}}}%
    \put(0,0){\includegraphics[width=\unitlength,page=6]{SplitTemplateWithLabelsAnnuliKnotRing_AfterDF.pdf}}%
  \end{picture}%
\endgroup%

%% file: figures/LinkCompleInSFSpaceResizedProjection.pdf_tex
\begingroup%
  \makeatletter%
  \providecommand\color[2][]{%
    \errmessage{(Inkscape) Color is used for the text in Inkscape, but the package 'color.sty' is not loaded}%
    \renewcommand\color[2][]{}%
  }%
  \providecommand\transparent[1]{%
    \errmessage{(Inkscape) Transparency is used (non-zero) for the text in Inkscape, but the package 'transparent.sty' is not loaded}%
    \renewcommand\transparent[1]{}%
  }%
  \providecommand\rotatebox[2]{#2}%
  \newcommand*\fsize{\dimexpr\f@size pt\relax}%
  \newcommand*\lineheight[1]{\fontsize{\fsize}{#1\fsize}\selectfont}%
  \ifx\svgwidth\undefined%
    \setlength{\unitlength}{273.24527717bp}%
    \ifx\svgscale\undefined%
      \relax%
    \else%
      \setlength{\unitlength}{\unitlength * \real{\svgscale}}%
    \fi%
  \else%
    \setlength{\unitlength}{\svgwidth}%
  \fi%
  \global\let\svgwidth\undefined%
  \global\let\svgscale\undefined%
  \makeatother%
  \begin{picture}(1,0.25660948)%
    \lineheight{1}%
    \setlength\tabcolsep{0pt}%
    \put(0,0){\includegraphics[width=\unitlength,page=1]{LinkCompleInSFSpaceResizedProjection.pdf}}%
    \put(0.67439754,0.11894461){\color[rgb]{1,0,1}\makebox(0,0)[lt]{\lineheight{1.25}\smash{\begin{tabular}[t]{l}\fontsize{5pt}{1em}$V_2$\end{tabular}}}}%
    \put(0.72754125,0.13459512){\color[rgb]{1,0,1}\makebox(0,0)[lt]{\lineheight{1.25}\smash{\begin{tabular}[t]{l}\fontsize{5pt}{1em}$V_3$\end{tabular}}}}%
    \put(0.83057496,0.14165178){\color[rgb]{1,0,1}\makebox(0,0)[lt]{\lineheight{1.25}\smash{\begin{tabular}[t]{l}\fontsize{5pt}{1em}$V_6$\end{tabular}}}}%
    \put(0,0){\includegraphics[width=\unitlength,page=2]{LinkCompleInSFSpaceResizedProjection.pdf}}%
    \put(0.94832263,0.08425077){\color[rgb]{0.50588235,0.50588235,0.50588235}\makebox(0,0)[lt]{\lineheight{1.25}\smash{\begin{tabular}[t]{l}$U$\end{tabular}}}}%
  \end{picture}%
\endgroup%